\documentclass[10pt,a4paper]{amsart}
\usepackage[utf8]{inputenc}
\usepackage{amsmath}
\usepackage{amsfonts}
\usepackage{amssymb}
\usepackage[left=2cm,right=2cm,top=2cm,bottom=2cm]{geometry}

\usepackage{bbm}
\usepackage{mathrsfs}
\usepackage[all]{xy}
\usepackage{bm}
\usepackage{manfnt}

\usepackage{tikz}
\usepackage[backref=page]{hyperref}
\author{Geoffrey Powell}
\title[Relative nonhomogeneous Koszul duality for PROPs]{Relative nonhomogeneous Koszul duality for PROPs associated to nonaugmented operads}
\address{Univ Angers, CNRS, LAREMA, SFR MATHSTIC, F-49000 Angers, France}
\email{Geoffrey.Powell@math.cnrs.fr}
\urladdr{https://math.univ-angers.fr/~powell/}

\keywords{Koszul duality; relative nonhomogeneous Koszul duality; operad; PROP}
\subjclass[2000]{16S37, 16E45}
\newtheorem{THM}{Theorem}
\newtheorem{PROP}[THM]{Proposition}

\newtheorem{thm}{Theorem}[section]
\newtheorem{prop}[thm]{Proposition}
\newtheorem{cor}[thm]{Corollary}
\newtheorem{lem}[thm]{Lemma}
\theoremstyle{definition}
\newtheorem{defn}[thm]{Definition}
\newtheorem{exam}[thm]{Example}

\theoremstyle{remark}
\newtheorem{rem}[thm]{Remark}
\newtheorem*{rem*}{Remark}
\newtheorem{nota}[thm]{Notation}
\newtheorem{hyp}[thm]{Hypothesis}

\renewcommand{\phi}{\varphi}
\renewcommand{\hom}{\mathrm{Hom}}
\newcommand{\sym}{\mathfrak{S}}
\newcommand{\fs}{{\mathbf{FS}}}
\newcommand{\kmod}{\kk\dash\modules}
\newcommand{\calc}{\mathcal{C}}
\newcommand{\cald}{\mathcal{D}}
\newcommand{\nat}{\mathbb{N}}
\newcommand{\zed}{\mathbb{Z}}

\newcommand{\op}{^\mathrm{op}}

\newcommand{\kring}{\mathbb{K}}
\newcommand{\fb}{\mathbf{FB}}
\newcommand{\finj}{{\mathbf{FI}}}
\newcommand{\id}{\mathrm{Id}}

\newcommand{\triv}{\mathsf{triv}}
\newcommand{\sgn}{\mathsf{sgn}}

\newcommand{\fin}{\mathbf{FA}}

\newcommand{\n}{\mathbf{n}}

\newcommand{\rmult}{\diamond}
\newcommand{\qhat}{\hat{q}}
\newcommand{\grad}{\mathsf{gr}^F}
\newcommand{\rgrad}{\mathsf{gr}^G}

\newcommand{\dgmod}{\mathbf{dgMod}}
\newcommand{\lbdgmod}{\dgmod^{\mathsf{lbd}}}
\newcommand{\kk}{\mathbbm{k}}
\newcommand{\ofb}{\odot_{\fb}}
\newcommand{\obf}{\odot_{\fb\op}}
\newcommand{\dash}{\hspace{-.25em}-\hspace{-.25em}}
\newcommand{\modules}{\mathsf{mod}}
\newcommand{\bmu}{\mathbbm{1}}

\newcommand{\opd}{\mathscr{O}}

\newcommand{\ppd}{\mathscr{P}}
\newcommand{\uppd}{\ppd^{\mathrm{u}}}
\newcommand{\cat}{\mathsf{Cat}}
\newcommand{\U}{\mathscr{U}}
\newcommand{\com}{\mathfrak{Com}}
\newcommand{\lie}{\mathfrak{Lie}}
\newcommand{\ucom}{\com^{\mathrm{u}}}

\newcommand{\ihom}{\underline{\hom}}
\newcommand{\e}{\mathfrak{e}}
\newcommand{\vk}{{}^\vee \! K}

\newcommand{\fsbase}{\fs^{\mathsf{ord}}}
\newcommand{\wt}[1]{^{(#1)}}

\newcommand{\ls}{^\mathfrak{L}}
\newcommand{\rs}{^\mathfrak{R}}
\newcommand{\dg}[1]{^{[#1]}}

\newcommand{\qd}{^\perp}
\newcommand{\desusp}{^{\mathfrak{L}\ddag}}
\newcommand{\susp}{^{\mathfrak{R}\ddag}}

\newcommand{\cn}{\ : \ }
\newcommand{\holbd}{\mathscr{H}^{\mathsf{lbd}}}

\newcommand{\g}{\mathfrak{g}}
\numberwithin{equation}{section}
\begin{document}

\begin{abstract}
The purpose of this paper is to show how Positselski's relative nonhomogeneous Koszul duality theory applies when studying the linear category underlying the PROP associated to a (non-augmented) operad of a certain form, in particular assuming that the reduced part of the operad is binary quadratic. In this case, the linear category has both a left augmentation and a right augmentation (corresponding to different units), using  Positselski's terminology.  

The general theory provides two associated linear differential graded (DG) categories; indeed, in this framework, one can work entirely within the DG  realm, as opposed to the curved setting required for Positselski's general theory. Moreover,  DG modules over these DG categories are related by adjunctions. 

When the reduced part of the operad is Koszul (working over a field of characteristic zero), the relative Koszul duality theory shows that there is a Koszul-type equivalence between the appropriate homotopy categories of DG modules. This gives a form of Koszul duality relationship between the above DG categories. 

This is illustrated by the case of the operad encoding unital, commutative associative algebras, extending the classical Koszul duality between commutative associative algebras and Lie algebras. In this case, the associated linear category is the linearization of the category of finite sets and all maps. The relative nonhomogeneous Koszul duality theory relates its derived category to the respective homotopy categories of modules over two explicit linear DG categories.
\end{abstract}

\maketitle

\section{Introduction}
\label{sect:intro}

The purpose of this paper is to exploit Positselski's relative nonhomogeneous Koszul duality theory \cite{MR4398644} in a  special case arising from operad theory. Whereas Positselski's general theory requires that one works in the curved differential setting, the framework considered here has the advantage that one can stay entirely within the differential graded (DG) realm. 

We consider an operad $\uppd$ over a field $\kk$ (of characteristic zero for this Introduction) with the property that $\uppd (0) = \uppd (1) = \kk$. For example, one can take $\uppd$ to be  $\ucom$, which encodes unital, commutative associative algebras.
 Define $\ppd$ to be the reduced suboperad of $\uppd$, namely the suboperad supported on positive arities. Writing $\U$ for the unique operad such that $\U(0) = \U (1) = \kk$ and $\U(t)=0$ for $t>1$, $\U$ is a sub operad of $\uppd$ and this fits into a commutative diagram of inclusions of operads
\[
\xymatrix{
I
\ar[r]
\ar[d]
&
\U 
\ar[d]
\\
\ppd 
\ar[r]
&
\uppd,
}
\] 
in which $I$ is the unit operad. At the level of the underlying symmetric sequences, one has $\uppd \cong \U \circ \ppd$, where $\circ$ is the operadic composition product.

One can pass to the associated PROPs, using the functor $\cat (-)$, retaining just the underlying $\kk$-linear category of the PROP. For example, $\cat I$ is equivalent to $\kk \fb$, the $\kk$-linearization of the category $\fb$ of finite sets and bijections, and $\cat \U$ is equivalent to $\kk\finj$, where $\finj$ is the category of finite sets and injections. This yields the commutative diagram of $\kk$-linear functors that are the identity on objects: 
\[
\xymatrix{
\kk \fb 
\ar[r]
\ar[d]
&
\kk \finj 
\ar[d]
\\
\cat \ppd 
\ar[r] 
&
\cat \uppd.
}
\] 

Now, the composition of $\cat \uppd$ induces an isomorphism of $\kk \finj \boxtimes \cat \ppd \op$-modules (where $\boxtimes$ is the exterior tensor product):
\[
\kk \finj \otimes_{\kk \fb} \cat \ppd \stackrel{\cong}{\rightarrow} 
\cat \uppd. 
\]
For example, in the case $\uppd = \ucom$,  $\ppd$ is the operad $\com$ that encodes commutative associative algebras; there are equivalences of categories 
$
\cat \com  \cong  \kk \fs$ and 
$\cat \ucom  \cong  \kk \fin$,
where $\fs$ is the category of finite sets and surjections, and $\fin$ that of finite sets  and all maps. The above commutative diagram identifies in this case with 
\[
\xymatrix{
\kk \fb 
\ar[r]
\ar[d]
&
\kk \finj 
\ar[d]
\\
\kk \fs \ar[r]
&
\kk \fin,
}
\]
induced by the inclusions of the wide subcategories $\fb$, $\finj$, $\fs$ of $\fin$. The isomorphism $
\kk \finj \otimes_{\kk \fb} \kk \fs 
\stackrel{\cong}{\rightarrow}
\kk \fin
$ 
simply reflects the fact that any map can be factorized uniquely as a surjection followed by the inclusion of its image as a subset.

The starting point for this work is the fact that $\kk \finj$ is {\em Koszul} over $\kk\fb$. Indeed, this Koszul property extends the BGG correspondence relating modules over symmetric algebras and modules over exterior algebras.  (A word of warning for those familiar with Koszul duality of operads: $\U$ is not {\em binary} quadratic, although it is {\em quadratic}.) Moreover, the canonical augmentation $\kk \finj \rightarrow \kk \fb$ induces a right augmentation $\cat \uppd \rightarrow \cat \ppd$, using the terminology of Positselski. 

The grading of $\kk \finj$ induces a filtration $G_\bullet \cat \uppd$ which exhibits $\cat \uppd$ as a nonhomogeneous quadratic $\kk$-linear category over $\cat \ppd$. In particular, one has the associated graded $\rgrad \cat \uppd$ that is homogeneous quadratic. One can then form the (right) quadratic dual, $(\rgrad \cat \uppd) \qd$; this has underlying bimodule
\[
\cat \ppd \otimes_{\kk \fb} \kk \finj\qd,
\]
where $\kk \finj \qd$ denotes the quadratic dual of $\kk \finj$ (considered as augmented over $\kk \fb$). 
This lets us apply Positselski's relative nonhomogeneous Koszul duality (working over $\cat \ppd$), staying within the DG realm. 

This leads to an adjunction 
\begin{eqnarray}
\label{eqn:adj_BGG_intro}
(\rgrad \cat \uppd)\qd \dash  \dgmod 
\rightleftarrows 
 \cat \uppd \dash  \dgmod 
\end{eqnarray}
between the categories of DG modules over the respective DG $\kk$-linear categories, induced by an explicit dualizing object $\vk (\cat \uppd)$. This restricts to an adjunction between the respective categories of `locally bounded' DG objects and 
 one has the following relative form of the BGG correspondence. 

\begin{THM}
(Corollary \ref{cor:relative_BGG}.)
The adjunction (\ref{eqn:adj_BGG_intro}) induces an equivalence of categories,
 \[
\holbd ((\rgrad \cat \uppd)\qd ) 
\stackrel{\simeq}{\rightleftarrows}
\holbd (\cat \uppd ),
\]
where $\holbd(-)$ denotes the homotopy category of locally bounded objects.
\end{THM}

In the case $\uppd = \U$, so that $\ppd = I$, this recovers a form of the BGG correspondence. The latter can be compared with the Fourier transform given by Sam and Snowden \cite{MR3430359}, which exploits the fact that $\kk \finj$ is `Koszul self-dual' together with vector space duality, to give an anti self-equivalence of the bounded derived category of $\kk \finj$-modules (imposing further finiteness hypotheses).

To develop the other side of the picture, we suppose that $\ppd$ is a finitely-generated binary quadratic operad. Then the above framework can be reflected: the canonical augmentation $\cat \ppd \rightarrow \kk \fb$ induces a left augmentation $\cat \uppd \rightarrow \kk \finj$ and we can again apply Positselski's relative nonhomogeneous Koszul duality (working over $\kk \finj$). 

The weight grading of $\cat \ppd$ induces a filtration $F_\bullet \cat \uppd$, exhibiting $\cat \uppd$ as a filtered $\kk$-linear category;  $\cat \uppd$ is nonhomogeneous quadratic over $\kk \finj$  with respect to this filtration and one has the associated graded $\grad \cat \uppd$ that is homogeneous quadratic. One can then form the (left) quadratic dual  $(\grad \cat \uppd)\qd$ with respect to $\kk \finj$, and this has underlying bimodule 
\[
(\cat \ppd)\qd \otimes_{\kk \fb} \kk \finj.
\]
This is equipped with a DG  structure, as above. 

In this case, there is an adjunction:
\begin{eqnarray}
\label{eqn:adj_ppd_intro}
(\cat \uppd)\dash\dgmod 
\rightleftarrows 
(\grad \cat \uppd)\qd \dash \dgmod 
\end{eqnarray}
induced by an explicit dualizing complex $K^\vee (\cat \uppd)$. This restricts to an adjunction between the respective categories of `locally bounded' DG objects. When the operad $\ppd$ is Koszul, one obtains an equivalence of the associated homotopy categories:

\begin{THM}
(Theorem \ref{thm:relative_lbd_equivalence_cat_ppd}.)
If $\kk$ is a field of characteristic zero and  $\ppd$ is Koszul, the adjunction (\ref{eqn:adj_ppd_intro}) induces an equivalence between the respective homotopy categories of locally bounded objects 
\[
\holbd(\cat \uppd)
\stackrel{\simeq}{\rightleftarrows}
\holbd((\grad \cat \uppd)\qd).
\]
\end{THM}

The category $\holbd (\cat \uppd)$ can be considered as  the `locally bounded derived category' of $\cat \uppd$-modules. The above two Theorems thus give two different models for this via relative nonhomogeneous Koszul duality.
Moreover, putting the above adjunctions together, one has an explicit adjunction 
\begin{eqnarray}
\label{eqn:adj_intro}
(\rgrad \cat \uppd)\qd \dash \dgmod 
\rightleftarrows 
(\grad \cat \uppd)\qd \dash\dgmod.
\end{eqnarray}
between the respective categories of DG modules and  the following:

\begin{THM}
\label{THM:bikoszul_equivalence}
(Theorem \ref{thm:composite_holbd_equivalence}.)
Suppose that $\kk$ is a field of characteristic zero and that $\ppd$ is Koszul. Then the adjunction (\ref{eqn:adj_intro}) induces an equivalence between the respective homotopy categories of locally bounded objects
\[
\holbd ((\rgrad \cat \uppd)\qd)
\stackrel{\simeq}{\rightleftarrows}
\holbd ((\grad \cat \uppd)\qd).
\]
\end{THM}

Heuristically, one can think of this result as giving a Koszul duality relationship between $(\rgrad \cat \uppd)\qd$ and $(\grad \cat \uppd)\qd$. For example, $\kk \fb$ can be considered as a DG $(\rgrad \cat \uppd)\qd$-module and the unit of the above adjunction identifies as 
\[
\kk \fb \rightarrow \ihom _{\kk \fb} ((\grad \cat \uppd)\qd, (\grad \cat \uppd)\qd),
\]  
where the codomain is equipped with an explicit twisted differential. Theorem \ref{THM:bikoszul_equivalence} includes the fact that this is a weak equivalence, thus giving an explicit homological relationship between $(\grad \cat \uppd)\qd$ and $(\grad \cat \uppd)\qd$. This generalizes one of the classical Koszul complexes that arises in classical Koszul duality for binary quadratic operads, going back to Ginzburg and Kapranov \cite{MR1301191}.

To be more in keeping with classical operadic Koszul duality for operads, we replace $(\grad \cat \uppd)\qd$ by its `desuspension', here denoted $\big((\grad \cat \uppd)\qd\big)\desusp$ (the sheering functor $(-)\desusp$ is reviewed in Appendix \ref{sect:sheer}). The underlying bimodule can be expressed in terms of $\cat \ppd^!$, where $\ppd^!$ is the Koszul dual of $\ppd$:
 \[
 \big((\grad \cat \uppd)\qd\big)\desusp \cong (\cat \ppd^!) \op \otimes_\fb \kk \finj\desusp,
 \]
where $\kk \finj\desusp$ is a signed counterpart of $\kk \finj$.

One can pass to (co)homology. In particular, restricting to degree zero (with respect to the appropriate grading),  this yields the interesting $\kk$-linear categories
\begin{eqnarray*}
&&
H^0 ((\rgrad \cat \uppd)\qd) \\
&&
H_0 (((\grad \cat \uppd)\qd)\desusp)
\end{eqnarray*}
that are respectively a subcategory of $\cat \ppd$ and a quotient category of $(\cat \ppd^!)\op$.

This theory applies to the case $\uppd = \ucom$, since $\com$ is Koszul, with $\com^!$ the Lie operad $\lie$. This case is analysed in Section \ref{sect:ucom}, where the DG $\kk$-linear categories $(\rgrad \kk \fin)\qd$ and $(\grad \kk\fin)\qd$ are identified (recall that $\cat \ucom\cong \kk \fin$). 

The DG category $(\rgrad \kk \fin)\qd$ is the most familiar; it appeared implicitly in \cite{MR4518761}:

\begin{PROP}
(Proposition \ref{prop:rgrad_kk_fin_qd}.)
The DG $\kk$-linear category $(\rgrad \kk \fin)\qd$ has underlying $\kk \fb$-bimodule 
$
\kk \fs \otimes_\fb \kk \finj\qd, 
$ 
with cohomological grading inherited from $\kk \finj\qd$. 

The differential is the Koszul complex differential associated to the right $\kk \finj$-module structure of $\kk \fs$ induced by the right augmentation $\kk \fin \rightarrow \kk \fs$.
\end{PROP}

The DG category $(\grad \kk\fin)\qd$ as constructed here is, at first sight, less familiar. However, it is shown to be  related to the Chevalley-Eilenberg complex with coefficients in $\cat \lie$; more precisely:

\begin{THM}
(Theorem \ref{thm:identify_with_CE}.)
The opposite of the DG $\kk$-linear category $\big((\grad \kk \fin)\qd\big)\desusp$ is isomorphic as a complex of right $\cat \lie$-modules to the Chevalley-Eilenberg complex of $\lie$ with coefficients in $\cat \lie$, which has underlying object 
$
(\kk \finj ^\ddag) \op
\otimes_\fb 
\cat \lie.
$ 
\end{THM}

On the $\com$ side, the $\kk$-linear subcategory of $\kk \fs \cong \cat \com$ given by $H^0((\rgrad \kk \fin)\qd)$ corresponds to the category studied  in \cite{MR4518761} (implicitly) and in \cite{P_filterFS}. On the $\lie$ side, $H_0 (((\grad \kk \fin)\qd) \desusp)$ is the quotient category of $(\cat \lie)\op$ that is the main player in \cite{MR4696223}.
 These results not only explain how these structures fit into the relative nonhomogeneous Koszul duality framework, but also {\em  why} the Chevalley-Eilenberg complex with coefficients in $\cat \lie$ arises in this theory.

\bigskip
{\bf Organization of the paper.} Most of the first part of the  paper is devoted to explaining Positselski's theory in the presence of either a left of right augmentation,  specializing to the situation in which there is a commutative diagram of unital, associative ring maps
\[
\xymatrix{
\kring 
\ar[r]
\ar[d]
&
C 
\ar[d]
\\
R 
\ar[r]
&
A}
\]
with the property that the multiplication of $A$ induces an isomorphism of $R \otimes C\op$-modules $R \otimes_\kring C \stackrel{\cong}{\rightarrow } A$. 

This general theory is developed in Part \ref{part:one}; in particular, this explains why the structures considered are DG structures (namely, that the curvature that appears in Positselski's work is zero). It culminates with Section \ref{sect:koszul}, which gives an outline of Koszulity for nonhomogeneous quadratic rings, a foretaste for the applications. 

Part \ref{part:two} is devoted to the applications to the operadic framework outlined above. One aim is to make the constructions explicit, also recalling the (absolute) Koszul duality results for $\kk \finj$ and for $\cat \ppd$ in a way that is suitable for generalization to the relative case. 
\tableofcontents

\part{The general theory}
\label{part:one}

\section{Left (or right) augmented rings}
\label{sect:leftaug}

The aim of this section is to review the basic theory of left augmented rings. This will be used when applying Positselski's relative Koszul duality to ensure that one remains within the DG setting, thus considerably simplifying the theory. The corresponding notion of a right augmented ring is outlined briefly in Section \ref{subsect:right_aug}.

\subsection{Basics}

We review the theory of left augmented rings over a fixed unital, associative ring $R$, following \cite[Section 3.8]{MR4398644}. Unadorned tensor products are understood to be $\otimes_\zed$. 

\begin{rem}
\label{rem:bimodules}
Recall that, for unital, associative rings $R_1$, $R_2$, a left $R_1$, right $R_2$ bimodule is equivalent to a (left) $R_1 \otimes R_2\op$-module. (Here, by convention, all modules are {\em left} modules unless otherwise specified; a right $R$-module is the same thing as a $R\op$-module.)
\end{rem}

\begin{defn}
\label{defn:left_aug}
A left augmented ring over $R$ is an associative unital ring $A$ equipped with a {\em unit} $\eta : R \hookrightarrow A$ (a morphism of unital rings) and  a {\em left augmentation} $\epsilon : A \rightarrow R$ such that 
\begin{enumerate}
\item 
$\epsilon$ is a morphism of left $R$-modules; 
\item 
$\epsilon \circ \eta = \id_R$; 
\item 
$\ker \epsilon$ is a left ideal of $A$, in particular it has the structure of a (non-unital) associative algebra.
\end{enumerate}
\end{defn}

It follows that $R$ has a unique left $A$-module structure such that $\epsilon : A \rightarrow R$ is a morphism of left $A$-modules. 

There is a direct sum decomposition of left $R$-modules 
$
A \cong R \oplus \ker \epsilon
$ 
induced by $\eta$ and $\epsilon$. With respect to this splitting, the restriction of the multiplication of $A$ gives 
\[
\ker \epsilon \otimes R \rightarrow A \cong R \oplus \ker \epsilon.
\]
This yields the two component  structure maps: 
\begin{eqnarray}
\rmult &:& \ker \epsilon \otimes R \rightarrow \ker \epsilon
\\
q &:& \ker \epsilon \otimes R \rightarrow R.
\end{eqnarray}

\begin{rem}
\label{rem:products}
We reserve the notation $xy$  for the product of the elements $x$ and $y$ in $A$. 
 For example, for $x \in \ker \epsilon$ and $r \in R$, $xr$  denotes the product formed in $A$, which need not belong to $\ker \epsilon$. Using the isomorphism $A \cong R \oplus \ker \epsilon$ one has the fundamental relation:
 \begin{eqnarray}
 \label{eqn:prod_relation}
 xr = q(x \otimes r) + x \rmult r.  
 \end{eqnarray}
(This differs from the notation used by Positselski in \cite[Chapter 3]{MR4398644} for example, where the product is denoted by $\ast$ and the right multiplication (denoted here by $\diamond$) without symbol.)
\end{rem}

Clearly one has:

\begin{lem}
\label{lem:ker_epsilon_bimodule}
\ 
\begin{enumerate}
\item
The map $\rmult$ together with the restriction of the left $A$-module structure makes $\ker \epsilon$ into an $A \otimes R\op$-module, hence an $R$-bimodule, by restriction. 
\item 
The product $\ker \epsilon \otimes \ker \epsilon \rightarrow \ker \epsilon$ is a morphism of left $A$-modules, hence of left $R$-modules, by restriction.
\end{enumerate}
\end{lem}

\begin{rem}
By definition, $\ker \epsilon$ is a sub left $R$-module of $A$; it is not in general a sub $R$-bimodule of $A$.
\end{rem}

As observed above, the product of $A$ restricts to an associative product on $\ker \epsilon$. The associativity of the product on $A \cong \ker \epsilon \oplus R$ is  encoded in the following properties of $q$ (using the $R$-bimodule structure of $\ker \epsilon$).

\begin{lem}
\label{lem:properties_q}
The morphism $q : \ker \epsilon \otimes R \rightarrow R$ satisfies the following properties, for all $x,y \in \ker \epsilon$ and $r, s \in R$:
\begin{enumerate}
\item 
\label{item:a}
$q (rx \otimes s) = r q (x \otimes s)$, i.e., $q$ is a morphism of left $R$-modules; 
\item 
\label{item:b}
$q (x \otimes rs) = q (x \rmult r \otimes s) + q (x \otimes r) s$;
\item 
\label{item:e}
$x (ry) - (x \rmult r) y = q (x \otimes r )y$;
\item 
\label{item:g}
$(xy)\rmult r = x(y \rmult r) + x \rmult q(y \otimes r)$;
\item 
\label{item:h}
$q (x \otimes q (y \otimes r) ) = q (xy \otimes r)$, encoding the left $\ker \epsilon$-module structure on $R$.
\end{enumerate}
\end{lem}

\begin{proof}
These identities are deduced from the associativity of the product of $A$ as follows:
\begin{enumerate}
\item 
consider the restriction of the product $R \otimes \ker \epsilon \otimes R \rightarrow A$ and then project to $R$ (the other projection to $\ker \epsilon$ encodes the $R$-bimodule structure of $\ker \epsilon$);
\item 
use $\ker \epsilon \otimes R \otimes R \rightarrow A$ and the projection to $R$, together with the fundamental identity (\ref{eqn:prod_relation}); (the other projection yields the associativity of the right $R$-module structure of $\ker \epsilon$);
\item 
consider $\ker \epsilon \otimes R \otimes \ker \epsilon \rightarrow A$, which maps to $\ker \epsilon$, again using (\ref{eqn:prod_relation});
\item 
consider $\ker \epsilon \otimes \ker \epsilon \otimes R \rightarrow A$ composed with the projection to $\ker \epsilon$, again in conjunction with (\ref{eqn:prod_relation});
\item 
compose $\ker \epsilon \otimes \ker \epsilon \otimes R \rightarrow A$ with the projection to $R$. 
\end{enumerate}
\end{proof}

\begin{rem}
\ 
\begin{enumerate}
\item 
The properties given in the Lemma are analogous to the self-consistency equations given in \cite[Proposition 3.14]{MR4398644} restricted to the case when the curvature $h$ (as defined in {\em loc. cit.}) is zero. (Compare the relations (a), (b), (e), (g), (h) of that Proposition, respectively.) The Lemma is more general in the sense that it applies without restriction on $x, y \in \ker \epsilon$. 
\item 
Point (\ref{item:g}) shows that the product $\ker \epsilon \otimes \ker \epsilon \rightarrow \ker \epsilon$ is not in general a morphism of right $R$-modules and (\ref{item:e}) shows that this product does not in general factor through $\ker \epsilon \otimes_R \ker \epsilon$.
\end{enumerate}
\end{rem}

Now, $\hom_R (\ker \epsilon, R) $ has the structure of an $R $-bimodule, since $\ker \epsilon$ does, by Lemma \ref{lem:ker_epsilon_bimodule}.  Then, since $q$ is a morphism of left $R$-modules by Lemma \ref{lem:properties_q}, it has adjoint 
\[
\qhat : R \rightarrow \hom_R (\ker \epsilon, R).
\]

The second property of Lemma \ref{lem:properties_q} is equivalent to the first statement of the following:

\begin{lem}
\label{lem:qhat_derivation}
The morphism $\qhat : R \rightarrow \hom_R (\ker \epsilon, R)$ is a derivation, using the $R$-bimodule structure of $\hom_R (\ker \epsilon, R) $.  In particular, $q (x \otimes 1)=0$, for all $x \in \ker \epsilon$,  where $1$ is the unit of $R$.
\end{lem}

\subsection{Left augmentations and distributive laws}
\label{subsect:distrib}

We now specialize to the case where $A$ is generated by $R$ together with a ring $C$, assuming that these satisfy a distributive law, working over a unital associative ring $\kring$. 
 
 More precisely, we suppose given morphisms of unital associative rings: 
\[
\xymatrix{
\kring 
\ar@{^(->}[r]
\ar@{^(->}[d]
&
C 
\ar[d]
\\
R
\ar[r]_{\eta_A}
&
A.
}
\]
In particular $A$ is a $\kring $-bimodule and the product of $A$ is given by a morphism of $\kring$-bimodules $
A \otimes_\kring A 
\rightarrow 
A $.

 In addition, we suppose that the following holds:

\begin{hyp}
\label{hyp:A}
The product of $A$ induces an isomorphism
 $R \otimes_\kring C 
\stackrel{\cong}{\rightarrow}
A$ of $R \otimes C\op$-modules.
\end{hyp}

\begin{rem}
\label{rem:hyp_A}
Under this hypothesis, the product of $A$ induces an  exchange map (or distributive law) $\lambda_C \cn C \otimes_\kring R \rightarrow R \otimes_\kring C$, which is a morphism of $\kring$-bimodules. This, together with the ring structures of $R$ and $C$, determines the ring structure of $A$. 
\end{rem}

\begin{lem}
\label{lem:R_C}
Suppose that $C$ is augmented, with augmentation $\epsilon_C : C \rightarrow \kring$,  and write $C_+$ for the augmentation ideal. 
 Then $A$ is left augmented with respect to $\epsilon_A := \id_R \otimes_\kring \epsilon_C : R \otimes_\kring C \rightarrow R$ and  $\ker \epsilon_A \cong R \otimes_\kring C_+$ as a left $R$-module. 
 Moreover, $\epsilon_A$ is a morphism of $R \otimes \kring\op$-modules.
\end{lem}

\begin{proof}
It is clear that $\epsilon_A$ is a morphism of left $R$-modules which is a retract of $\eta_A : R \rightarrow A$. The identification of  $\ker \epsilon_A $ as a left $R$-module is immediate.  

It remains to check that $\ker \epsilon_A$ is a left $A$-ideal. This is a straightforward consequence of associativity of the product of $A$. Indeed, it clearly suffices to show that product restricted to $C_+  \otimes \ker \epsilon_A$ maps to $R \otimes_\kring C_+$. Now $C_+ \otimes  \ker \epsilon_A \cong C_+ \otimes (R \otimes _\kring C_+)$. The product on the first two factors induces $C_+ \otimes R \otimes_\kring C_+  \rightarrow R \otimes_\kring C  \otimes_\kring C_+$. Now, since $C_+$ is a left $C$-ideal, the product on the second two factors of the latter expression  maps to $C_+$. By associativity, this yields the assertion.

The final statement is immediate from the fact that $\epsilon_C$ is a morphism of $\kring$-bimodules.
\end{proof}

\begin{rem}
\label{rem:qbar}
In this context, the structure morphisms $\rmult$ and $q$ are derived from the restriction of $\lambda_C$:
\[
C_+ \otimes_\kring R \rightarrow R \otimes_\kring C \cong R \ \oplus \  R \otimes_\kring C_+.
\]
This yields the two components 
\begin{eqnarray*}
\overline{\rmult} &:& C_+\otimes_\kring R \rightarrow R \otimes_\kring C_+ \\
\overline{q} &:& C_+ \otimes_\kring R \rightarrow R
\end{eqnarray*}
which are both morphisms of $\kring$-bimodules. 

They determine $\rmult$ and $q$ respectively, using the identification $\ker \epsilon_A \cong R \otimes_\kring C_+$. Namely, $\rmult$ is the composite 
\[
R \otimes_\kring C_+\otimes_\kring R
\stackrel{\id_R \otimes \overline{\rmult}}{\longrightarrow}
R \otimes_\kring R\otimes_\kring C_+
 \stackrel{\mu_R \otimes \id_{C_+}}{\longrightarrow}
  R \otimes_\kring C_+,
 \]
 where $\mu_R$ denotes the product of $R$; $q$ is the composite:
 \[
R \otimes_\kring C_+\otimes_\kring R
\stackrel{\id_R \otimes \overline{q}}{\longrightarrow}
R \otimes_\kring R
\stackrel{\mu_R}{\rightarrow} 
R.
\]
\end{rem}

The morphism $\overline{q}$ of Remark \ref{rem:qbar} can also be identified as follows:

\begin{lem}
\label{lem:qbar_restriction}
The morphism $\overline{q}: C_+ \otimes_\kring R \rightarrow R$  is the restriction of the left $A$-module structure of $R$ (given by the augmentation $\epsilon_A$) along  $C_+ \subset C  \rightarrow A$.
\end{lem}

\subsection{Right augmentations}
\label{subsect:right_aug}

One has the  notion of a right augmented ring over $R$, given by the following counterpart of Definition \ref{defn:left_aug}:

\begin{defn}
\label{defn:right_aug}
A right augmented ring over $R$ is an associative unital ring $A$ equipped with a unit $\eta : R \hookrightarrow A$ and  a {\em right augmentation} $\epsilon : A \rightarrow R$ such that 
\begin{enumerate}
\item 
$\epsilon$ is a morphism of right $R$-modules; 
\item 
$\epsilon \circ \eta = \id_R$; 
\item 
$\ker \epsilon$ is a right ideal of $A$.
\end{enumerate}
\end{defn}

The basic results on left augmentations go through for right augmentations, {\em mutatis mutandis}. Alternatively, one can exploit passage to the opposite ring, using:

\begin{prop}
\label{prop:left/right_aug}
The pair $(\eta : R \rightarrow A, \ \epsilon : A \rightarrow R)$ is a left augmentation for $A$ over $R$ if and only if $(\eta\op : R\op \rightarrow A\op, \ \epsilon\op : A\op \rightarrow R\op)$ is a right augmentation for $A\op$ over $R\op$.
\end{prop}

\begin{exam}
\label{exam:right_aug_ARC}
Suppose that  $\kring$, $R$, $C$, $A$ are as in Section \ref{subsect:distrib} and that  Hypothesis \ref{hyp:A} holds. The analogue of Lemma \ref{lem:R_C} yields a right augmentation over $C$. Namely, if $R$ is augmented with augmentation $\epsilon _R \cn R \rightarrow \kring$, then $A$ is right augmented over $C$ with respect to $\id_C \otimes_\kring \epsilon_R$.
\end{exam}

\section{Quadratic structures}
\label{sect:quad}

This section reviews the basic theory of  homogeneous and nonhomogeneous quadratic algebras. We also treat the case where, for a homogenous quadratic algebra $C$ over $\kring$ and $\kring \rightarrow R$,  $R \otimes_\kring C$ carries the structure of a {\em nonhomogeneous} quadratic algebra.

Some projectivity hypotheses are required when applying quadratic duality and considering Koszul duality, motivating the introduction of: 

\begin{defn}
\label{defn:left_fg_proj}
(Compare \cite[Definition 1.4]{MR4398644}.)
Let $B$ be a $\nat$-graded ring with $B_0=R$. The graded ring $B$ is 
\begin{enumerate}
\item 
$\ell$-left finitely-generated projective (for $\ell \in \nat$) if $B_n$ is finitely-generated projective as a  left $R$-module for each $n \leq \ell$; 
\item 
left finitely-generated projective if each $B_n$ is finitely-generated projective as a  left $R$-module.
\end{enumerate}  
The conditions $\ell$-right finitely-generated projective and right finitely-generated projective are defined analogously.  
\end{defn}

\subsection{Homogeneous quadraticity for $C$}

\begin{nota}
\label{nota:tensor_alg} 
For $W$ a $\kring$-bimodule, $T_\kring (W)$ denotes the tensor algebra on $W$, with underlying $\kring$-bimodule defined recursively by $T^0_\kring (W): = \kring$ and $T^{n+1}_\kring(W) := W \otimes_\kring T^n_\kring (W)$. 
The product $T_\kring (W) \otimes T_\kring (W) \rightarrow T_\kring (W)$  is given by concatenation.
\end{nota}

\begin{lem}
\label{lem:projectivity}
Suppose that $W$ is projective (respectively finitely-generated projective) as a left $\kring$-module, then, for each $n \in \nat$, $T^n _\kring (W)$ is projective (resp. finitely-generated projective)  as a left $\kring$-module. 

Hence $T_\kring (W)$ is left finitely-generated projective as an $\nat$-graded ring if and only if $W$ is finitely-generated projective as a left $\kring$-module.
\end{lem}

Suppose  that $C$ is a homogeneous quadratic ring over $\kring$, in the sense of \cite[Chapter 1]{MR4398644},  i.e., $C$ is augmented with augmentation $\epsilon_C$ and 
\begin{enumerate}
\item 
there is a $\kring$-bimodule $W \subset \ker \epsilon_C$ that generates $C$ over $\kring$, with induced surjection 
\begin{eqnarray}
\label{eqn:T_to_C}
T_\kring (W) \twoheadrightarrow C;
\end{eqnarray}
\item 
$C$ inherits an $\nat$-grading from the tensor algebra $T_\kring (W)$, with $C_0 = \kring$, $C_1 = W$;
\item 
the kernel of (\ref{eqn:T_to_C}) is the two-sided ideal generated by the $\kring$-bimodule $I_C \subseteq W  \otimes_\kring W$ given by the kernel of the multiplication:
\[
I_C := \ker \big( 
W \otimes_\kring W \rightarrow C_2
\big).
\]
\end{enumerate}
The algebra $C$ is determined by the quadratic datum $(\kring; W, I_C)$.

\begin{rem}
\label{rem:aug_ideal_C}
The  augmentation ideal $\ker \epsilon_C$ identifies as the sub $C$-bimodule of positive degree elements, 
$C_+:=\bigoplus_{n>0}  C_n$.
\end{rem}

\begin{lem}
\label{lem:proj_W}
Suppose that $(\kring; W, I_C)$ is a quadratic datum, where $W$ and $C_2:= W\otimes_\kring W/ I_C$ are both left $\kring$-projective. Then $I_C$ is left $\kring$-projective and there is a direct sum decomposition of projective left $\kring$-modules 
$
W\otimes_\kring W \cong I_C \  \oplus \  C_2$.

If, in addition, $W$ is finitely-generated projective as a left $\kring$-module, then so are both $I_C$ and $C_2$. 
\end{lem}

We extract the hypotheses of Lemma \ref{lem:proj_W} as the following Definition:

\begin{defn}
\label{defn:left_(fg)_proj_qdatum}
A quadratic datum $(\kring; W, I_C)$ is $2$-left projective (respectively $2$-left finitely-generated projective) if  
$W$ and $W\otimes_\kring W/ I_C$ are both left $\kring$-projective (resp. finitely-generated projective).
\end{defn}

For later usage, we introduce the following: 

\begin{nota}
\label{nota:length_filt_C}
Denote by $F_\bullet C$ the length filtration  associated to the graded algebra $C$, 
$
F_n C := \bigoplus_{i=0}^n C_i.
$ 
\end{nota}

Clearly this is  a filtration of the algebra $C$: each $F_n C$ is a $\kring$-bimodule and,  for $s, t \in \nat$, the product of $C$ restricts to 
$ 
F_s C \otimes_\kring F_t C 
\rightarrow 
F_{s+t}C.
$

\subsection{Half base change for homogeneous quadratic algebras}

Suppose given an  inclusion of unital, associative rings $\eta_R : \kring \hookrightarrow R$. Then, for any $\kring$-bimodule $X$, one can form the $R \otimes \kring\op$-module $R \otimes_\kring X$ and ask whether this has a compatible $R$-bimodule structure. If this is the case, the right $R$-module structure gives an exchange map 
$
\lambda_X : X \otimes_\kring R \rightarrow R \otimes_\kring X
$
and is determined by this. 

\begin{rem}
\label{rem:exchange_map_bimodule}
One can give necessary and sufficient conditions on a map of the form $\lambda_X$ ensuring that it induces a $R$-bimodule structure on  $R \otimes_\kring X$ as above.
\end{rem}

We consider a  quadratic datum $(\kring; W , I_C)$ and suppose that the following holds:

\begin{hyp}
\label{hyp:I_C_left_split_inclusion}
The inclusion $I_C \hookrightarrow W \otimes_\kring W$ admits a retract in left $\kring$-modules. 
\end{hyp}

\begin{exam}
This hypothesis holds if $(W \otimes_\kring W) /I_C$ is projective as a left $\kring$-module.
\end{exam}

Consider the $R \otimes \kring\op$-modules $R \otimes_\kring W$ and $R \otimes_\kring I_C$; Hypothesis \ref{hyp:I_C_left_split_inclusion} ensures that $I_C \hookrightarrow W \otimes_\kring W$ induces an inclusion of $R\otimes \kring\op$-modules:
\[ 
R \otimes_\kring I_C \hookrightarrow R \otimes_\kring (W \otimes_\kring W).
\]

In order for these to yield a quadratic datum over $R$, we require the following:

\begin{hyp}
\label{hyp:extension}
Given a quadratic datum $(\kring ; W ,I_C)$ satisfying Hypothesis \ref{hyp:I_C_left_split_inclusion}, suppose the following:
\begin{enumerate}
\item 
there exists an $R$-bimodule structure on $R \otimes_\kring W$ extending the canonical $R \otimes \kring \op$-module structure, encoded by the exchange map $\lambda_W : W \otimes_\kring R \rightarrow R \otimes_\kring W$;
\item 
with respect to the induced $R$-bimodule structure on $R \otimes_\kring (W \otimes_\kring W) \cong (R\otimes_\kring W)\otimes_R (R \otimes_\kring W)$, $R \otimes_\kring I_C$ is a sub $R$-bimodule (i.e.,  is stable under the right $R$-action).
\end{enumerate}
\end{hyp}

The following is clear:

\begin{lem}
\label{lem:half_base_change_qdatum}
Suppose that Hypothesis \ref{hyp:extension} holds for $(\kring ; W ,I_C)$. Then $(R ; R \otimes_\kring W, R \otimes_\kring I_C)$ is a quadratic datum over $R$. Moreover, if $(\kring ; W ,I_C)$ is $2$-left  projective (respectively finitely-generated projective), then so is $(R ; R \otimes_\kring W, R \otimes_\kring I_C)$.
\end{lem}

For the remainder of this subsection, we suppose that Hypothesis \ref{hyp:extension} holds, so that we have the quadratic datum $(R; R \otimes_\kring W, R \otimes_\kring I_C)$. One can thus form the associated homogeneous  quadratic algebra over $R$, denoted below by $A (\lambda_W)$, referencing the exchange map $\lambda_W$ that encodes the $R$-bimodule structure of $R \otimes_\kring W$.  

By construction, this fits into a commutative diagram of morphisms of unital, associative rings:
\[
\xymatrix{
\kring 
\ar@{^(->}[r]
\ar@{^(->}[d]
&
C
\ar[d]
\\
R 
\ar@{^(->}[r]
&
A (\lambda_W).
}
\]
One thus has a map of $R \otimes C\op$-modules $R \otimes_\kring C \rightarrow A (\lambda_W)$. 

\begin{prop}
\label{prop:homog_quadratic_distributive}
Hypothesis \ref{hyp:A} holds for $A(\lambda_W)$; i.e., 
the product of $A (\lambda_W)$ induces an isomorphism of $R \otimes C\op$-modules 
$$
R \otimes_\kring C \stackrel{\cong}{\rightarrow} A(\lambda_W).
$$
\end{prop}  

\begin{proof}
By definition, $C$ is the homogeneous quadratic algebra $T_\kring (C)/ \langle I_C \rangle$ over $\kring$ and  $A(\lambda_W)$ is the homogeneous quadratic algebra 
$
A(\lambda_W):= T_R (R \otimes_\kring W )/ \langle R \otimes_\kring I_C \rangle.
$ 
In particular, there is a commutative diagram in $R$-modules:
\[
\xymatrix{
R \otimes_\kring T_\kring (W) 
\ar[r]^\cong
\ar@{->>}[d]
&
T_R (R \otimes_\kring W) 
\ar@{->>}[d]
\\
R \otimes_\kring C
\ar[r]
&
A(\lambda_W)
}
\]
in which the right hand vertical map is a morphism of algebras and the horizontal maps are induced by the product of $T_R (R \otimes_\kring W)$ and of $A(\lambda_W)$ respectively. 

The kernel of the left hand vertical map is the image of $R \otimes_\kring \langle I_C \rangle$ in $R \otimes_\kring T_\kring (W)$, by right exactness of $R \otimes_\kring -$.  The kernel of the right hand vertical map is $\langle R \otimes_\kring I_C \rangle$, by definition of $A (\lambda_W)$. To prove the result, it suffices to prove that the image of $R \otimes_\kring \langle I_C \rangle$ in $T_R (R \otimes_\kring W) $ is equal to $\langle R \otimes_\kring I_C \rangle$. 

That the image is contained in $\langle R \otimes_\kring I_C \rangle$ is immediate. Now $\langle R \otimes_\kring I_C \rangle$ is generated as a two-sided ideal by $R \otimes_\kring I_C$, which is clearly in the image. To conclude, it suffices to show that the image is a two-sided ideal in $T_R (R \otimes_\kring W)$; this follows by using the exchange law.
\end{proof}

\subsection{The nonhomogeneous case}
\label{subsect:non-homog}

Given a homogeneous quadratic ring $C$ as above, 
we suppose that we are in  the framework of Section \ref{subsect:distrib}, so that we have a commutative diagram
of morphisms of unital associative rings
\[
\xymatrix{
\kring 
\ar@{^(->}[r]^{\eta_C}
\ar@{^(->}[d]_{\eta_R}
&
C 
\ar[d]
\\
R \ar[r]
&
A
}
\]
such that the product of $A$ induces an isomorphism 
$
R \otimes_\kring C \stackrel{\cong}{\rightarrow} A
$ 
of $R \otimes C\op$-modules.

\begin{rem}
\ 
\begin{enumerate}
\item 
By Lemma \ref{lem:R_C}, the augmentation of $C$ induces a left augmentation of $A$.
\item 
Since the unit  $\kring \rightarrow C$ splits as a morphism of $\kring$-modules, the map $R \rightarrow A$ is a split inclusion of $R$-modules.
\end{enumerate}
\end{rem}

\begin{nota}
\label{nota:filt_A}
For $n \in \nat$, set $F_n A := R \otimes _\kring F_n C$ as left $R$-modules (where $F_n C$ is as in Notation \ref{nota:length_filt_C}); in particular $F_0 A = R$ and $F_1A = R \ \oplus \ R \otimes_\kring W$.   
\end{nota}

One has the following projectivity property:

\begin{lem}
\label{lem:proj_R}
Suppose that $F_n C$ is projective (respectively finitely-generated projective) as a left $\kring$-module, then $F_nA$ is projective (resp. finitely-generated projective) as a left $R$-module. 
\end{lem}

The following is clear:

\begin{lem}
\label{lem:filt_R_mod}
For $n \in \nat$, there is an isomorphism of left $R$-modules 
$ 
F_n A \cong \bigoplus_{i=0}^n R \otimes_\kring C_i$.
 In particular, the inclusions $F_n C \subset F_{n+1} C$ induce split inclusions of left $R$-modules 
$F_n  A \hookrightarrow F_{n+1}A$. This filtration is exhaustive: $A \cong \lim_{\substack{\rightarrow \\n}} F_n A$. 
\end{lem}

It is natural to ask when $(F_nA \mid n\in\nat)$ defines a filtration of the algebra $A$.

\begin{rem}
\label{rem:associated_graded}
If $F_\bullet A$ makes $A$ into a filtered algebra, then the associated graded 
$
\grad A:= \bigoplus_{n\in \nat} \grad_n A
$ (where  $\grad_n A := F_n A/ F_{n-1} A$) is an $\nat$-graded algebra.  Then  $\grad_n A \cong R \otimes_\kring C_n$ as a left $R$-module and has the structure of an $R$-bimodule, with the right multiplication induced by the restriction 
$ 
F_n A \otimes F_0 A \rightarrow F_{n} A$ of the product of $A$. 
\end{rem}

\begin{prop}
\label{prop:filt_hyp}
The left $R$-module filtration $(F_n A \mid n \in \nat)$ of $A$ yields a filtration of the algebra $A$ if and only if the restriction of the multiplication of $A$ to the image of $W \otimes_\kring R$ in $ A\otimes_\kring A$ maps to $F_1 A$:
\[
W \otimes_\kring R \rightarrow F_1 A = R \  \oplus \ R \otimes_\kring W.
\]
 
If this condition holds, then
\begin{enumerate}
\item 
the map $W \otimes _\kring R \rightarrow  R \otimes_\kring W$ obtained by composing with the projection to $F_1 A/F_0A \cong R \otimes_\kring W$ induces an $R$-bimodule structure on $R \otimes_\kring W$ extending the canonical $R \otimes \kring\op$-module structure;
\item 
the composite $W \otimes_\kring R \rightarrow F_1 A \twoheadrightarrow R$ is the restriction of the left $A$-module structure of $R$ to $W$.
\end{enumerate}
\end{prop}

\begin{proof}
Since $F_0 A =R $ is a subring of $A$, the condition given in the statement is equivalent to the requirement that the multiplication restricted to $F_1 A \otimes_\kring F_0 A$ maps to $F_1 A$. This condition is necessary. 

It remains to show that the condition is sufficient. The condition implies that $F_1 A$ is a sub $R$-bimodule of $A$. Thus the inclusion $F_1 A \hookrightarrow A$ induces a morphism of algebras
$
T_R (F_1 A) \rightarrow A
$.  
One checks that this is surjective as follows. The hypothesis implies that the morphism $T^n_R (F_1 A) \rightarrow A$ maps to $F_n A$ for each $n \in \nat$ and one checks by induction on $n$ that this yields a surjection $T^n_R (F_1 A) \twoheadrightarrow F_n A$. Since the filtration $F_\bullet A$ is exhaustive (by Lemma \ref{lem:filt_R_mod}), this gives the claimed surjectivity. 

Since the multiplication $A \otimes_R A \rightarrow A$ is a morphism of $R$-bimodules, the image of $T^n_R(F_1 A) \rightarrow A$ is a sub $R$-bimodule of $A$. It follows that $F_n A$ is a sub $R$-bimodule of $A$. 
 Thus, for any $m, n \in \nat$, the above construction yields the commutative diagram of morphisms of $R$-bimodules:
\[
\xymatrix{
T^m _R (F_1 A) \otimes_R T^n _R (F_1 A) 
\ar[r] ^(.6)\cong 
\ar@{->>}[d]
&
T^{m+n} _R (F_1 A)
\ar@{->>}[r]
&
F_{m+n}A 
\ar@{^(->}[d]
\\
F_m A \otimes_R F_n A 
\ar[r]
&
A \otimes_R A 
\ar[r]
&
A,
}
\]
in which the bottom right hand map is the product of $A$. 
 This shows that the product restricts to $F_m A \otimes_R F_n A \rightarrow F_{m+n} A$, for any $m,n$,  as required. 
\end{proof}

\begin{rem}
Henceforth, we suppose that the hypothesis of Proposition \ref{prop:filt_hyp} holds. In particular, $R \otimes _\kring W$ has the structure of a right $R$-module, given by the restriction of $\rmult$: 
\[
\rmult : (R \otimes_\kring W) \otimes_\kring R 
\rightarrow 
(R \otimes_\kring W)
\]
Likewise, we have the restriction of the structure morphism $q$: 
\[
\tilde{q} : W \otimes_\kring R 
\rightarrow 
R
\]
that induces a left $T_\kring W$-module structure on $R$. 
\end{rem}

\begin{defn}
\label{defn:non-homog_quad}
(Cf. \cite[Definition 3.2]{MR4398644}.) A filtered algebra $(A, F_\bullet A)$ is said to be nonhomogeneous quadratic if the associated graded $\grad A$ is homogeneous quadratic over $F_0 A$. 
\end{defn}

\begin{rem}
 Such definitions are standard when $\kring$ is a field; for example they occur in the foundational work of Priddy \cite{MR0265437} on Koszul algebras.
 \end{rem}

\begin{prop}
\label{prop:A_quadratic}
Suppose that Hypothesis \ref{hyp:A} is satisfied and that  $C$ is the quadratic algebra corresponding to the 
$2$-left projective quadratic datum $(\kring; W, I_C)$. 

If  the condition of Proposition \ref{prop:filt_hyp} is satisfied,  then $A$ equipped with the filtration $F_n A$ is a nonhomogeneous quadratic algebra over $R: = F_0 A$. The associated graded $\grad A$ is isomorphic to the quadratic algebra over $R$ associated to the quadratic datum 
$$(R; R \otimes_\kring W, I_A := R \otimes_\kring I_C \subset (R \otimes_\kring W) \otimes_R (R \otimes_\kring W))$$ 
 and this is $2$-left projective; it is finitely-generated if and only if  $(\kring; W, I_C)$ is so.
 
Moreover, the product of $\grad A$ induces an isomorphism of $R\otimes C\op$-modules $
R \otimes_\kring C \stackrel{\cong}{\rightarrow} \grad A.
$
\end{prop}

\begin{proof}
By construction, $\grad_0 A=R$ and  $\grad_1 A$ is isomorphic to $R \otimes _\kring W$ as an $R$-bimodule, using the right module structure provided by $\rmult$ (this corresponds to specifying the exchange map $\lambda_W$). Moreover, from the proof of  Proposition \ref{prop:filt_hyp} it is clear that $\grad A$ is generated as an algebra over $R$ by $R \otimes_\kring W$. 

Consider the kernel of the product restricted to $(R \otimes_\kring W) \otimes _R (R \otimes _\kring W) \cong R \otimes_\kring (W \otimes_\kring W)$ mapping to $\grad_2 A$. This product identifies as the map of left $R$-modules
\begin{eqnarray}
\label{eqn:prod_grad_2}
R \otimes_\kring W \otimes_\kring W \rightarrow R \otimes _\kring C_2
\end{eqnarray}
induced by the product $W \otimes_\kring W \twoheadrightarrow C_2$. Since the latter fits into the split short exact sequence of left $\kring$-modules $0 \rightarrow I_C \rightarrow W \otimes_\kring W \twoheadrightarrow C_2 \rightarrow 0$ (which splits by the left projectivity hypothesis), it follows that the kernel of (\ref{eqn:prod_grad_2}) is isomorphic to $R \otimes_ \kring I_C$ as a left $R$-module and this is a submodule of $(R \otimes_\kring W) \otimes _R (R \otimes _\kring W)$. 

Now (\ref{eqn:prod_grad_2}) is necessarily a morphism of $R$-bimodules, where the $R$-bimodule structure of $R \otimes_\kring W \otimes_\kring W$ is induced by that of $R \otimes_\kring W$. It follows that $R \otimes_\kring I_C$ is a sub $R$-bimodule of $(R \otimes_\kring W) \otimes_R (R \otimes_\kring W)$, thus Hypothesis \ref{hyp:extension} is satisfied, so that $(R; R\otimes_\kring W, R \otimes_\kring I_C)$ is a quadratic datum by Lemma \ref{lem:half_base_change_qdatum}. The projectivity statement follows from Lemma \ref{lem:half_base_change_qdatum}.

Form the associated quadratic algebra $A (\lambda_W)$. By construction, one has a surjective morphism 
$ 
A (\lambda_W) \twoheadrightarrow \grad A
$ 
extending the identifications $\grad_0 A = R $ and $\grad_1 A = R \otimes_\kring W$. It remains to show that this is an isomorphism; this is true in degrees $0, 1, 2$ by construction.  To establish this isomorphism, it is equivalent to show that $\grad A$ is quadratic. This follows from the fact that $C$ is quadratic. 

Finally, the isomorphism of $R\otimes C\op$-modules $R \otimes_\kring C \cong A$ given by hypothesis passes to the associated graded. 
\end{proof}

\subsection{The mirror situation}
\label{subsect:crossing_hands}

One can  reverse the roles of $R$ and $C$ above.  The modifications that are required are the following: 
\begin{enumerate}
\item
$R$ is a homogeneous quadratic algebra over $\kring$; 
\item 
the projectivity properties are for the {\em right} module structures; 
\item 
the half base change functor is taken to be $- \otimes_\kring C$; 
\item 
the exchange map yields the {\em left} module structure.
\end{enumerate}

The results of the previous subsections have the obvious counterparts.

\section{Quadratic duality}
\label{sect:qdual}

This section reviews quadratic duality working over a unital, associative base ring, using \cite{MR4398644} as reference. 

\subsection{Recollections}
Throughout this subsection, $R$ is a fixed unital ring.  We primarily focus upon the theory using left duality, for which the relevant projectivity hypotheses are for the  left $R$-module structure. 

\begin{hyp}
\label{hyp:q_datum_proj}
The quadratic datum 
$(R;V , I \subset V \otimes_ R V)$ is $2$-left finitely-generated projective.
\end{hyp}

\begin{rem}
Under this hypothesis, Lemma \ref{lem:projectivity} implies that the associated quadratic algebra $T^R (V)/ \langle I \rangle$ is $2$-left finitely-generated projective quadratic, in the sense of Definition \ref{defn:left_fg_proj}.
\end{rem}

Recall that the  (left) duality  functor $(-)^\vee := \hom_R (-, R)$ is a (contravariant) endofunctor of the category of $R$-bimodules. 

\begin{lem}
\label{lem:duality_bimodules_proj}
Suppose that $M$ is a $R$-bimodule which is finitely-generated projective as a left $R$-module, then the $R$-bimodule $M^\vee = \hom_R (M, R)$ is finitely-generated projective as a {\em right} $R$-module.
\end{lem}
 
\begin{defn}
\label{defn:dual_quad_datum}
(Cf. \cite[Proposition 1.6]{MR4398644}.) 
The (left) dual quadratic datum to $(R; V, I)$  is 
$$
(R; V^\vee = \hom_R (V, R), I^\perp	 := \hom_R ((V \otimes_R V)/I, R)).
$$ 
\end{defn}

Applying Lemma \ref{lem:duality_bimodules_proj}, one has:

\begin{lem}
\label{lem:quadratic_dual_properties}
Suppose that the quadratic datum $(R; V, I)$ satisfies Hypothesis \ref{hyp:q_datum_proj}. Then 
the dual quadratic datum is $2$-right finitely-generated projective. 
\end{lem}

\begin{proof}
This is proved by standard arguments. Note that the hypothesis that $V$ is finitely-generated projective as a left $R$-module gives the isomorphism $\hom_R (V \otimes_R V, R)
\cong 
\hom_R (V, R)\otimes_R \hom_R (V, R)$.
\end{proof}

\begin{rem}
\label{rem:left_right_duality}
\ 
\begin{enumerate}
\item 
One has the counterpart of the above switching the r\^ole of left and right $R$-modules, thus imposing projectivity hypotheses 
for the right $R$-module structure and using the (right) duality functor $\hom_{R\op} (-, R)$. 
\item 
The duality functor $\hom_R (-,R)$ induces an anti-equivalence between the category of quadratic data satisfying the (left) projectivity Hypothesis \ref{hyp:q_datum_proj} and that of quadratic data satisfying the respective projectivity hypothesis with respect to the right $R$-module structure. (Compare \cite[Proposition 1.6]{MR4398644}.)
\end{enumerate}
\end{rem}

\begin{nota}
\label{nota:B!}
Given a quadratic datum $(R; V, I)$ satisfying Hypothesis \ref{hyp:q_datum_proj} and with associated homogeneous quadratic algebra $B:= T^R (V)/ \langle I \rangle$, denote by $B\qd$ the quadratic algebra associated to the dual quadratic datum $(R; \hom_R (V,R), \hom_R(B_2,R))$. By definition, $B\qd$ is the quadratic dual of $B$.

The same notation will be used when the role of left and right $R$-modules is switched, using the (right) duality functor $\hom_{R\op} (-, R)$. The intended meaning of the notation should always be clear from the context.
\end{nota}

\subsection{The quadratic dual of $\grad A$}
\label{subsect:quad_dual_grad_A}

Consider $C$, the homogeneous quadratic algebra over $\kring$ associated to the quadratic datum $(\kring; W , I_C \subset W \otimes_\kring W)$, which we assume  to be $2$-left finitely-generated projective. 
 The (left) dual quadratic datum is $(\kring; \hom_\kring (W, \kring), I_{C\qd}=  \hom_\kring (C_2, \kring))$, which is $2$-right finitely-generated projective. As in Notation \ref{nota:B!}, the associated homogeneous quadratic algebra over $\kring$ is denoted $C\qd$.

Now, suppose that we are in the context of Section \ref{sect:quad}, so that we have the isomorphism of $R \otimes C\op$-modules $R \otimes_\kring C \stackrel{\cong}{\rightarrow} A$ and suppose that 
 $F_nA := R \otimes _\kring F_n C$ yields a filtered algebra structure on $A$. By Proposition \ref{prop:A_quadratic}, the associated graded algebra $\grad A$ is quadratic, determined by the quadratic datum:
\begin{eqnarray}
\label{eqn:q_datum_A}
(R; R \otimes_\kring W , I_A := R \otimes_\kring I_C).
\end{eqnarray}
This satisfies Hypothesis \ref{hyp:q_datum_proj}.

\begin{rem}
\label{rem:exhange_map_lambda_W}
By  Lemma \ref{lem:half_base_change_qdatum}, the $R$-bimodule structure of $R \otimes_\kring W$ is encoded in the exchange map $
\lambda_W : W \otimes_\kring R \rightarrow R\otimes_\kring W$ of $\kring$-bimodules.
\end{rem}

\begin{prop}
\label{prop:dual_quadratic_datum_R_C}
The (left) dual of the quadratic datum $
(R; R \otimes_\kring W , I_A := R \otimes_\kring I_C)
$ is isomorphic to 
\begin{eqnarray}
\label{eqn:dual_q_datum_A}
(R;  \hom_\kring (W, \kring)\otimes_\kring  R , I_{C\qd} \otimes_\kring R)
\end{eqnarray}
where the $R$-bimodule structure of $\hom_\kring (W, \kring)\otimes_\kring  R$ extends the canonical $\kring \otimes R\op$-module structure. This is $2$-right finitely-generated projective.
\end{prop}

\begin{proof}
The dual of $R \otimes_\kring W$ is, by definition, $\hom_R (R \otimes_\kring W, R) \cong \hom_\kring (W, R)$. Under the finitely-generated projectivity hypothesis on $W$, this is isomorphic to $\hom_\kring (W, \kring) \otimes_\kring R$ as a  $\kring \otimes R\op$-module and the full $R$-bimodule structure is obtained by transport of structure. 

The left $R$-module structure is described explicitly in terms of the exchange map $\lambda_W$ (see Remark \ref{rem:exhange_map_lambda_W}) by giving the exchange map (in this `opposite' context)
\[
R \otimes_\kring \hom_\kring (W, \kring) \rightarrow \hom_\kring (W, \kring) \otimes_\kring R \cong \hom_\kring ( W, R).
\] 
This is the morphism of $\kring$-bimodules that sends $r \otimes \phi \in R \otimes_\kring \hom_\kring (W, \kring)$ to the map $w \mapsto \tilde {\phi} (\lambda (w \otimes r))$, where $\tilde{\phi}$ is the image of $\phi$ in $\hom_R (R\otimes_\kring W, R)$ via the unit $\eta_R$.

By construction of $I_A$, there is a short exact sequence of $R$-bimodules 
\[
0 
\rightarrow 
I_A = R \otimes_\kring I_C 
\rightarrow 
R \otimes_\kring (W \otimes _\kring W) 
\cong 
(R \otimes_\kring W) \otimes_R (R \otimes_\kring W) 
\rightarrow 
R \otimes _\kring C_2
\rightarrow 
0
\]
that splits as a sequence of left $R$-modules. Applying the duality functor $\hom_R (-, R)$  thus gives the short exact sequence of $R$-bimodules
\[
0
\rightarrow 
\hom_\kring (C_2, R) 
\rightarrow 
\hom_\kring (W \otimes_\kring W, R) 
\cong 
\hom_\kring (W , R) 
\otimes_R
\hom_\kring (W , R) 
\rightarrow 
\hom_R ( I_A, R) 
\cong 
\hom_\kring (I_C, R)
\rightarrow 
0
\]
that is split as a sequence of {\em right} $R$-modules. Moreover, $\hom_\kring (C_2, R) $ is isomorphic to $\hom_\kring (C_2, \kring) \otimes_\kring R$ as a $\kring \otimes R\op$-module, and is a sub $R$-bimodule of $\hom_\kring (W , R) 
\otimes_R
\hom_\kring (W , R)$ by construction. This yields the required identification, since $I_{C\qd} = \hom_\kring (C_2, \kring)$ as a $\kring$-bimodule.

The projectivity statement follows from Lemma \ref{lem:quadratic_dual_properties}.
\end{proof}

The dual quadratic datum (\ref{eqn:dual_q_datum_A}) yields the associated homogeneous quadratic algebra $(\grad A)\qd$ over $R$. This is generated over $R$ by $\hom_\kring (W,\kring) \subset \hom_\kring (W,\kring)\otimes_\kring R$ (that this is an  inclusion is spelled out in Lemma \ref{lem:include_dual_W} below). This inclusion, together with the given  inclusion $\kring \hookrightarrow R$, induces a morphism of rings $C\qd \rightarrow (\grad A)\qd$ and hence, using the product of $(\grad A)\qd$, a morphism of $C\qd \otimes R\op$-modules 
$
C\qd \otimes_\kring R \rightarrow (\grad A)\qd$.

\begin{prop}
\label{prop:dual_C_R}
The map $C\qd \otimes_\kring R \rightarrow (\grad A)\qd$ is a bijection. 
\end{prop}

\begin{proof}
This follows from Proposition \ref{prop:homog_quadratic_distributive}.
\end{proof}

\begin{rem}
One has the obvious counterparts of these results for (right) quadratic duality, as in Section \ref{subsect:crossing_hands}.
\end{rem}

\section{The DG algebra associated to a left augmented nonhomogeneous quadratic algebra}
\label{sect:dgalg}

The purpose of this section is to explain that, under suitable hypotheses, the quadratic dual derived from a left augmented nonhomogeneous quadratic algebra carries a natural DG algebra structure, following \cite{MR4398644}. This is then identified in the very special case of interest here. 

\subsection{Generalities}
\label{subsect:DG_generalities}

Suppose that $(F_nA \mid n \in \nat)$ is a filtered ring structure on $A$ where $F_0 A =R$, providing the unit $\eta_A : R \hookrightarrow A$. Suppose furthermore that $A$ is left augmented, with left augmentation $\epsilon_A : A \rightarrow R$. This restricts to $F_nA \rightarrow R$, a morphism of left $R$-modules, for each $n \in \nat$; for $n=0$ it is an isomorphism. In particular, this gives the splitting of left $R$-modules
\[
F_nA \cong R \oplus \overline{F_nA},
\]
where $\overline{F_nA}$ is the kernel of $\epsilon_A|_{F_nA}$. Setting $V:= \overline{F_1 A}$, one has $F_1 A \cong V \oplus R$ as left $R$-modules and hence $V$ is a sub left $R$-module of $A$. Moreover, the projection $F_1 A \twoheadrightarrow F_1 A/F_0 A$ induces an isomorphism of left $R$-modules $V\cong  F_1 A/F_0 A$, where the right hand side is naturally an $R$-bimodule.  This induces an $R$-bimodule structure on $V$ (this is not in general a sub $R$-{\em bimodule} of $A$).

Using this $R$-bimodule structure on $V$, the ring structure of $\grad A$ induces the  $R$-bimodule map:
$
V \otimes_R V \rightarrow F_2 A/ F_1A.
$

\begin{nota}
\label{nota:I_A} 
Denote by $I_A$  the kernel of $V \otimes_R V \rightarrow F_2 A/ F_1A$, so that $I_A \subset V \otimes_R V$ is a sub $R$-bimodule.
\end{nota}

Working with $A$ rather than the associated graded, one has the  product
$ 
\mu : V \otimes V 
\rightarrow 
F_2 A
$  
(this does not factor across $V \otimes_R V$ in general). It is a morphism of left $R$-modules (but not of $R$-bimodules in general).

Following \cite[Chapter 3]{MR4398644}, define $\widehat{I}$ by the pullback diagram in left $R$-modules:
\[
\xymatrix{
\widehat{I}
\ar@{^(->}[r]
\ar@{->>}[d]
&
V \otimes V 
\ar@{->>}[d]
\\
I_A 
\ar@{^(->}[r]
& V\otimes_R V,
}
\]
in which the right hand vertical map is the canonical surjection.

\begin{lem}
\label{lem:map_p}
The product $\mu$ induces a map $p : \widehat {I} \rightarrow \overline{F_1A} =V \subset F_1 A$ of left $R$-modules.
\end{lem}

\begin{proof}
Composing $\widehat{I} \hookrightarrow V \otimes V$ with $\mu$ gives a map $\widehat{I} \rightarrow F_2 A$. By definition of $I$, this  maps to $F_1 A \subset F_2 A$. To conclude, it suffices to observe that it maps to the kernel of $\epsilon_A$; this follows since $\ker \epsilon_A$ is a (non-unital) subring of $A$.
\end{proof}

\begin{rem}
\label{rem:zero_curvature}
This is a particular case of the constructions of \cite[Section 3.2]{MR4398644}. The Lemma essentially asserts that the {\em curvature} $h$ appearing in equation (3.2) of {\em loc. cit.} is zero in this left-augmented context.
\end{rem}

\begin{rem}
\label{rem:p_s}
\ 
\begin{enumerate}
\item 
Suppose that the  projection $\widehat{I}\twoheadrightarrow I_A$ admits a section 
$
 s:  I_A
 \hookrightarrow 
 \widehat{I} 
$
 in left $R$-modules. Using this, one can form the composite morphism of left $R$-modules 
$ 
p \circ s : I_A \rightarrow V$. Clearly, 
 in general this depends upon the choice of section, $s$. 
\item 
If $I_A$ is projective as a left $R$-module, then such a section exists. Indeed, since $\widehat{I}$ is defined as the pullback of $I_A$, it suffices to exhibit a section of the canonical surjection $V \otimes V \rightarrow V \otimes_R V$ in left $R$-modules; such a section always exists if $V$ is finitely-generated projective as a left $R$-module.
\end{enumerate}
\end{rem}

\begin{nota}
\label{nota:pairing}
For $M$ a left $R$-module, denote by $\langle - , - \rangle $ the pairing 
\[
M \otimes \hom_R (M, R) \rightarrow R
\]
induced by evaluation. (This is a morphism of left $R$-modules, for the action on the domain given by that of the left hand $M$.)
\end{nota}

\begin{nota}
\label{nota:q_tilde}
Denote by
\begin{enumerate}
\item 
$q : V\otimes R \rightarrow R$ the composite of the product $V \otimes R \rightarrow F_1 A$ with the projection to $R$; 
\item 
$\tilde{q} : V\otimes V \otimes \hom_R (V, R) \rightarrow R$  the composite of the map induced by evaluation  
$\langle -, - \rangle :  V \otimes \hom_R (V, R) \rightarrow R$ with $q$, so that $\tilde{q} (x \otimes y \otimes \alpha) = q (x \otimes \langle y, \alpha \rangle)$.
\end{enumerate}
\end{nota}

\subsection{Positselski's DG algebra}

In addition to the hypotheses of the previous section, we now suppose that the graded ring $\grad A$ is quadratic and satisfies the appropriate left projectivity hypothesis. This allows Positselski's results \cite[Proposition 3.16]{MR4398644} and \cite[Theorem 3.27]{MR4398644} to be applied to construct a natural DG algebra.

\begin{thm}
\label{thm:DG_posit}
Let  $A$ be a filtered algebra, with filtration $F_\bullet A$, together with a compatible left augmentation $\epsilon_A : A \rightarrow R = F_0 A$.  Suppose furthermore that:
\begin{enumerate}
\item 
$\grad A$ is homogeneous quadratic with generating $R$-bimodule $V = F_1 A/F_0A$ and quadratic relations $I_A \subset V\otimes_R V$;
\item 
$\grad A$ is $3$-left finitely-generated projective. 
\end{enumerate}
Thus the dual homogeneous quadratic ring $(\grad A)\qd$ is determined by 
$(R; \hom_R (V,R), \hom_R (\grad_2 A, R) ) $
which is a $2$-right finitely-generated projective quadratic datum.

Then $(\grad A)\qd$ has a natural DG algebra structure with differential $d$ determined by its restriction to $R$ and $ \hom_R (V,R)$, given respectively by 
\begin{eqnarray}
\label{eqn:dR}
\langle v, d r \rangle &:= & q (v \otimes r)  \mbox{ for any $v \in V$, $r\in R$;}
\\
\label{eqn:dphi}
\langle i , d \phi \rangle &:=& \langle p\circ s (i), \phi \rangle - \tilde{q} (s(i)\otimes \phi) \mbox{ for any $v \in V$, $\phi\in\hom_R (V,R) $,}
\end{eqnarray}
where $s$ is any left $R$-module section of $\widehat{I} \twoheadrightarrow I_A$. 
\end{thm}

\begin{proof}
This is a special case of the results of \cite[Chapter 3]{MR4398644} and the specific results cited above. As explained in Remark \ref{rem:zero_curvature}, the left augmentation ensures that the curvature $h$ appearing in \cite[Proposition 3.16]{MR4398644} is zero (as in \cite[Theorem 3.27]{MR4398644}) and hence that the Proposition yields a DG algebra structure. 

The projectivity hypothesis ensures the existence of a section $s$ as required. Positselski proves (\cite[page 54]{MR4398644}) independence of choices, which implies that the differential as defined in the statement is independent of the choice of $s$.
\end{proof}

\begin{rem}
\ 
\begin{enumerate}
\item 
The projectivity hypothesis on $\grad_1 A$ and $\grad_2 A$ is used in forming the quadratic dual of the quadratic datum. 
\item 
The projectivity hypothesis on $\grad_3 A$ is used in the proof of \cite[Proposition 3.16]{MR4398644}  when checking the relations. 
\end{enumerate}
\end{rem}

\subsection{Specializing to the case $A \cong_{\mathrm{bimod}} R \otimes_\kring C$}
\label{subsect:dgalg_application}
We now specialize to the case where we have a commutative diagram of unital ring maps 
\[
\xymatrix{
\kring 
\ar@{^(->}[r]^{\eta_C}
\ar@{^(->}[d]_{\eta_R}
&
C 
\ar[d]
\\
R \ar[r]
&
A}
\]
such that the product of $A$ induces an isomorphism of $R\otimes C\op$-modules
$
R \otimes_\kring C \stackrel{\cong}{\rightarrow} A.
$

Moreover, we suppose that the following hypotheses hold:
\begin{enumerate}
\item 
 $C$  is homogeneous quadratic over $\kring$, associated to the quadratic datum $(\kring; W , I_C)$; 
 \item 
 the conclusion of Proposition \ref{prop:A_quadratic} holds, so that $A$ is a filtered algebra with filtration $F_n A = R \otimes_\kring F_n C$, hence the associated graded $\grad A$ is homogenous quadratic over $R$, associated to the quadratic datum: 
$
(R; R \otimes_\kring W, R \otimes_\kring I_C).
$
\item 
$C$ is left finitely-generated projective over $\kring$, so that $\grad A$ is left finitely-generated projective over $R$. 
\end{enumerate}

By Proposition \ref{prop:dual_quadratic_datum_R_C}, the homogeneous quadratic algebra $(\grad A)\qd$ is defined by the quadratic datum 
\[
(R; \hom_\kring (W, \kring) \otimes_\kring R, I_C^\perp \otimes_\kring R)
\]
that is $2$-right finitely-generated projective.

\begin{lem}
\label{lem:include_dual_W}
The inclusion of $\kring$-bimodules $\kring \hookrightarrow R$ induces an inclusion of right $\kring$-modules:
\[
\hom_\kring (W, \kring) \hookrightarrow \hom_\kring (W, \kring) \otimes_\kring R.
\]
\end{lem}

\begin{proof}
By hypothesis, $W$ is finitely-generated projective as a left $\kring$-module, so $\hom_\kring (W, \kring)$ is finitely-generated projective as a right $\kring$-module, hence flat. This proves the injectivity of the map, which is clearly one of right $\kring$-modules.
\end{proof}

\begin{thm}
\label{thm:DG_split_R_C}
Under the above hypotheses on $A$, the DG algebra structure given by Theorem \ref{thm:DG_posit}  on the quadratic homogeneous algebra $(\grad A)\qd$  has differential $d$ determined by 
\begin{eqnarray*}
d|_R &=&\hat{q}  : R \rightarrow \hom_R (R \otimes_\kring W, R) \cong \hom_\kring (W, \kring) \otimes_\kring R \subset (\grad A)\qd\\
d |_{\hom_\kring (W, \kring)} &=&0,
\end{eqnarray*}
where $\hat{q}$ is the adjoint to $q$.
\end{thm}

\begin{proof}
It is clear that the ring $(\grad A)\qd$ is generated over  $R$ by the sub right $\kring$-module $\hom_\kring (W, \kring)$, hence it suffices to specify the restrictions $d|_R$ and $d|_{\hom_\kring (W, \kring)}$.

The identification of $d|_R$ follows directly from Theorem \ref{thm:DG_posit} equation (\ref{eqn:dR}), so it remains to consider $d|_{\hom_\kring (W, \kring)}$. This is described by equation (\ref{eqn:dphi}) in terms of the maps $p \circ s$ and $\tilde{q}\circ(s \otimes \id)$. We proceed to analyse these restricted to $\hom_\kring (W, \kring)$.

In the current setting, the product of $A$ gives $
(R \otimes_\kring W) \otimes_\kring (R \otimes_\kring W) \rightarrow F_2 A 
$, 
which fits into the commutative diagram in left $R$-modules:
\[
\xymatrix{
(R \otimes_\kring W) \otimes_\kring (R \otimes_\kring W) 
\ar[r]
\ar@{->>}[d]
&
F_2 A
\ar@{->>}[d]
\\
(R \otimes_\kring W) \otimes_R (R \otimes_\kring W) 
\ar[r]&
\grad_2 A,
}
\]
in which the left hand vertical map is the canonical surjection. As in Notation \ref{nota:I_A}, $I_A$ is the kernel of the lower horizontal map, so that $I_A \cong R \otimes_\kring I_C$. 

Define $\widehat{I_\kring}$ as the $R$-submodule of $(R \otimes_\kring W) \otimes_\kring (R \otimes_\kring W) $ given by the  pullback of $I_A$. Then, as in Section \ref{subsect:DG_generalities}, {\em mutatis mutandis}, the product of $A$ induces the map 
\[
p_\kring : \widehat{I_\kring} \rightarrow R \otimes_\kring W.
\]
Now, using the notation of Section \ref{subsect:DG_generalities}, one has the surjection $\widehat{I}\twoheadrightarrow \widehat{I_\kring}$ induced by passage from $\otimes$ to $\otimes_\kring$. By construction, the map $p : \widehat{I}  \rightarrow R \otimes_\kring W$ factors across this surjection and $p_\kring$. Hence, rather that use a section $s$ as in Remark \ref{rem:p_s}, it suffices to use a section of left $R$-modules $s_\kring : I_A \rightarrow \widehat{I_\kring}$.  Such a section exists, since $I_A$ is projective as a left $R$-module.

There is a convenient choice of $s_\kring$ given as follows. 
 One has the commutative diagram in left $R$-modules:
\[
\xymatrix{
(R \otimes _\kring W)\otimes _\kring (\kring   \otimes _\kring W)
\ar[r]
\ar[d]_\cong 
&
 (R \otimes _\kring W)\otimes _\kring (R  \otimes _\kring W)
 \ar@{->>}[d]
 \\
 R \otimes _\kring W \otimes _\kring W
 \ar[r]_(.4)\cong 
 &
 (R \otimes _\kring W)\otimes _R (R  \otimes _\kring W),
}
\]
in which the top horizontal map is induced by $\kring \rightarrow R$, which is a morphism of $\kring$-bimodules. This provides a section to the right hand vertical map.  Now, since   $I_A = R \otimes_\kring I_C \subset R \otimes _\kring W \otimes _\kring W \cong (R \otimes _\kring W)\otimes _R (R  \otimes _\kring W)$, on restricting to $I_A$ we obtain the desired section $s_\kring : I_A \rightarrow \widehat{I_\kring}$.

We also refine the map $\tilde{q}$ introduced in Notation \ref{nota:q_tilde}. Namely, this factors across 
\[
\widetilde{q_\kring} : 
 (R \otimes _\kring W)\otimes _\kring (R  \otimes _\kring W) 
 \otimes 
 \hom_R (R \otimes_\kring W, R) 
 \rightarrow 
 R,
\]
in which the first tensor product $\otimes$ for $\tilde{q}$ has been replaced by $\otimes_\kring$. (This is well-defined since the evaluation map $(R  \otimes _\kring W) 
 \otimes 
 \hom_R (R \otimes_\kring W, R) 
 \rightarrow 
 R$ is a morphism of left $R$-modules, hence of left $\kring$-modules, by restriction.)

By the above, we may replace $p \circ s$ and $\tilde{q}(s \otimes \id)$ by $p_\kring \circ s_\kring$ and $\widetilde{q_\kring} (s_\kring \otimes \id)$ respectively. By Theorem \ref{thm:DG_posit},  the vanishing of $d|_{\hom_\kring (W, \kring)}$ follows from the following two facts:
\begin{enumerate}
\item 
the composite map $p_\kring \circ s_\kring : R \otimes_\kring I_C \rightarrow R \otimes_\kring W$ is zero; 
\item 
the restriction of  $\widetilde{q_\kring} \circ (s_\kring \otimes \id)$ to $(R \otimes_\kring I_C) \otimes \hom_\kring ( C_1, \kring)$ is zero.
\end{enumerate}

The first assertion follows  from the construction of $s_\kring$  and the definition of $I_C$ as the kernel of the product  $W \otimes_\kring W \rightarrow C_2$.  For the second assertion, one checks that the composite identifies as 
\[
R \otimes _\kring I_C \otimes \hom_\kring (W, \kring) 
 \rightarrow 
R\otimes_\kring W \otimes _\kring W \otimes \hom_\kring (W, \kring) 
\rightarrow 
R\otimes_\kring W \otimes _\kring \kring
\rightarrow 
R
\]
in which the first map is induced by $I_C \subset W \otimes_\kring W$, the second by the evaluation 
$  W \otimes \hom_\kring (W, \kring) \rightarrow \kring$ and the final map is induced by the restriction of $q : (R\otimes_\kring W) \otimes R \rightarrow R$ to $(R\otimes_\kring W) \otimes \kring$. The latter vanishes, since $R \otimes_\kring W$ is a sub right $\kring$-module of the algebra $A$ (a consequence of the fact that the inclusion $C \hookrightarrow A$ is a morphism of $\kring$-bimodules).
\end{proof}

\begin{rem}
That the differential is a derivation follows from this property for $\qhat$, stated in Lemma \ref{lem:qhat_derivation}.
\end{rem}

\subsection{Units and augmentations}
\label{subsect:units_augs}

One can consider units and augmentations for the DG algebra given by Theorem \ref{thm:DG_posit}. 

\begin{rem}
\ 
\begin{enumerate}
\item
In the context of Theorem \ref{thm:DG_posit}, the algebra $(\grad A)\qd$ comes equipped with an inclusion $R \hookrightarrow (\grad A)\qd$. However, if $A$ is not homogeneous quadratic (i.e., $\qhat$ is non-zero),  this is not compatible with the differential. 
\item 
Restricting to $\kring \subset R$ gives $\kring  \hookrightarrow (\grad A)\qd$. Since $\hat{q}$ vanishes on $\kring \subset R$, this gives a morphism of DG algebras. 
\item 
The algebra $(\grad A)\qd$ is $\nat $-graded; in particular, the projection to the component of degree $0$ induces a surjective ring map $(\grad A)\qd \twoheadrightarrow R$ that is an augmentation in the non DG setting.
\end{enumerate}
\end{rem}

One has the following:

\begin{prop}
\label{prop:DG_aug_grad_A}
In the context of Theorem \ref{thm:DG_posit}, if $R$ admits an augmentation $\epsilon_R : R \rightarrow \kring$ (i.e., a morphism of algebras), then the composite 
$$
(\grad A)\qd \twoheadrightarrow R \stackrel{\epsilon_R}{\twoheadrightarrow}  \kring 
$$ 
is an augmentation as DG algebras of $(\grad A)\qd$ equipped with the unit $\kring \hookrightarrow (\grad A)\qd$. 
\end{prop}

\section{The dualizing complexes $K^\vee $ and $\vk$}
\label{sect:kcx}

The purpose of this section is to describe the structure of the nonhomogeneous dual Koszul complex of \cite[Chapter 6]{MR4398644} (using Positselski's terminology), first treating  the left augmented case. The  modifications for the right augmented case are outlined in Section \ref{subsect:vk}. 

\subsection{The left-augmented framework}
\label{subsect:framework}

Throughout  this section (apart from in Section \ref{subsect:vk}), $(A, F_\bullet A, \epsilon_A)$ is a left augmented filtered algebra over $R$ (the unit is $R = F_0 A \hookrightarrow A)$ such that 
\begin{enumerate}
\item 
the left augmentation induces the splitting as left $R$-modules $F_1 A \cong R \oplus V$ (defining $V$, which is isomorphic to $F_1A/F_0 A$, thus equipping it with an $R$-bimodule structure extending the left $R$-module structure); 
\item 
$\grad A$ is quadratic (so that $A$ is nonhomogeneous quadratic), with associated quadratic datum $(R;V,I_A)$; 
\item 
$\grad A$ is $3$-left finitely-generated projective over $R$. 
\end{enumerate}

This yields the dual quadratic  datum $(R; \hom_R (V, R) , \hom_R (\grad_2 A, R) )$ and the associated quadratic dual algebra $(\grad A)\qd$. The latter is equipped with the DG algebra structure given by Theorem \ref{thm:DG_posit}. 

By Proposition \ref{prop:dual_C_R}, the multiplication of $(\grad A)\qd$ gives a bijection
 of $C\qd \otimes R\op$-modules 
  $$
  C\qd \otimes_\kring R \stackrel{\cong}{\longrightarrow} (\grad A)\qd.
  $$

\begin{nota}
Denote by $q : V \otimes R \rightarrow R$ the component of the right $R$-module structure of $F_1 A$ restricted to $V$ and $\hat{q} : R \rightarrow \hom_R (V, R)$ the adjoint morphism.
\end{nota}

\subsection{Preliminaries}

Recall that, for $R$-modules  $ M, N$ such that $M$ is finitely-generated projective, the natural map 
\begin{eqnarray*}
\hom_R (M, R) \otimes_R N &\rightarrow & \hom_R (M, N) \\
\phi \otimes n & \mapsto & (m \mapsto \phi (m) n)
\end{eqnarray*}  
is an isomorphism of abelian groups. If $M$ and $N$ are $R$-bimodules and $\hom_R (M, R)\otimes_R N$ and $\hom_R (M, N)$ are equipped with the usual associated $R$-bimodule structures, then the bijection is one of $R$-bimodules. 

This applies to the $R$-bimodule $V$, which is assumed to be finitely-generated projective as a left $R$-module. 

\begin{nota}
\label{nota:e_e_prime}
Denote by
\begin{enumerate}
\item 
 $e_V : R \rightarrow \hom_R (V, R) \otimes_R V$ the map sending $1 \in R$ to the preimage of $\id_V$ under the above bijection;
 \item 
  $e'_V : R \rightarrow \hom_R (V, R) \otimes_R F_1 A$ the composite of $e_V$ with the map induced by the inclusion $V \subset F_1 A$ of left $R$-modules.
  \end{enumerate}
\end{nota}

The map $e_V$ is a morphism of $R$-bimodules (see \cite[Lemma 2.5 (ii)]{MR4398644}). However, 
the map $e'_V$ is not in general a morphism of $R$-bimodules; the defect is measured by $\hat{q}$:

\begin{lem}
\label{lem:qhat_e'}
For $r \in R$, one has the equalities
\begin{eqnarray*}
r e'_V(1) &=& e'_V (r) \\
e'_V (1) r &=& r e'_V(1) + \hat{q}(r)
\end{eqnarray*}
in $\hom_R (V, R) \otimes_R F_1 A \cong \hom_R (V, R) \otimes_R V  \ \oplus \  \hom_R (V,R)$. 
\end{lem}

\begin{proof}
This follows from the definition of $q$ and its adjoint. 
\end{proof}

We are interested in the case where there is an injective ring map $\kring \hookrightarrow R$,  $W$ is a $\kring$-bimodule that is finitely-generated projective as a left $\kring$-module, and 
 $V = R \otimes _\kring W$ as a $R\otimes \kring\op$-module. 
 
 Exploiting the respective left module structures, one has the commutative diagram
\[
\xymatrix{
\kring 
\ar[rr]^{1 \mapsto \id_W}
\ar@{^(->}[d]
&&
\hom_\kring (W, W) 
\ar[d]
&
\hom_\kring (W, \kring) \otimes_\kring W 
\ar[l]_\cong 
\ar[d]
\\
R
\ar[rr]_(.3){1 \mapsto \id_{R\otimes_\kring W}}
&&
\hom_R (R \otimes_\kring W, R\otimes_\kring W)
&
\hom_R (R \otimes_\kring W, R) \otimes_R (R \otimes_\kring W),
\ar[l]^(.55)\cong 
}
\]
in which the vertical maps are given  by applying the functor $R \otimes_\kring -$. Commutation of the left hand square is immediate; the commutativity of the right hand square is a simple verification.

This yields the following compatibility statement:

\begin{lem}
\label{lem:compat_e}
For $W$, $V= R\otimes_\kring W$ as above, the map $\hom_\kring (W, \kring) \otimes_\kring W  \rightarrow \hom_R (R \otimes_\kring W, R) \otimes_R (R \otimes_\kring W)$ induced by the functor $R \otimes_\kring -$ sends
$
e_W (1)$ to $e_V(1) = e_{R\otimes_\kring W} (1)$.
\end{lem}

\begin{rem}
\label{rem:e_prime_right_modules}
The above has an obvious counterpart in which the roles of left and right $R$-modules are switched. 
\end{rem}

The definition of the differential in the complexes considered below uses a form of inner multiplication. So as to clarify the construction, this is considered in the following Remark in a general context.

\begin{rem}
\label{rem:inner_multiplication}
Given a right (respectively left) $\kring$-module $M$ (resp. $N$) and a $\kring$-bimodule $Z$,  one  can form $M \otimes_\kring N$ and  $M \otimes_\kring Z \otimes_\kring N$. A morphism of $\kring$-bimodules 
$\kring \rightarrow Z$ is equivalent to an element $z \in Z$ such that $k z = zk$ for all $k \in \kring$. Given such a $z$, one has the natural map:
\[
M\otimes_\kring N \rightarrow M \otimes_\kring Z \otimes _\kring N
\]
which is given explicitly by $ m\otimes n \mapsto m \otimes z \otimes n$. 

Now, suppose that we have unital associative rings $R_1$, $R_2$, equipped with unital ring maps $\kring \rightarrow R_i$ ($i \in \{1, 2\}$) such that $M$ is a right $R_1$-module, $N$ a left $R_2$-module and  the above $\kring$-module structures are given by restriction.  Take $Z$ to be the $\kring$-bimodule $R_1 \otimes_\kring R_2$ and suppose that we have an element $z$ as above (i.e., corresponding to a $\kring$-bimodule map $\kring \rightarrow R_1 \otimes_\kring R_2$). We can therefore form the composite 
\[
M\otimes_\kring N \rightarrow M \otimes_\kring (R_1 \otimes_\kring R_2)  \otimes _\kring N
\rightarrow 
M \otimes_\kring N,
\]
where the second map is given by the respective module structures of $M$ and $N$. This can be viewed as forming the `inner multiplication' by $z$, written 
$
m\otimes n \mapsto m z n.
$

If $z$ is written symbolically in Sweedler-type notation as $z_1 \otimes z_2$ (i.e., suppressing the summation from the notation), then one can write this as 
$ 
m \otimes n \mapsto m z_1 \otimes z_2 n$, 
 which makes transparent that the image is an element of $M \otimes_\kring N$.
\end{rem}

\subsection{The dualizing  complex $K^\vee (A)$ in the left augmented case}
\label{subsect:Kvee_leftaug}

In \cite[Section 6.2]{MR4398644}, Positselski introduces his nonhomogeneous dual Koszul CDG module in the general setting that allows for non-zero curvature. This is reviewed here in the simpler left-augmented situation. 

\begin{rem}
Positselski indicates how his theory behaves in the left-augmented case in \cite[Examples 6.15(5)]{MR4398644} and \cite[Remark 6.16(2)]{MR4398644}.
\end{rem}

In this setting (assuming the hypotheses of Section \ref{subsect:framework}), $A$ is a filtered algebra, left augmented over $R$; the quadratic dual algebra has the structure of a DG algebra $((\grad A)\qd, d)$ by Theorem \ref{thm:DG_posit}. In particular, the underlying algebra is $\nat$-graded, which corresponds to the cohomological grading. 
 Recall that $V$ is isomorphic (as an $R$-bimodule) to $\grad_1 A$, which is finitely-generated projective as a left $R$-module, by hypothesis.

\begin{defn}
\label{defn:kcx}
The dualizing complex $K^\vee = K^\vee ( A) $ is the complex with underlying graded module:
\[
K^\vee (A)
:= 
(\grad A)\qd \otimes_R A
\]
and differential given for $x \in (\grad A)\qd $ and $a \in A$ by 
\[
 d_{K^\vee}( x \otimes a) = (dx) \otimes a + (-1)^{|x|} x e'_V(1) a.
\]
\end{defn} 

\begin{rem}
The differential $d$ on $(\grad A)\qd$ is not right $R$-linear in general, but brings into play $\hat{q}$.  
The fact that the differential $d_{K^\vee}$ is well-defined (with respect to the $\otimes_R$) is a consequence of Lemma \ref{lem:qhat_e'}. 
\end{rem}

\begin{prop}
\cite[Proposition 6.4]{MR4398644} 
$(K^\vee ( A), d_{K^\vee})$ is a complex of left $((\grad A)\qd , d)$- right $(A,0)$-bimodules. 
\end{prop}

\begin{rem}
As explained by Positselski, this is a generalization of the {\em homogeneous} dual Koszul complex (as treated in \cite[Chapter 2]{MR4398644}, for example). In the homogeneous case, $(\grad A)\qd$ has zero differential.
\end{rem}

\subsection{Specializing to the case $A \cong_{\mathrm{bimod}} R \otimes _\kring C$}
\label{subsect:kcx_R_C}
We now specialize to the framework considered in Section \ref{subsect:dgalg_application}, so that $C$ is the homogeneous quadratic algebra over $\kring$ specified by the quadratic datum $(\kring; W , I_C)$ and $A$ is isomorphic to $R \otimes_\kring C$ as a $R \otimes C\op$-module, also satisfying the hypotheses of that section with respect to the filtration induced from $C$.

Thus, we have the (left) quadratic duals $C\qd$ and  $(\grad A)\qd$; the latter is isomorphic to $C\qd \otimes_\kring R$ as a $C\qd \otimes R\op$-module, by Proposition \ref{prop:dual_C_R}.

\begin{prop}
\label{prop:kcx_R_C}
In this framework, the dualizing complex $K^\vee ( A) $ is isomorphic, as a complex of left $C\qd$-, right $A$-bimodules to 
$
C\qd \otimes_\kring A$,
 where the differential acts for $x \in C\qd$ and $a \in A$ by 
$
x \otimes a \mapsto (-1)^{|x|} x e'_W(1) a$.

With respect to this isomorphism, the $R$-module structure of $K^\vee (A)$ has structure  map given by the  composite:
\[
R \otimes_\kring (C\qd \otimes_\kring A) 
\rightarrow 
C\qd \otimes_\kring R \otimes_\kring A 
\rightarrow 
C\qd \otimes_\kring A,
\]
where the first map is the interchange map $R \otimes_\kring C\qd \rightarrow C\qd \otimes_\kring R $ induced by the product of $(\grad A)\qd$ and the second  is induced by the product $R \otimes_\kring A \rightarrow A$.
\end{prop}

\begin{proof}
This follows from an analysis of Definition \ref{defn:kcx} in this special case. The underlying graded $A$-module is given by  the bimodule isomorphism $(\grad A)\qd \cong C\qd \otimes_\kring R$ of Proposition \ref{prop:dual_C_R}. Theorem \ref{thm:DG_split_R_C} gives that the differential of $(\grad A)\qd$ acts trivially upon the subalgebra $C\qd$. From this one deduces the form of the differential. 
 The $R$-module structure map is identified by unwinding the definitions. 
\end{proof}

\begin{rem}
The differential only depends on the structure of the algebra $C\qd$ (more specifically, its structure as a right module over itself) and the left $C$-module structure of $A$.
\end{rem}

The argument used in the proof of Proposition \ref{prop:kcx_R_C} to give the full module structure can be abstracted (forgetting the differential and the grading as follows):

\begin{lem}
\label{lem:left_grad_A_module}
For an $R$-module $M$, the isomorphism $(\grad A)\qd \cong C\qd \otimes _\kring R$ of $C\qd \otimes R\op$-modules induces
\[
(\grad A)\qd \otimes_R M \cong C\qd \otimes_\kring M.
\] 
This is an isomorphism of $(\grad A)\qd$-modules, where $C\qd \otimes_\kring M$ is a $(\grad A)\qd$-module via 
 the extended $C\qd$-module structure and with $R$-module structure map 
 \[
 R \otimes_\kring (C\qd \otimes_\kring M) 
\rightarrow 
C\qd \otimes_\kring R \otimes_\kring M 
\rightarrow 
C\qd \otimes_\kring M,
 \] 
where the first map is the interchange map $R \otimes_\kring C\qd \rightarrow C\qd \otimes_\kring R $ induced by the product of $(\grad A)\qd$ and the second  is induced by the structure map $R \otimes_\kring M \rightarrow M$.
\end{lem}

There is an analogous result when considering the functor $\hom_R (A, -)$:

\begin{lem}
\label{lem:left_A_mod_hom}
For an $A$-module $M$, the  isomorphism $A \cong R \otimes _\kring C$ of $R \otimes C\op$-modules induces
\[
\hom_R (A, M) 
\cong 
\hom_\kring (C, M) .
\]
This 
is an isomorphism of $A$-modules, where $\hom_\kring (C, M)$ is given its canonical $C$-module structure and the $R$-module structure map $R \otimes _\kring \hom _\kring (C, M) \rightarrow \hom_\kring (C,M)$ is adjoint to 
\[
C \otimes_\kring R \otimes_\kring \hom_\kring (C,M) \rightarrow M
\]
given as the composite 
\[
C \otimes_\kring R \otimes_\kring \hom_\kring (C,M)
\rightarrow 
R \otimes_\kring C \otimes_\kring \hom_\kring (C,M)
\rightarrow 
R \otimes _\kring M
\rightarrow 
M, 
\]
where the first map is induced by the interchange map, the second by the evaluation map, and the last the $R$-module structure map of $M$.
\end{lem}

\subsection{The dualizing complex $\vk(A)$ in the right augmented case}
\label{subsect:vk}

We explain the construction in the right augmented case, using the appropriate analogue of the hypotheses of Section \ref{subsect:framework},  working over $C$. 

Suppose that  $(A, G_\bullet A, \varepsilon_A)$ is a right augmented filtered algebra over $C$ (so that the unit is $C = G_0 A \hookrightarrow A)$ such that 
\begin{enumerate}
\item 
the right augmentation  induces the splitting as right $C$-modules $G_1 A \cong C \oplus U$; 
\item 
$\rgrad A$ is quadratic, with associated quadratic datum $(C;U,J_A)$; 
\item 
$\rgrad A$ is $3$-right  finitely-generated projective over $C$. 
\end{enumerate}

This yields the (right) dual quadratic datum $(C; \hom_{C\op} (U, C) , \hom_{C\op} (\rgrad_2 A, C) )$ and the associated quadratic dual algebra $(\rgrad A)\qd$. The latter is equipped with a DG algebra structure, given by the  analogue of Theorem \ref{thm:DG_posit} for this framework.

We use the element ${}_Ue'(1)$, the counterpart in the right augmented case of $e'_V(1)$ of Notation \ref{nota:e_e_prime}.

\begin{defn}
\label{defn:right_kcx}
The (right) dualizing complex $\vk = \vk ( A) $ is the complex with underlying graded module:
\[
\vk (A)
:= 
A \otimes _C (\rgrad A)\qd 
\]
and differential given for $y \in (\rgrad A)\qd $ and $a \in A$ by 
$ 
 d_{\vk}( a \otimes y) = a \otimes dy + a \ {}_U e'(1) y.
$ 
\end{defn} 

Now, suppose that we are in the right-augmented analogue of the context of Section \ref{subsect:kcx_R_C}. In particular,  $A \cong R \otimes_\kring C$ as $R \otimes C\op$-modules and $R$ is a homogenous quadratic algebra over $\kring$ given by the quadratic datum $(\kring; Z, I_R)$, where the $\kring$-bimodule $Z$ is finitely-generated projective as a right $\kring$-module and $U \cong Z \otimes_\kring C$ as a right $C$-module. 

One has the (right) quadratic dual $R\qd$ of $R$ that is associated to the dual quadratic datum 
$$(\kring, \hom_{\kring\op} (Z, \kring) , \hom_{\kring\op} (R_2, \kring) ).$$

\begin{rem}
The counterpart of Proposition \ref{prop:dual_C_R} implies that the product of $(\rgrad A)\qd$ induces an isomorphism of left $C \otimes (R\qd)\op$ modules 
$
C \otimes_\kring R\qd
\stackrel{\cong}{\rightarrow}
 (\rgrad A)\qd$.
\end{rem}

\begin{nota}
\label{nota:efrak_prime}
Write $\e$ for the element of $Z \otimes_\kring \hom_{\kring \op} (Z, \kring)$ corresponding to the identify map $\id _Z$ 
 and $\e'$ for this considered in $A \otimes_\kring R\qd$.
\end{nota}

\begin{prop}
\label{prop:right_kcx_R_C}
The dualizing complex $\vk ( A) $ is isomorphic, as a complex of right $R\qd$-, left $A$-bimodules to 
$ 
A \otimes_\kring R\qd$,  
where the differential acts for $x \in C\qd$ and $a \in A$ by 
$
a \otimes y \mapsto a \e'y$.

The right $C$-module structure of $\vk ( A) $ has structure map 
\[
(A \otimes_\kring R\qd) \otimes_\kring C \rightarrow A \otimes_\kring C \otimes_\kring R\qd 
\rightarrow A\otimes_\kring R\qd,
\]
where the first map is induced by the product of $(\grad A)\qd$ and the second by the right $C$-module structure of $A$.
\end{prop}

\section{The adjunctions associated to $K^\vee(A)$ - the left augmented case}
\label{sect:adjoints}

The significance for us of the dualizing complex $K^\vee (A)$ introduced in Section \ref{sect:kcx}  is that it provides  adjunctions relating  DG $A$-modules to DG $(\grad A)\qd$-modules. This is the key ingredient in Positselski's relative Koszul duality theory \cite{MR4398644}. Positselski treats a much more general context, requiring the usage of curved structures;  in the left augmented case considered here, this is much simpler.

\subsection{Recollections}
Take $A$ and $B$ to be unital associative algebras and let $M$ be a $B \otimes A\op$-module. Then, for a left $A$-module $X$ and left $B$-module $Y$, one has the natural isomorphism:
\[
\hom_B(M \otimes_A X, Y) 
\cong 
\hom_A (X , \hom_B(M, Y))
\]
in which $M \otimes_A X$ is considered as a left $B$-module via the action on $M$ and $\hom_B(M, Y)$ as a left $A$-module via the (right) action of $A$ on $M$.  Thus $M \otimes_A -$ is left adjoint to $\hom_B (M, -) $ and one has the adjunction unit 
\begin{eqnarray*}
\label{eqn:M_adj_unit}
X & \rightarrow & \hom_B (M, M \otimes _A X) 
\\
x & \mapsto & (m \mapsto m \otimes x) 
\end{eqnarray*}
and the adjunction counit 
\begin{eqnarray*}
\label{eqn:M_adj_counit}
M \otimes_A \hom_B (M , Y) 
&\rightarrow & 
Y
\\
m \otimes \psi & \mapsto & \psi (m).
\end{eqnarray*}

\begin{exam}
Suppose give inclusions $R \hookrightarrow A $ and $R \hookrightarrow B$ of unital associative algebras and take $M:= B \otimes_R A$, equipped with the obvious left $B$ and right $A$-actions.  Then, for $X$ and $Y$ as above,  there is a natural isomorphism $M \otimes_A X \cong B \otimes_R X$ of $B$-modules and a natural isomorphism $\hom_B (M, Y) \cong \hom_R (A,Y)$ of left $A$-modules, using the restrictions of the left module structures of $X$ and $Y$ to $R$ respectively.

The adjunction unit identifies as 
\begin{eqnarray*}
X & \rightarrow & \hom_R (A, B \otimes_R X) \\
x & \mapsto & (a \mapsto 1 \otimes ax ) 
\end{eqnarray*}
and  the  counit as
\begin{eqnarray*}
B \otimes_R \hom_R (A, Y) & \rightarrow & Y \\
b \otimes \psi & \mapsto & b \psi (1).
\end{eqnarray*}
\end{exam}

These adjunctions generalize to the differential graded (DG) setting. Recall that the category of chain complexes (of abelian groups) admits a closed, symmetric monoidal structure. The tensor product of two chain complexes  $C$ and $D$ is defined as usual so that $(C \otimes D) _n = \bigoplus_{i+j =n} C_ i \otimes D_j$ and with differential given by $d (x \otimes y) = (d_C x) \otimes y + (-1)^{|x|}x \otimes d_D y$. The symmetry involves Koszul signs: $x \otimes y \mapsto (-1)^{|x| |y|} y \otimes x$.
 The internal hom is given by $\ihom (C, D) _n = \prod_i \hom (C_i , D_{i+n}) $ and with differential given by $(df)(x) = d_D (f (x)) - (-1)^{|f|} f (d_C x)$. 

A differential graded algebra $A$ is simply a unital monoid in this monoidal category (the symmetry does not intervene). The category of left (respectively right) $A$-modules is defined in the usual way.

\begin{nota}
\label{nota:dgmod}
For $A$ a DG algebra, write 
\begin{enumerate}
\item 
$A\dash\dgmod$ for the category of left DG modules over $A$.
\item 
$\dgmod \dash A$, for the category of right DG modules over $A$.
\end{enumerate}
\end{nota}

If $M$ is a right DG $A$-module and $N$ is a left DG $A$-module, the tensor product $M \otimes _A N$ is defined as the coequalizer of the canonical maps defined by the respective module structures:
\[
M \otimes A \otimes N
\rightrightarrows 
M \otimes N.
\]
This is a chain complex of abelian groups. 

Likewise, if $M$, $N$ are both DG left $A$-modules, the chain complex of abelian groups $\ihom_A (M,N)$ is defined as the 
equalizer of the structure maps 
\[
\ihom (M, N) \rightrightarrows \ihom (A \otimes M, N). 
\]
There is the corresponding construction of $\ihom_{A\op} (M, N)$ when $M$, $N$ are both DG right $A$-modules.

If $X$ is a DG left (respectively right) $A$-module, for any chain complex $Y$, $\ihom (X, Y)$ is naturally a DG right (resp. left) $A$-module, as in the non DG case. Using this structure, the DG counterparts of the above properties hold. 
 In particular, for DG algebras $A$, $B$ and a DG left $B$, right $A$ bimodule $M$, one has the adjunction:
 \[
 M \otimes_ A - \  : \   A \dash \dgmod \rightleftarrows B \dash \dgmod \ : \ihom _B(M, -). 
 \]

Similarly, for right DG modules:
\[
- \otimes_ B M \  : \  \dgmod\dash B \rightleftarrows  \dgmod\dash A \ : \ihom _{A\op}(M, -). 
 \]

\subsection{The adjunction   for left DG  modules}
We work in the context considered in  Section \ref{sect:kcx}, so that $ K^\vee (A)$ is  a left $(\grad A)\qd$, right $A$ bimodule (where $(\grad A)\qd$ is considered as a DG algebra).

The adjunction given in the previous subsection yields the following, in which flatness and projectivity refer to the underlying (non-DG) structures:

\begin{prop}
\label{prop:Kvee_adjunction}
The complex $K^\vee := K^\vee (A)$ yields the adjunction 
\[
K^\vee \otimes_A -  \ : \  
A\dash\dgmod 
\rightleftarrows 
(\grad A)\qd \dash \dgmod 
\ : \ 
\ihom_{(\grad A)\qd} (K^\vee, -) .
\]
\begin{enumerate}
\item 
If $K^\vee$ is flat as a right $A$-module, then $K^\vee \otimes_A -$ is exact. 
\item 
If $K^\vee$ is projective as a $(\grad A)\qd$-module, then $\ihom_{(\grad A)\qd} (K^\vee, -)$ is exact.
\end{enumerate}
For $X$ a DG  $A$-module, the unit of the adjunction gives the natural map
\[
X \rightarrow \ihom_{(\grad A)\qd } (K^\vee , K^\vee \otimes_A X)
\]
and, for $Y$ a DG $(\grad A)\qd$-module, the counit gives the natural map:
\[
K^\vee \otimes _A \ihom_{(\grad A)\qd} (K^\vee , Y)
\rightarrow Y.
\]
\end{prop}

The explicit description of the complex $K^\vee$ given in Section \ref{subsect:Kvee_leftaug} allows these functors to be identified below.  For concision, we use the following notation:

\begin{nota}
\label{nota:e'}
Write 
\begin{enumerate}
\item 
$e'_V (1) \in \hom_R (V,R) \otimes _R F_1 A$ (see Notation \ref{nota:e_e_prime}) in Sweedler-type notation as 
 $ 
e'_V (1)
=
e'_1 \otimes e'_2
$ 
(i.e., both leaving the summation implicit and omitting $V$);
\item 
$\otimes^{e'}$ and $\ihom^{e'}$  to indicate that the differential of the construction is twisted by a term involving $e':= e'_1 \otimes e_2'$. 
\end{enumerate}
\end{nota}

\begin{lem}
\label{lem:K_left_right_adjs}
\ 
\begin{enumerate}
\item 
For $X$ a DG $A$-module, there is a natural isomorphism of DG $(\grad A)\qd$-modules
\[
K^\vee \otimes_A X 
\cong 
(\grad A)\qd \otimes_R^{e'} X
\]
where the differential on $(\grad A)\qd \otimes_R X$ is given by 
\[
b \otimes x \  \mapsto \  (db) \otimes x + (-1)^{|b|} b e'_1 \otimes e'_2 x  + (-1)^{|b|} b \otimes d_X x .
\]
\item 
For $Y$ a DG $(\grad A)\qd$-module, there is a natural isomorphism of DG $A$-modules
\[
\ihom_{(\grad A)\qd} (K^\vee, Y) 
\cong 
\ihom_R^{e'} (A, Y),
\]
where the action on $\ihom_R^{e'} (A, Y)$ is induced by the right $A$-module structure of $A$.
The differential on $\ihom_R^{e'} (A, Y)$ is given by 
\[
(df) (a) = d (f(a)) - e'_1 f (e'_2 a).
\]
\end{enumerate} 
\end{lem}

\begin{proof}
This follows from the description of $K^\vee$. (Note that, for the differential on $\ihom_R^{e'} (A, Y)$, the sign $(-1)^{|f|}$ does not appear, since there is a second Koszul sign $(-1)^{|f|}$ arising from the interchange of $f$ and $e'_1$.)
\end{proof}

Using these natural isomorphisms, the adjunction unit and counit are described as follows:

\begin{prop}
\label{prop:(co)unit_Kvee_adjunction}
For the adjunction $K^\vee \otimes_A - \ \dashv \ \ihom _{(\grad A)\qd} (K^\vee , -)$:
\begin{enumerate}
\item 
for $X$ a DG $A$-module, the unit is the natural transformation:
\begin{eqnarray*}
X & \rightarrow & \ihom_R^{e'} (A, (\grad A)\qd \otimes_R^{e'} X ) 
\\
x & \mapsto & (a \mapsto (1 \otimes ax)); 
\end{eqnarray*}
\item 
for $Y$ a DG $(\grad A)\qd$-module, the counit is the natural transformation:
\begin{eqnarray*}
(\grad A)\qd \otimes_R^{e'} \ihom_R^{e'} (A,Y) 
&
\rightarrow & Y 
\\
b \otimes \psi &\mapsto & b \psi (1).
\end{eqnarray*}
\end{enumerate}
\end{prop}

\begin{exam}
\label{exam:(co)unit_R}
\ 
\begin{enumerate}
\item 
Since  $A$ is left augmented over $R$,  $R$ has a canonical left $A$-module structure.
Then, taking $X =R$ (considered as a complex concentrated in degree zero), the adjunction unit gives:
\[
R \rightarrow \ihom_R^{e'} (A, (\grad A)\qd). 
\]
\item 
Likewise, we can consider $R$ as a DG $(\grad A)\qd$-module, since $(\grad A)\qd$ is cohomologically $\nat$-graded and is $R$ in degree zero. Taking $Y=R$, the adjunction counit gives: 
\[
(\grad A)\qd \otimes _R^{e'} \ihom_R^{e'} (A, R) \rightarrow R.
\] 
\end{enumerate}
If $A$ is {\em homogeneous} quadratic, these can be described using the homogeneous Koszul complexes of \cite[Chapter 2]{MR4398644} (see Section \ref{subsect:recollect_homog_Koszul} below).
\end{exam}

\subsection{Specializing to the case $A \cong_{\mathrm{bimod}} R \otimes_\kring C$}
\label{subsect:laug_R_C}

We now specialize to the context of Section \ref{subsect:kcx_R_C}, where the product of $A$ induces an isomorphism of $R \otimes C\op$-modules $ R \otimes_\kring C \stackrel{\cong}{\rightarrow} A$.
 In this case, Lemma \ref{lem:K_left_right_adjs} becomes:

\begin{lem}
\label{lem:identify_adjoints_R_C}
\ 
\begin{enumerate}
\item 
For $X$ a DG $A$-module, there is a natural isomorphism of DG $(\grad A)\qd$-modules
\[
K^\vee \otimes_A X 
\cong 
C\qd \otimes_\kring^{e'} X
\]
where $C\qd \otimes_\kring^{e'} X$ has underlying object $C\qd \otimes_\kring X$ with 
 $R$-module structure given by Lemma \ref{lem:left_grad_A_module}. 

The differential on $C\qd \otimes_\kring ^{e'}X$ is given by 
\[
b \otimes x \  \mapsto \   (-1)^{|b|} b e'_1 \otimes e'_2 x  + (-1)^{|b|} b \otimes d_X x .
\]
\item 
For $Y$ a DG $(\grad A)\qd$-module, there is a natural isomorphism:
\[
\ihom_{(\grad A)\qd} (K^\vee, Y) 
\cong 
\ihom_\kring^{e'}  (C , Y)
\]
of DG $A$-modules, where the left $A$-module structure of $\ihom_\kring^{e'}  (C , Y)$ is given by transport along the isomorphism $\ihom_\kring^{e'}  (C , Y) \cong \ihom_R^{e'} (A, Y)$.  

The differential on $\ihom_\kring^{e'} (C, Y)$ is given by 
\[
(df) (c) = d (f(c)) - e'_1 f (e'_2 c).
\]
\end{enumerate} 
\end{lem}

Likewise, Proposition \ref{prop:(co)unit_Kvee_adjunction} specializes to:

\begin{prop}
\label{prop:(co)unit_Kvee_C}
For the adjunction $K^\vee \otimes - \ \dashv \ \hom (K^\vee , -)$, 
\begin{enumerate}
\item 
for $X$ a DG $A$-module, the unit is the natural transformation:
\begin{eqnarray*}
X & \rightarrow & \ihom_\kring^{e'} (C, C\qd \otimes_\kring^{e'} X ) 
\\
x & \mapsto & (c \mapsto (1 \otimes cx)); 
\end{eqnarray*}
\item 
for $Y$ a DG $(\grad A)\qd$-modules, the counit is the natural transformation:
\begin{eqnarray*}
C\qd \otimes_\kring^{e'} \ihom_\kring^{e'} (C, Y) 
&
\rightarrow & Y 
\\
b \otimes \psi &\mapsto & b \psi (1).
\end{eqnarray*}
\end{enumerate}
\end{prop}

\begin{rem}
\label{rem:adj_unit_counit_R_C_case}
In this  context:
\begin{enumerate}
\item
the adjunction unit, considered as a morphism of complexes of $\kring$-modules, only depends on the underlying left $C$-module structure of $X$; 
\item 
the adjunction counit only depends on the underlying left $C\qd$-module structure of $Y$.
\end{enumerate}
\end{rem}

\subsection{The adjunction   for DG right modules}

The general theory gives the following:

\begin{prop}
\label{prop:right_Kvee_adjunction}
The complex $K^\vee$ yields the adjunction 
\[
- \otimes_{(\grad A)\qd} K^\vee    \ : \  
\dgmod\dash  (\grad A)\qd
\rightleftarrows 
 \dgmod  \dash A
\ : \ 
\ihom_{A\op} (K^\vee, -) .
\]
\end{prop}

Specializing to the case where $A$ is isomorphic to $R \otimes _\kring C$ as $R \otimes C\op$-modules, one has: 

\begin{prop}
\ 
\begin{enumerate}
\item 
For $X$ a DG right $(\grad A)\qd$-module, there is an isomorphism of complexes 
\[
X \otimes_{(\grad A)\qd} K^\vee \cong X \otimes_\kring^{e'} C
\]
where the right hand side is equipped with the differential
$$
d( x \otimes c) = d_X x \otimes c + (-1)^{|x|} x e_1' \otimes e_2' c.
$$
\item
For $Y$ a DG right $A$-module, there is an isomorphism of complexes 
\[
\ihom_{A \op} (K^\vee, Y) \cong \ihom_{\kring\op}^{e'} (C\qd, Y)
\]
where the right hand side has differential given by 
$$
(df) (b) = d_Y (f(b)) - (-1)^{|f|} f (b e_1')e_2'.
$$
\end{enumerate}
\end{prop}

\begin{proof}
This follows as for Lemma \ref{lem:K_left_right_adjs}. (Note that, in the $\ihom$ case, since we are working with right modules, there is no second Koszul sign to take into account, thus the sign $(-1)^{|f|}$ persists.)
\end{proof}

\section{The  adjunctions induced by $\vk (A)$ - the right augmented case}
\label{sect:adjoints_bis}

The purpose of this short section is to present the analogue of the adjunctions of Section \ref{sect:adjoints} in the right augmented case, using the complex $\vk(A)$ of Definition \ref{defn:right_kcx}.

\begin{prop}
\label{prop:vk_adjunctions}
The complex $\vk:= \vk (A)$ yields adjunctions
\[
\vk \otimes_{(\rgrad A)\qd} -  \ : \ 
(\rgrad A)\qd\dash \dgmod 
\rightleftarrows
A \dash \dgmod 
\ : \ 
\ihom_A (\vk, -)
\]
and
\[
- \otimes_A \vk \ : \ 
\dgmod\dash A 
\rightleftarrows
\dgmod \dash (\rgrad A)\qd
\ : \ 
\ihom _{((\rgrad A)\qd)\op}
(\vk, -). 
\]
\end{prop}

Now suppose that we are in the right augmented analogue of Section  \ref{subsect:kcx_R_C}, as in Section \ref{subsect:vk}. In particular, $A \cong R \otimes_\kring C$ as $R \otimes C\op$-modules, with $R$ the homogeneous quadratic algebra over $\kring$ given by the quadratic datum $(\kring ; Z, I_R)$. As in Notation \ref{nota:efrak_prime}, we have  $\e \in Z \otimes_\kring \hom_{\kring \op} (Z, \kring)$ and the associated element $\e'$,   denoted by $\e_1' \otimes \e_2'$ in Sweedler-style notation.

The underlying object of $(\rgrad A)\qd$ is $C \otimes_\kring R\qd$ as a $C \otimes (R\qd)\op$-module.
 In this situation, forgetting the differential, we can write 
\begin{eqnarray*}
\vk (A) & \cong & R \otimes_\kring (\rgrad A)\qd \\
& \cong & A \otimes_ \kring R\qd
\end{eqnarray*}
as right $(\rgrad A)\qd$-module and left $A$-module respectively.

Analogously to Notation \ref{nota:e_e_prime}, we use the following:

\begin{nota}
\label{nota:efrak_twist}
The twisting of the differential for $\ihom$ and $\otimes$  using $\e'$ is indicated by $\ihom^{\e'}$ and $\otimes^{\e'}$ respectively. 
\end{nota}

In this context, Proposition \ref{prop:vk_adjunctions} becomes:

\begin{prop}
\label{prop:vk_adjunctions_R_C}
\ 
\begin{enumerate}
\item 
There is an adjunction 
\[
R\otimes_\kring^{\e'} -  \ : \ 
(\rgrad A)\qd\dash \dgmod 
\rightleftarrows
A \dash \dgmod 
\ : \ 
\ihom_\kring ^{\e'} (R\qd, -),
\]
where:
\begin{enumerate}
\item 
for $M$ a DG $(\rgrad A)\qd$-modules, $R \otimes_\kring^{\e'} M$ has differential $d (r \otimes m) = r \e' m + r \otimes d_M (m)$;
\item 
for $N$ a DG $A$-module, $\ihom_\kring ^{\e'} (R\qd, N)$ has differential given by $(dg) (\rho) = d_N (g (\rho)) - \e_1' g( \e_2' n)$.
\end{enumerate}
\item
There is an adjunction 
\[
- \otimes_\kring^{\e'} R\qd \ : \ 
\dgmod\dash A 
\rightleftarrows
\dgmod \dash (\rgrad A)\qd
\ : \ 
\ihom _{\kring\op}^{\e'}
(R, -) ,
\]
where
\begin{enumerate}
\item 
 for $Y$ a DG right $A$-module, $Y \otimes_\kring ^{\e'} R\qd$ has differential $d (y \otimes \rho) = d_Y y \otimes \rho + (-1)^{|y|} y \e' \rho$; 
\item
for $X$ a DG right $(\rgrad A)\qd$-module, $\ihom _{\kring\op}^{\e'}(R,X)$ has differential 
$(df ) (r) = d_X (f(r)) - (-1)^{|f|} (r \e'_1) \e_2'$.
\end{enumerate}
\end{enumerate}
\end{prop}

\section{The nonhomogeneous biquadratic situation}
\label{sect:biquad}

We suppose given a commutative diagram of morphisms of unital associative rings 
\[
\xymatrix{
\kring 
\ar[r]
\ar[d]
&
C 
\ar[d]
\\
R 
\ar[r]
&
A
}
\]
such that the product of $A$ induces an isomorphism $R \otimes_\kring C \stackrel{\cong}{\rightarrow} A $ of $R\otimes C\op$-modules. In addition, we require that both $R$ and $C$ be homogeneous quadratic satisfying the requisite hypotheses ensuring that 
 $A$ acquires two different nonhomogeneous quadratic structures, induced by those of $R$ and $C$ respectively. 
 
\subsection{The framework}
\label{subsect:biquad_framework}

We suppose the following:

\begin{hyp}
\label{hyp:C_R}
Suppose that 
\begin{enumerate}
\item 
$C$ is $3$-left finitely-generated projective over $\kring$ and is a homogeneous quadratic algebra defined by the quadratic datum $(\kring; W , I_C)$; 
\item 
$R$ is $3$-right finitely-generated projective over $\kring$ and is a homogeneous quadratic algebra defined by the quadratic datum $(\kring; Z, I_R)$. 
\end{enumerate}
\end{hyp}

Since $C$ is, in particular, $2$-left finitely-generated projective over $\kring$, one has the (left) quadratic dual $C\qd$ defined by the quadratic datum $(\kring; \hom_\kring (W, \kring) , \hom_\kring (C_2, \kring))$; thus  $C\qd$ is $2$-right finitely-generated projective over $\kring$.
 Likewise, one has the (right) quadratic dual $R\qd$ that is $2$-left finitely-generated projective over $\kring$, defined by the quadratic datum $(\kring; \hom_{\kring\op} (Z, \kring), \hom_{\kring \op} (R_2, \kring))$.  

\begin{rem}
\label{rem:Z_W_not_proj}
The hypothesis that the product of $A$ induces an isomorphism $R \otimes_\kring C \stackrel{\cong}{\rightarrow} A$ of $R \otimes C\op$-modules implies that $A$ is generated as an algebra over $\kring$ by the $\kring$-bimodule $Z \oplus W$. 
 However, in general, $Z \oplus W$ is neither projective as a left $\kring$-module nor projective as a right $\kring$-module.
\end{rem}

By the hypothesis, both $R$ and $C$ are $\nat$-graded algebras over $\kring$; in particular there are isomorphisms of $\kring$-bimodules $C \cong \bigoplus_i C_i$ and $R \cong \bigoplus_j R_j$ such that $C_0 = R_0 = \kring$ and $C_1= W$, $R_1= Z$.
 One has the following identification:
 
 \begin{lem}
 \label{lem:A_bigraded_as_bimodule}
 There is an isomorphism of $\kring$-bimodules $
 A \cong \bigoplus_{i,j} R_i \otimes_\kring C_j.
 $
 \end{lem}

\begin{nota}
Write 
\begin{enumerate}
\item 
$F_\bullet A$ for the filtration of $A$ by $R$-modules given by $F_n A := \bigoplus_{i=0}^n R \otimes_\kring C_i$; 
\item 
$G_\bullet A$ for the filtration of $A$ by right $C$-modules given by $G_m A:= \bigoplus_{j=0}^m R_j \otimes_\kring C$. 
\end{enumerate}
\end{nota}

As in Section \ref{subsect:non-homog}, one may ask whether $F_\bullet A$ (respectively $G_\bullet A$) makes $A$ into a filtered algebra; to answer this, one appeals to Proposition \ref{prop:filt_hyp} and its `opposite' counterpart. Since $A$ is generated by the $\kring$-bimodule $Z \oplus W$, the respective filtration criteria can be restricted to considering the behaviour of the map induced by the product: 
\begin{eqnarray}
\label{eqn:prod_W_T}
W \otimes_\kring Z \rightarrow A \cong_{\mathrm{bimod}} R \otimes_\kring C.
\end{eqnarray}

\begin{prop}
\label{prop:biquad}
The filtrations $F_\bullet A$ and $G_\bullet A$ both define filtered algebra structures on $A$ if and only if the product map 
 (\ref{eqn:prod_W_T}) maps to 
 \[
 \kring
  \ \oplus \ 
(Z \oplus W) 
 \ \oplus \ 
 Z \otimes_\kring W
 \subset 
 A
 \]
 using the inclusion furnished by Lemma \ref{lem:A_bigraded_as_bimodule}.
\end{prop}

\begin{proof}
Proposition \ref{prop:filt_hyp} implies that $F_\bullet A$ makes $A$ into a filtered algebra if and only if the product factors as 
\[
W \otimes_\kring Z \rightarrow R \  \oplus \  R\otimes_\kring W \subset R\otimes_\kring C.
\]
Likewise, the right module analogue of Proposition \ref{prop:filt_hyp} implies that $G_\bullet A$ makes $A$ into a filtered algebra if and only if the product factors as 
\[
W \otimes_\kring Z \rightarrow C \  \oplus \  Z \otimes_\kring C \subset R\otimes_\kring C.
\]

Putting these conditions together gives the result, by appealing to Lemma \ref{lem:A_bigraded_as_bimodule} to identify the intersection of the $\kring$-bimodules $R \  \oplus \  R\otimes_\kring W $ and $C \  \oplus \  Z\otimes_\kring C $ in $R \otimes_\kring C$.
\end{proof}

Henceforth we shall assume that the criterion of Proposition \ref{prop:biquad} holds, so that the product restricts to 
\begin{eqnarray}
\label{eqn:pro_map_W_T}
W \otimes_\kring Z \rightarrow \kring \ \oplus \ 
(Z \oplus W) 
 \ \oplus \ 
 Z \otimes_\kring W.
\end{eqnarray}

\begin{rem}
\label{rem:bi_nonhomog}
Using this criterion, one could consider $A$ as a nonhomogeneous quadratic algebra over $\kring$, generated by the $\kring$-bimodule $Z \oplus W$ and subject to the homogeneous quadratic relations from $I_C$ and $I_R$ together with the nonhomogeneous quadratic relation given by (\ref{eqn:pro_map_W_T}).

However, due to the failure of projectivity in general (cf. Remark \ref{rem:Z_W_not_proj}), it is better simply to consider $A$ as being both
 nonhomogeneous quadratic over $R$ with respect to the filtration $F_\bullet A$ and 
nonhomogeneous quadratic over $C$ with respect to the filtration $G_\bullet A$.
\end{rem}

\subsection{Associated structures}

Under the hypotheses of Section \ref{subsect:biquad_framework}, supposing that the criterion of Proposition \ref{prop:biquad} holds, we have the following:

\begin{enumerate}
\item 
The DG algebra $(\grad A)\qd$ (left quadratic duality with respect to $R$), with underlying right $R$-module $C\qd \otimes _\kring R$. 
\item 
The DG algebra $(\rgrad A)\qd$ (right quadratic duality  with respect to $C$), with underlying left $C$-module $C \otimes_\kring R\qd$. 
\item 
The complex $K^\vee:= K^\vee (A)$, with underlying bimodule $(\grad A)\qd \otimes_R A$ and differential twisted using $e'$. 
\item 
The complex $\vk := \vk (A)$, with underlying bimodule $A \otimes_C (\rgrad A)\qd$ and differential  twisted using $\e'$.
\end{enumerate}

This gives the following adjunctions for left modules:
\[
\xymatrix{
(\rgrad A)\qd \dash \dgmod
\ar@<.5ex>[rr]^(.55){\vk \otimes_{(\rgrad A)\qd}-}
&\quad &
A \dash \dgmod 
\ar@<.5ex>[rr]^{K^\vee \otimes_A -}
\ar@<.5ex>[ll]^(.45){\ihom _A (\vk, -)}
&\quad &
(\grad A)\qd \dash \dgmod
\ar@<.5ex>[ll]^{\ihom_{(\grad A)\qd} (K^\vee, -)}
}
\]
and hence the composite adjunction
\[
\xymatrix{
(\rgrad A)\qd \dash \dgmod
\ar@<.5ex>[rrr]^{K^\vee \otimes_A \vk \otimes_{(\rgrad A)\qd}-}
&\quad&\quad &
(\grad A)\qd \dash \dgmod,
\ar@<.5ex>[lll]^{\ihom_{(\grad A)\qd} (K^\vee \otimes _A \vk, -)}
}
\]
in which $K^\vee \otimes_A \vk$ is a  DG left $(\grad A)\qd$, right $(\rgrad A)\qd$ bimodule.

One has the analogous adjunctions for right modules:
\[
\xymatrix{
 \dgmod \dash (\grad A)\qd 
\ar@<.5ex>[rr]^(.55){ - \otimes_{(\grad A)\qd} K^\vee}
&
\quad
&
A \dash \dgmod 
\ar@<.5ex>[rr]^{- \otimes_A \vk}
\ar@<.5ex>[ll]^(.45){\ihom _{A\op} (K^\vee, -)}
&
\quad
&
\dgmod \dash (\rgrad A)\qd
\ar@<.5ex>[ll]^{\ihom_{((\rgrad A)\qd)\op} (\vk, -)}
}
\]
and hence the composite adjunction:
\[
\xymatrix{
 \dgmod \dash (\grad A)\qd
\ar@<.5ex>[rrr]^{- \otimes_{(\grad A)\qd} K^\vee \otimes_A \vk }
&\quad &\quad &
 \dgmod \dash (\rgrad A)\qd.
\ar@<.5ex>[lll]^{\ihom_{((\rgrad A)\qd)\op} (K^\vee \otimes _A \vk, -)}
}
\]

This shows the importance of the DG left $(\grad A)\qd$, right $(\rgrad A)\qd$ bimodule  $K^\vee \otimes_A \vk$.

\begin{lem}
There is an isomorphism of DG bimodules
\[
K^\vee \otimes_A \vk
\cong 
(\grad A)\qd \otimes _R^{e'} A \otimes_C ^{\e '} (\rgrad A)\qd, 
\]
where the superscripts $e'$ and $\e'$ indicate that the differential is twisted.
\end{lem}

\begin{rem}
The underlying bimodule of $K^\vee \otimes_A \vk$ is isomorphic to 
$$
(\grad A)\qd \otimes_\kring^{e', \e'} (\rgrad A)\qd,
$$
where the differential is twisted using both $e'$ and $\e'$.
\end{rem}

\begin{lem}
\label{lem:underlying_kv_vk}
The underlying $C\qd\otimes (R\qd)\op$-module of $K^\vee \otimes_A \vk$ is isomorphic to 
\[
C\qd \otimes_\kring A \otimes_\kring R\qd.
\]
 With respect to this isomorphism, the differential on $K^\vee \otimes_A \vk$ is given for $\gamma \in C\qd$, $a \in A$, and $\rho \in R\qd$  by 
 \[
 d (\gamma \otimes a \otimes \rho) 
 = 
 (-1)^{|\gamma|} \gamma e_1' \otimes e_2' a \otimes \rho 
 + 
 (-1)^{|\gamma|} \gamma \otimes a \e_1' \otimes \e_2' \rho.
 \]
 which can be written as $d (\gamma \otimes a \otimes \rho) 
 = 
 (-1)^{|\gamma|}  \big( \gamma e' a \otimes \rho + \gamma \otimes a \e' \rho\big)$.
 \end{lem}
 
\begin{rem}
\label{rem:kv_vk}
The full structure of $K^\vee \otimes_A \vk$ is recovered by specifying its left $R$-module structure and right $C$-module structure. The left action is  given by Lemma \ref{lem:left_A_mod_hom} and the right action by its counterpart for right $(\rgrad A)\qd$-modules.
\end{rem}

\subsection{The composite adjunction for left DG modules}
We focus  upon the case of left DG modules; the right DG module case can be analysed similarly. 
 
Using Lemma \ref{lem:underlying_kv_vk}, one can identify the left adjoint $K^\vee \otimes_A \vk \otimes_{(\rgrad A)\qd} -$.

 \begin{prop}
 \label{prop:left_kv_vk_adjoint}
 For a DG $(\rgrad A)\qd$-module $M$, there is a natural isomorphism of DG $C\qd$-modules 
 \[
 K^\vee \otimes_A \vk \otimes_{(\rgrad A)\qd} M 
 \cong 
 C\qd \otimes_\kring^{e'} R \otimes_\kring^{\e'} M.
 \]
 This is a natural isomorphism of DG $(\grad A)\qd$-modules when the right hand side is equipped with the $R$-module structure furnished by Lemma \ref{lem:left_A_mod_hom}.
 
The differential on the right hand side is given, for $\gamma \in C\qd$, $r \in R$, and $m \in M$, by
\[
d( \gamma \otimes r \otimes m) 
= 
(-1)^{|\gamma|} \Big( 
\gamma e_1' \otimes e_2'\star (r \otimes m)
\ + \ 
\gamma \otimes r \e_1' \otimes \e_2' m 
\ + \ 
\gamma \otimes r \otimes d_M m
\Big),
\]
where the product $e_2'\star (r \otimes m)$ is given by the $C$-module structure of $R \otimes _\kring M$ obtained from the isomorphism $R \otimes _\kring M \cong A \otimes_C M$. 
 \end{prop}
 
 \begin{proof}
 Using the identification in Lemma \ref{lem:left_A_mod_hom}, one has the natural isomorphism
 \[
 K^\vee \otimes_A \vk \otimes_{(\rgrad A)\qd} M 
 \cong 
 C\qd \otimes_\kring A \otimes_C M,
 \]
 with differential deduced from that in the Lemma, and similarly for the $(\grad A)\qd$-module structure.
 
 Using the isomorphism of $R \otimes C\op$-modules $A \cong R \otimes_\kring C$, one deduces the result.
 \end{proof}
 
 \begin{exam}
 \label{exam:kv_vk_k}
 Using the augmentation $C \rightarrow \kring$ (cf. Section \ref{subsect:units_augs}), one can consider $\kring$ as a DG $(\rgrad A)\qd$-module with trivial differential. Hence one can apply the functor $K^\vee \otimes_A \vk \otimes_{(\rgrad A)\qd} -$. There is an isomorphism of DG $(\grad A)\qd$-modules:
 \[
 K^\vee \otimes_A \vk \otimes_{(\rgrad A)\qd} \kring
\cong 
(\grad A)\qd, 
 \]
 where the right hand side is the DG algebra $(\grad A)\qd$ (considered as a DG-module).   In particular, the differential is determined by $e'$ (i.e., $\e'$ does not intervene).
 \end{exam}
 
Likewise, one can determine the right adjoint functor $\ihom_{(\grad A)\qd} (K^\vee \otimes_A \vk , -)$. The identification of the structure relies on the natural isomorphisms of the underlying graded objects (for $N$ a DG $(\grad A)\qd$-module)
\[
\ihom_{(\grad A)\qd} (K^\vee \otimes_A \vk , N)
\cong 
\ihom_R (A \otimes_C (\rgrad A)\qd, N)
\cong 
\ihom _\kring ((\rgrad A)\qd , N)
\]
which are isomorphisms of $(\rgrad A)\qd$-modules. This serves both to identify the module structure and to show how to evaluate an morphism $f \in \ihom _\kring ((\rgrad A)\qd , N)$ on an element of $A \otimes_C (\rgrad A)\qd$.

\begin{prop}
\label{prop:right_kv_vk}
For $N$ a DG $(\grad A)\qd$-module, there is a natural isomorphism:
\[
\ihom_{(\grad A)\qd} (K^\vee \otimes_A \vk , N)
\cong 
\ihom_\kring^{e'} (C \otimes_\kring^{\e'} R\qd, N)
\]
where the right hand side is equipped with the twisted differential:
\[
(df) (c \otimes \rho) 
= 
d_N (f (c \otimes \rho))
- 
(-1)^{|f|} f (c \e' \rho) 
- 
e_1' f (e_2'c \otimes \rho), 
\]
where $c \e' \rho \in  A \otimes_\kring R\qd \cong A \otimes_C (\rgrad A)\qd$.
\end{prop}

\begin{proof}
As for Proposition \ref{prop:left_kv_vk_adjoint}, this follows from Lemma \ref{lem:underlying_kv_vk}.
\end{proof}

\begin{exam}
\label{exam:hom_kv_vk_k}
One has the following counterpart of Example \ref{exam:kv_vk_k}. Using the augmentation $R \rightarrow \kring$, one can consider $\kring$ as a DG $(\grad A)\qd$-module. Then one has the isomorphism
\[
\ihom_{(\grad A)\qd} (K^\vee \otimes_A \vk , \kring)
\cong 
\ihom_\kring ((\rgrad A)\qd, \kring),
\]
where the right hand side is a DG  $(\rgrad A)\qd$-module.
\end{exam}

The adjunction unit gives the natural transformation for a DG $(\rgrad A)\qd$-module $M$
\begin{eqnarray*}
M &\rightarrow & \ihom_\kring^{e', \e'} ((\rgrad A)\qd, (\grad A)\qd \otimes_\kring M ),
\end{eqnarray*}
where the differential is twisted using both $e'$ and $\e'$ (as indicated by the superscripts).

Similarly, for a DG $(\grad A)\qd$-module $N$, one has the adjunction counit, which has underlying map
\begin{eqnarray*}
(\grad A) \qd \otimes_\kring
\ihom_\kring ((\rgrad A)\qd, N)
\rightarrow N.
\end{eqnarray*}
Here the domain must be equipped with the appropriate twisted differential, using both $e'$ and $\e'$ .

\begin{exam}
\label{exam:unit_counit_kk}
\ 
\begin{enumerate}
\item 
Taking $M=\kring$, as in Example \ref{exam:kv_vk_k}, the adjunction unit gives 
\[
\kring 
\rightarrow 
\ihom_\kring ^{e', \e'} ((\rgrad A)\qd, (\grad A)\qd  ).
\]
The differential of the codomain is given  by Proposition \ref{prop:right_kv_vk}.
\item 
Taking $N= \kring$ as in Example \ref{exam:hom_kv_vk_k}, the adjunction counit gives:
\[
(\grad A) \qd \otimes_\kring
\ihom_\kring ((\rgrad A)\qd, \kring)
\rightarrow \kring,
\]
in which the domain must be equipped with the appropriate twisted differential, using both $e'$ and $\e'$ .
\end{enumerate}

These can be considered as being a relative nonhomogeneous counterpart of classical Koszul complex constructions (cf. the brief review of relative nonhomogeneous Koszul duality in Section \ref{sect:koszul}). In particular, one can ask for criteria that guarantee that the unit (respectively the counit) is a weak equivalence. The natural candidate is to require that both $C$ and $R$ are (homogeneous) Koszul over $\kring$, in the appropriate sense (see Section \ref{subsect:apply_Koszul_property} for further indications).
\end{exam}

\section{The Koszul property for nonhomogeneous quadratic rings}
\label{sect:koszul}

This section outlines the Koszul theory of nonhomogeneous quadratic rings, following the presentation given by Positselski, and indicates how this behaves in the special cases of interest here. Section \ref{subsect:apply_Koszul_property} outlines how this can be applied, providing motivation for the concrete results in Part \ref{part:two}, rather than precise mathematical statements.

\subsection{Recollections for the homogeneous case}
\label{subsect:recollect_homog_Koszul}
The homogeneous case is analogous to classical Koszul duality, as presented in \cite{MR2177131}. Here we follows Positselski's presentation  \cite{MR4398644}.

\begin{defn}
\label{defn:koszul}
Suppose that $B$ is a $\nat$-graded ring with $B_0 =R$. Then,
\begin{enumerate}
\item 
(\cite[Definition 2.29]{MR4398644})
$B$ is {\em left flat Koszul} if $B_n$ is a flat left $R$-module for each $n \in \nat$ and $\mathrm{Tor}^B_{i,j} (R,R)=0$ for all $i \neq j$. 
\item 
(\cite[Definition 2.34]{MR4398644} and \cite[Theorem 2.35]{MR4398644}) 
$B$ is {\em left finitely projective Koszul} if  it is both left  finitely-generated projective  and  left flat Koszul.
\end{enumerate}
\end{defn}

\begin{rem}
\ 
\begin{enumerate}
\item 
Equivalent conditions for $B$ being left flat Koszul are given in \cite[Theorem 2.30]{MR4398644}.
\item 
Equivalent criteria for $B$ to be left finitely projective Koszul are given in \cite[Theorem 2.35]{MR4398644}, as outlined below.
\item 
\cite[Proposition 2.37]{MR4398644} gives a criterion using $\mathrm{Ext}$.
\item 
A left finitely projective Koszul ring $B$ is quadratic over $R$. 
\end{enumerate}
\end{rem}

We note that quadratic duality behaves as expected on Koszul rings:

\begin{prop}
\label{prop:left_right_Koszul}
\cite[Corollary 2.38]{MR4398644}
The quadratic duality functor $A \mapsto A\qd$ from the category of $2$-left finitely-generated projective homogeneous quadratic rings over $R$
 to the category of $2$-right finitely-generated projective homogeneous quadratic rings over $R$ restricts to an anti-equivalence between 
 the category of left finitely projective Koszul  rings over $R$ and right finitely projective Koszul rings over $R$.
\end{prop}

Positselski gives a characterization of Koszulity using his first and second Koszul complexes (see \cite[Chapter 2]{MR4398644}), as indicated below. For this, suppose that:
\begin{enumerate}
\item 
$A$ and $B$ are $\nat$-graded rings with $A_0=R=B_0$, and $B_1$ is finitely-generated projective as a right $R$-module; 
\item  
$\tau$ is a $R$-bimodule morphism $\hom_{R\op} (B_1, R) \rightarrow A_1$ (equivalently an element  $e \in B_1 \otimes_R A_1$ such that $re=er$ for all $r \in R$).
\end{enumerate}

\begin{exam}
\label{exam:tau_dual}
These hypotheses hold when $A$ is quadratic over $R$ and $2$-left finitely-generated projective, and $B:= A\qd$, its (left) quadratic dual, equipped with $\tau$ provided by the isomorphism $\hom_{R\op} (B_1, R) \cong A_1$.  
\end{exam}

Positselski's first Koszul complex $K^\tau (B,A) $ has the form 
\[ 
\ldots 
\rightarrow 
A \otimes_R \hom_{R\op} (B_2, R) 
\rightarrow 
A \otimes_R \hom_{R\op} (B_1, R)
\rightarrow 
A
\rightarrow 
0 
\]
with  $\partial^\tau$ defined using $\tau$. In the case of Example \ref{exam:tau_dual}, $(\partial^\tau)^2=0$ and the latter will be assumed henceforth. Thus $K^\tau (B,A) $ is a complex of left $A$-modules; it comes equipped with a natural morphism of $R$-bimodules $R \rightarrow K^\tau (B, A)$, where $R$ is placed in homological degree zero. Written globally, one has $K^\tau (B,A) \cong A \otimes_R \hom_{R\op} (B, R)$, equipped with the differential $\partial^\tau$, with homological grading given by the $\nat$-grading of $B$. 
 
Positselski's second Koszul complex ${}^\tau K(B,A)$  has the form 
 \[ 
\ldots 
\rightarrow 
 \hom_{R\op} (B_2, R) \otimes _R A
\rightarrow 
\hom_{R\op} (B_1, R)\otimes_R A
\rightarrow 
A
\rightarrow 
0;
\]
it is a complex of right $A$-modules.  Written globally, one has ${}^\tau K(B,A)  \cong \hom_{R\op} (B, R)\otimes_R A$; this is equipped with a natural morphism of $R$-bimodules ${}^\tau K(B,A) \rightarrow R$.

Then one has the following weak form of \cite[Theorem 2.35]{MR4398644}:

\begin{prop}
\label{prop:Koszul_equivalent_unit/counit}
An $\nat$-graded ring $A$ with $A_0=R$ is left finitely projective Koszul if and only if one of the following equivalent conditions holds:
\begin{enumerate}
\item 
$A$ is left finitely-generated projective and quadratic and $K^\tau (B, A)$ is exact in homological degrees $\geq 1$; 
\item 
$A$ is left finitely-generated projective and quadratic and ${}^\tau K(B, A)$ is exact in homological degrees $\geq 1$. 
\end{enumerate} 
\end{prop}

\subsection{Koszul duality in the nonhomogeneous case}

We follow \cite[Section 4.6]{MR4398644}, but restricting to nonhomogeneous quadratic rings (rather than the more general weak nonhomogeneous quadratic rings in {\em loc. cit.}). In particular,  $A$ is a unital, associative ring equipped with a filtration $F_\bullet A$ such that the associated graded $\grad A$ is $\nat$-graded and  quadratic over $\grad_0 A$. (In  \cite[Section 4.6]{MR4398644}, the filtered algebra $A$ is denoted $\tilde{A}$, the notation $A$ being reserved there for the associated graded.)

\begin{defn}
\label{defn:nonhomog_Koszul}
The nonhomogeneous quadratic ring $(A, F_\bullet A)$ is left finitely projective Koszul if $\grad A$ is left finitely projective Koszul (in the homogeneous sense).  
\end{defn}

\begin{rem}
\label{rem:P's_triangulated_equivalences}
Positselski presents his general theory without supposing that $A$ is left augmented. In the general case, rather than the DG algebra of Section \ref{sect:dgalg}, one has a CDG algebra $(B,d,h)$, where $h$ is the curvature. The adjunctions of Sections \ref{sect:adjoints} and \ref{sect:adjoints_bis} generalize to the curved setting, with the fundamental difference that one is no longer working with DG objects, so the associated (co)derived categories cannot be defined simply by inverting quasi-isomorphisms. Instead, Positselski localizes with respect to appropriate `acyclic objects'.

If $A$ is left finitely projective Koszul, Positselski exhibits triangulated equivalences in this setting between appropriate  (co)derived categories. This involves a number of technical considerations in the general case which we shall not require for the application here, so we only indicate the results, without giving precise statements:

\begin{enumerate}
\item 
If $A$ is left finitely projective Koszul, in \cite[Theorem 6.14]{MR4398644} Positselski exhibits a triangulated equivalence between (co)derived categories associated respectively to DG right $A$-modules and right CDG $B$-(co)modules.  
\item 
If $A$ is left finitely projective Koszul, in \cite[Theorem 7.11]{MR4398644} he exhibits triangulated equivalence between the (semicontra)derived category of DG left $A$-modules and the (contra)derived category of left CDG $B$-(contra)modules.  
\end{enumerate} 
The interested reader should consult  \cite{MR4398644} for the definitions and precise statements.
\end{rem}

\begin{rem}
\label{rem:left-augmented}
In the left augmented setting, one can stay within the DG realm, studying the adjunctions  of Sections \ref{sect:adjoints} and \ref{sect:adjoints_bis}. Then the derived categories can be constructed from the relevant categories of DG modules by inverting quasi-isomorphisms. Moreover, in the applications in Part \ref{part:two}, the various adjoint functors preserve quasi-isomorphisms, so pass directly to functors between the corresponding `derived categories'.

In the Koszul case, in the applications in Part \ref{part:two}, one can impose finiteness restrictions on the DG modules considered so as to obtain triangulated equivalences analogous to those of  \cite[Theorems 6.14 and 7.11]{MR4398644}. 
\end{rem}

\subsection{The case $A \cong_{\mathrm{bimod}}R \otimes_\kring C$ with $C$ Koszul}
\label{subsect:Koszul_semibase_change}
Here we specialize to  the context studied in Section \ref{sect:quad}, where $ R \otimes_\kring C \cong A$ as $R \otimes C\op$-modules.

Take  $C$ to be quadratic over $\kring$ associated to the quadratic datum $(\kring; W, I_C)$ 
 and suppose that the filtration $F_nA := R \otimes_\kring F_n C$ exhibits $A$ as a filtered algebra (cf. Proposition \ref{prop:filt_hyp}), so that $\grad A$ is a graded algebra over $R \cong F_0 A$. By Proposition \ref{prop:A_quadratic}, $\grad A$ is a quadratic algebra over $R$, associated to the quadratic datum  $(R; R \otimes_\kring W, R\otimes _\kring I_C)$.

 If $C$ is left finitely projective Koszul over $\kring$, it is natural to ask whether  the associated graded algebra $\grad A$ is also left finitely projective Koszul (over $R$). This is the case, under a flatness hypothesis on $R$:

\begin{thm}
\label{thm:koszul}
Suppose that $C$ is left finitely projective Koszul over $\kring$ and that $(F_nA \mid n \in \nat)$ exhibits $A$ as a filtered algebra. If $R$ is flat as a {\em right} $\kring$-module, then $A$ is left finitely projective Koszul over $R$. 
\end{thm}

\begin{proof}
The hypothesis that  $C$ is left finitely-generated projective over $\kring$ implies that $\grad A$ is left finitely-generated projective over $R$. Thus it remains to show that it is left flat Koszul; this is checked by using the $\mathrm{Tor}$ criterion of Definition \ref{defn:koszul}, as follows.

By hypothesis, $C_+$ is a projective left $\kring$-module, hence is flat. It follows (see \cite[Section 2.2]{MR4398644} for example), that $\mathrm{Tor}^C (\kring, \kring)$ can be calculated as the homology of the reduced relative bar complex:
\[
\ldots 
\rightarrow 
C_+^{\otimes_\kring  n} 
\rightarrow 
C_+^{\otimes_\kring n-1}
\rightarrow 
\ldots
\rightarrow 
C_+ \otimes_\kring C_+ 
\rightarrow 
C_+ 
\rightarrow 
0.  
\]
The left flat Koszul hypothesis implies that the homology of this complex in homological degree $i$ is concentrated in grading $i$.

Similarly, $\mathrm{Tor}^{\grad A} (R,R)$ is calculated as the homology of the complex:
\[
\ldots 
\rightarrow 
(R \otimes_\kring C_+)^{\otimes_R  n} 
\rightarrow 
(R \otimes_\kring C_+)^{\otimes_R n-1}
\rightarrow 
\ldots
\rightarrow 
(R\otimes_\kring C_+)\otimes_R (R \otimes_\kring C_+) 
\rightarrow 
R \otimes_\kring C_+ 
\rightarrow 
0,   
\]
using the identification $\grad_+ A \cong R \otimes_\kring C_+$ as left $R$-modules.

One checks that this is isomorphic to the complex 
\[
\ldots 
\rightarrow 
R \otimes_\kring C_+^{\otimes_\kring  n} 
\rightarrow 
R \otimes_\kring C_+^{\otimes_\kring n-1}
\rightarrow 
\ldots
\rightarrow 
R \otimes_\kring C_+ \otimes_\kring C_+ 
\rightarrow 
R\otimes_\kring C_+ 
\rightarrow 
0
\]
obtained by applying the functor $R \otimes_\kring -$ to the relative reduced bar complex for $C$. 

By hypothesis, $R$ is flat as a right $\kring$-module, hence one has $\mathrm{Tor}^{\grad A} (R, R) \cong R \otimes_\kring \mathrm{Tor}^C (\kring, \kring)$. It follows from the hypothesis upon $C$ that the groups $\mathrm{Tor}^{\grad A} (R,R)$ satisfy the Koszul condition, as required.  
\end{proof}

\begin{rem}
\label{rem:rgrad_A_Koszul}
\ 
\begin{enumerate}
\item
There is the  `opposite' counterpart of this result, under the following hypotheses:
\begin{enumerate}
\item 
 $R$ is right finitely projective Koszul over $\kring$;
\item 
the associated filtration   $(G_nA \mid n \in \nat)$ exhibits $A$ as a filtered algebra; 
\item 
$C$ is flat as a left $\kring$-module,
\end{enumerate}
Then $\rgrad A$ is right finitely projective Koszul over $C$.
\item 
These results will  be applied to the associated quadratic duals, for which Proposition \ref{prop:left_right_Koszul} applies. Namely, under the respective hypotheses, one deduces that $(\grad A)\qd$ is right finitely projective Koszul over $R$ and $(\rgrad A)\qd$ is left finitely projective Koszul over $C$. 
\end{enumerate}
\end{rem}

\begin{rem}
In the applications in Part \ref{part:two}, the appropriate relative Koszul properties will be established directly, without applying Theorem \ref{thm:koszul} or its counterpart outlined in Remark \ref{rem:rgrad_A_Koszul}. Nevertheless, these results explain {\em why} one should expect to obtain such results. 
\end{rem}

\subsection{Applications of the Koszul property}
\label{subsect:apply_Koszul_property}
We continue working in the context of  Section \ref{subsect:Koszul_semibase_change}  and indicate how the results of Section \ref{subsect:Koszul_semibase_change} can be applied.
\bigskip

\begin{enumerate}
\item 
{\bf The left augmented case} (Section \ref{sect:adjoints}, specialized as in Section \ref{subsect:laug_R_C}).

\noindent
Suppose that:
\begin{enumerate}
\item 
$C$ is left finitely projective Koszul over $\kring$;
\item
$R$ is flat as a right $\kring$-module,
\end{enumerate}
then $A$ is  left finitely projective Koszul over $R$ by Theorem \ref{thm:koszul}. 

Positselski's triangulated equivalences  relate the homotopy categories of DG $A$-modules and of DG $(\grad A)\qd$-modules (subject to appropriate restrictions). 
\bigskip 
\item 
{\bf The right augmented case} (Section \ref{sect:adjoints_bis}). 

\noindent
Suppose that: 
\begin{enumerate}
\item 
$R$ is right finitely projective Koszul over $\kring$;
\item
$C$ is flat as a left $\kring$-module,
\end{enumerate}
then $A$ is  right finitely projective Koszul over $C$ by the `opposite' of Theorem \ref{thm:koszul}. 

In this case, Positselski's triangulated equivalences  relate the homotopy categories of DG $A$-modules and of DG $(\rgrad A)\qd$-modules (again subject to appropriate restrictions). 
\bigskip
\item 
{\bf The nonhomogeneous biquadratic situation} (Section \ref{sect:biquad}). 

\noindent
Here, we suppose that:
\begin{enumerate}
\item 
$R$ is right finitely projective Koszul over $\kring$;
\item 
$C$ is left finitely projective Koszul over $\kring$. 
\end{enumerate}

Then both of the above Koszul results apply. Thus, Positselski's triangulated equivalences should relate the homotopy categories of DG $(\grad A)\qd$-modules and of DG $(\rgrad A)\qd$-modules (still subject to appropriate restrictions). 

This is a Koszul-type duality relating the DG algebras $(\grad A)\qd$ and $(\rgrad A)\qd$ (cf. Example \ref{exam:unit_counit_kk}).
\end{enumerate}

\newpage
\part{Application to PROPs}
\label{part:two}

\section{Background on $\kk \fb$-modules and bimodules}
\label{sect:bimodules}

This section reviews the basic theory of modules and bimodules over the category $\fb$ of finite sets and bijections. 
Most of this material is well-known; it is included in part so as to fix notation. The following Section \ref{sect:finiteness_properties} introduces  finiteness and boundedness properties of bimodules that are used in the applications. 

Throughout, $\kk$ is a commutative ring with unit, specializing to the case of a field when vector space dimension is employed as a finiteness condition.

\subsection{Generalities}

For $\calc$ an essentially small category and $\kk$ a commutative ring with unit, $\kk \calc$ denotes its $\kk$-linearization; this is considered as a ring with several objects. A $\kk \calc$-module is equivalent to a functor from $\calc$ to $\kmod$, the category of $\kk$-modules; morphisms are natural transformations.  
 The category of $\kk \calc$-modules is denoted $\kk \calc \dash\modules$.  This has a tensor product defined by $\otimes_\kk$ applied objectwise, defining a symmetric monoidal structure on $\kk \calc\dash\modules$ with unit the constant functor with value $\kk$.

Recall that, for $\mathscr{A}$ an essentially small $\kk$-linear category, one can form $\otimes_\mathscr{A}$,  
the  `multi-object ring' generalization of the tensor product $\otimes_R$ over a  unital  associative ring $R$. 
 Taking $\mathscr{A} = \kk \calc$, this gives the tensor product (denote $\otimes_\calc$ for typographical clarity):
\begin{eqnarray*}
\otimes_\calc \cn \kk \calc\op \dash \modules  \times \kk \calc \dash \modules & \rightarrow & \kmod.
\end{eqnarray*}

For a $\kk \calc$-module $Y$ and a $\kk$-module $M$, $\hom_\kk (Y, M)$ is naturally a $\kk \calc\op$-module. Then, for $X$ a $\kk \calc\op$-module, one has the natural isomorphism:
\[
\hom_\kk (X \otimes_\calc Y, M) 
\cong 
\hom_{\kk \calc\op} (X, \hom_\kk (Y, M)).
\]
Thus,  for an essentially small category $\cald$, if $Y$ is a $\kk (\calc \times \cald\op)$-module (i.e., a bimodule with left $\kk \calc$-action and right $\kk\cald$-action), then $X \otimes_\calc Y$ is a $\kk \cald\op$-module. Moreover, if  $M$ is a $\kk \cald \op$-module,  the above isomorphism restricts to 
\[
\hom_{\kk\cald\op} (X \otimes_\calc Y, M) 
\cong 
\hom_{\kk \calc\op} (X, \hom_{\kk \cald\op} (Y, M)), 
\]
where $\hom_{\kk \cald\op} (Y, M)$ is a sub $\kk \calc\op$-module of $\hom_\kk (Y, M)$.

Suppose now that both $X$, $Y$ are $\kk (\calc \times \calc\op)$-modules (aka. $\kk \calc$-bimodules). Then one has the  $\kk \calc$-bimodules 
$\hom_{\kk \calc} (X, Y)$ and $
\hom_{\kk \calc \op} (X, Y)$. In general, these two bimodules are not isomorphic.

\subsection{$\kk\fb$-modules, bimodules and more}

The category of finite sets and  bijections is denoted $\fb$ and that of finite sets and injections $\finj$, so that there is an inclusion $\fb \hookrightarrow \finj$ as a wide subcategory. These are both essentially small, with skeleton $\{ \n \mid n \in \nat\}$, where $\n = \{1, \ldots, n\}$ (so that $\mathbf{0} = \emptyset$). The category $\fb$ is a groupoid, so that there is an isomorphism of categories $\fb \cong \fb\op$ induced by  passage to the inverse. Hence $\kk \fb\dash\modules$ is isomorphic to $\kk \fb\op\dash\modules$. 

\begin{exam}
\label{exam:sym_seq}
A $\kk\fb$-module $N$ is equivalent to a symmetric sequence, i.e., $\{ N(t) \mid  t \in \nat \}$, where  $N(t)$ is a $\kk\sym_t$-module; morphisms are sequences of equivariant maps. Similarly for $\kk \fb\op$, with variance adjusted accordingly. 
 Then, for $M$ a $\kk \fb\op$-module, $M \otimes _\fb N$ identifies as 
$ 
\bigoplus_{t \in \nat}
M(t) \otimes_{\sym_t} N(t).
$
\end{exam}

The category of $\kk\fb$-modules is equipped with the Day convolution product $\ofb$; for two $\kk \fb$-modules $N_1, N_2$, the convolution product $N_1 \ofb N_2$ evaluated on a finite set $S$ is 
\[
(N_1 \ofb N_2) (S) = \bigoplus_{U\amalg V =S} N_1 (U) \otimes_\kk N_2 (V).
\]
This defines a symmetric monoidal structure $(\kk\fb\dash\modules, \ofb , \kk_\mathbf{0} )$, where $\kk_\mathbf{0}$ denotes the module supported on $\mathbf{0}$ with value $\kk$.  Likewise, one has $\obf$ on $\kk \fb\op\dash\modules$, which corresponds to $\ofb$ under the isomorphism of categories $\kk\fb\dash\modules \cong \kk \fb\op\dash\modules$. 

\begin{nota}
\label{nota:triv_sgn_dagger}
Denote by 
\begin{enumerate}
\item 
$\triv$ (respectively $\sgn$) the $\kk\fb$-module with $\triv (\n) = \triv_n$, the trivial representation of $\kk \sym_n$ (resp. $\sgn(n) = \sgn_n$, the sign representation);
\item 
$(-)^\dagger$ the involution of $\kk\fb\dash\modules$ defined by the objectwise tensor product $(M)^\dagger := \sgn \otimes_\kk M$.  
\end{enumerate}
One has the corresponding structures in $\kk \fb\op \dash\modules$. Where necessary to distinguish such structures, a superscript will be used, for instance writing $\triv^\fb$ for the $\kk \fb$-module above.
\end{nota}

The functor $(-)^\dagger$ is a  monoidal self-equivalence of $(\kk \fb\dash\modules, \ofb, \kk_\mathbf{0})$ but not a {\em symmetric} monoidal self-equivalence. It induces an equivalence of symmetric monoidal categories when one introduces a  `twisted' symmetry for $\ofb$, as explained in \cite{2012arXiv1209.5122S} and \cite{MR3430359}. 

\begin{rem}
\label{rem:triv_sgn_dagger}
The $\kk \fb$-module $\triv$ is the unit for the objectwise tensor product: for any $\kk \fb$-module $M$, one has the canonical equivalences $\triv \otimes_\kk M \cong M \cong M \otimes_\kk \triv$. This gives the isomorphism $\sgn \cong \triv^\dagger$ and hence $\triv \cong \sgn^\dagger$.
\end{rem}

\begin{exam}
\label{exam:triv_sym_monoids}
Both $\triv$ and $\sgn$ are unital monoids in $(\kk\fb\dash\modules, \ofb , \kk_\mathbf{0} )$, and $\triv$ is commutative (whereas $\sgn$ is `anticommutative', corresponding to the fact that $(-)^\dagger$ is not a symmetric monoidal self-equivalence, but introduces signs). 

One can consider the category of (left) $\triv$-modules (respectively $\sgn$-modules) in $\kk\fb\dash\modules$. A $\triv$-module is a $\kk \fb$-module $M$ equipped with a structure map $\triv \ofb M \rightarrow M$ which satisfies the unit and associativity constraints; $\sgn$-modules are defined similarly. The involution $(-)^\dagger$ yields an equivalence of categories between $\triv$-modules and $\sgn$-modules.

The category of left $\triv$-modules is equivalent to that of right $\triv$-modules using the symmetry of $\ofb$. (The analogous statement holds for $\sgn$-modules;  in this case, one must use the `twisted' symmetry of $\ofb$, since $\sgn$ is only commutative with respect to this.)
\end{exam}

Recall the following well-known (cf. \cite{2012arXiv1209.5122S}, for example) and  fundamental fact:

\begin{prop}
\label{prop:finj-mod_triv-mod}
The category of $\kk\finj$-modules is equivalent to the category of $\triv$-modules in $\kk\fb\dash\modules$. 
\end{prop}

We also consider $\kk \fb$-bimodules. By convention, we take  this to be  the category $\kk (\fb\op \times \fb)\dash\modules$, so that, for a $\kk \fb$-bimodule $X$, one has the  $\kk (\sym_a\op \times \sym_b)$-module $X (\mathbf{a}, \mathbf{b})$. (This convention is chosen to agree with the notation for morphisms in the $\kk$-linear categories considered here, such as $\kk \finj$.)

\begin{defn}
\label{defn:Z-grade_fb-bimodules}
For $M$ a $\kk \fb$-bimodule and $n \in \zed$, let $M\dg{n}$ be the sub $\kk \fb$-bimodule given by 
\[
M\dg{n}(\mathbf{a}, \mathbf{b}) = 
\left\{
\begin{array}{ll}
M(\mathbf{a}, \mathbf{b}) & b-a = n \\
0 & \mbox{otherwise.}
\end{array}
\right.
\]
This defines a natural $\zed$-grading of $M$, so that $M \cong \bigoplus_{n \in \zed} M\dg{n}$. 
\end{defn}

\begin{exam}
\label{exam:bimodunit}
One has the $\kk\fb$-bimodule $\bmu$, with 
\[
\bmu(\mathbf{a}, \mathbf{b}) 
= 
\left\{ 
\begin{array}{ll}
\kk \sym_a & a=b \\
0 & \mbox{otherwise,}
\end{array}
\right.
\]
where the bimodule structure is given by the left and right regular actions. Clearly $\bmu$ is concentrated in $\zed$-degree zero; it is the bimodule underlying the $\kk$-linear category $\kk \fb$.

Below we use the fact that, for $a \in \nat$, $\bmu(\mathbf{a}, \mathbf{a}) \cong \kk \sym_a$ has a $\kk$-module basis  $\{ [\sigma] \mid \sigma \in \sym_a\}$.
\end{exam}

The tensor product $\otimes_\fb$ induces a functor 
\[
\otimes_\fb 
: 
\kk(\fb\op \times \fb)\dash\modules \times \kk (\fb\op \times \fb)\dash\modules 
\rightarrow 
\kk (\fb\op \times \fb)\dash\modules.
\]
This defines a monoidal structure $( \kk (\fb\op \times \fb)\dash\modules, \otimes_\fb, \bmu)$ (not symmetric). 
 
\begin{exam}
\label{exam:kk_finj}
The $\kk$-linear category $\kk \finj$ has underlying $\kk \fb$-bimodule given by restriction along $\fb \hookrightarrow \finj$. Moreover, the composition in $\kk \finj$ makes this into a unital monoid in $\kk \fb$-bimodules, with structure map 
\[
\kk \finj \otimes_\fb \kk \finj \rightarrow \kk \finj 
\]
and unit $\bmu \rightarrow \kk \finj$ induced by the inclusion $\fb \hookrightarrow \finj$.  There is also the  augmentation $\kk \finj \rightarrow \bmu$; at the level of $\kk \fb$-bimodules, this is the unique morphism that is a retract of the unit. 

The $\zed$-grading of Definition \ref{defn:Z-grade_fb-bimodules} specializes to 
a $\nat$-grading of $\kk \finj$ so that
\[
\kk \finj \dg{n} (\mathbf{a}, \mathbf{b}) = 
\left\{ 
\begin{array}{ll}
\kk \finj (\mathbf{a}, \mathbf{b}) & \mbox{$b-a = n$,   $n \in \nat$} \\
0 & \mbox{otherwise.}
\end{array}
\right.
\]
By definition, $\kk \finj\dg{0}$ identifies with $\bmu \subset \kk\finj$. 

For $m,n \in \nat$, the composition restricts to a surjection
\begin{eqnarray}
\label{eqn:graded_prod_kkfinj}
\kk \finj\dg{n} \otimes_\fb \kk \finj \dg{m}\twoheadrightarrow \kk \finj\dg{m+n},
\end{eqnarray}
so that $\kk \finj$ is  a  $\nat$-graded monoid in $\kk \fb$-bimodules.
\end{exam}

\begin{nota}
\label{nota:op_ddag_sharp}
Denote by 
\begin{enumerate}
\item 
$(-)\op$ the involution of $\kk(\fb\op \times \fb)\dash\modules$ defined by $X\op (\mathbf{a}, \mathbf{b}) := X (\mathbf{b}, \mathbf{a})$ (with variance adjusted appropriately); 
\item 
$\sgn\boxtimes \sgn$ the $\kk \fb$-bimodule  given by 
$
(\sgn \boxtimes \sgn) (\mathbf{a}, \mathbf{b} ) 
= 
\sgn_a \boxtimes \sgn_b
$ 
considered as a $\kk (\sym_a \op \times \sym_b)$-module;
\item 
$(-)^\ddag$ the involution of $\kk(\fb\op \times \fb)\dash\modules$   given by  $(\sgn \boxtimes \sgn) \otimes_\kk -$, so that 
\[
X^\ddag (\mathbf{a}, \mathbf{b} ) \cong (\sgn_a \boxtimes \sgn_b) \otimes X (\mathbf{a}, \mathbf{b}) 
\]
with diagonal action of $\sym_a\op \times \sym_b$.
\item 
$(-)^\sharp$ the duality functor given by $X^ \sharp(\mathbf{a}, \mathbf{b}) = \hom_{\kk} (X (\mathbf{b}, \mathbf{a}), \kk)$.
\end{enumerate}
\end{nota}

\begin{rem}
\label{rem:ddag}
The functor $(-)^\ddag$ is the bimodule analogue of the functor $(-)^\dagger$ of Notation \ref{nota:triv_sgn_dagger}.
\end{rem}

The following relates the possible duality functors working with $\kk \fb$-bimodules:

\begin{lem}
\label{lem:duality}
For  a $\kk \fb$-bimodule $X$, there are natural isomorphisms of  $\kk \fb$-bimodules 
$$
\hom_{\kk \fb} (X, \bmu)
\cong 
X^\sharp 
\cong 
\hom_{\kk \fb\op} (X, \bmu).
$$ 
\end{lem}

\begin{proof}
Consider the first case (the second is treated similarly); by definition, for any $\kk \fb$-bimodule $Y$, $\hom _{\kk \fb} (X,Y)$ evaluated on $(\mathbf{a}, \mathbf{b})$ is given by  
\[
\prod_{t \in \nat} \hom_{\kk \sym_t} (X (\mathbf{b}, \mathbf{t}) , Y (\mathbf{a}, \mathbf{t}) ).
\]
Taking $Y = \bmu$, the only non-zero contribution is from the term $t=a$, which contributes $\hom_{\kk \sym_a} (X (\mathbf{b}, \mathbf{a}) , \kk \sym_a)$. The latter is naturally isomorphic to $\hom_\kk (X (\mathbf{b}, \mathbf{a}) , \kk)$ (using that coinduction and induction coincide for finite groups), thus identifying with $X^\sharp (\mathbf{a}, \mathbf{b})$. 

One checks  that, in both cases, the isomorphisms respect the bimodule structures.
\end{proof}

The following states standard properties of the functors introduced in Notation \ref{nota:op_ddag_sharp}, using the notation for the canonical $\kk$-module basis for $\bmu$ given in Example \ref{exam:bimodunit}:

\begin{prop}
\label{prop:basic_properties_ddag_op}
\ 
\begin{enumerate}
\item 
There is a natural isomorphism $\bmu\op \cong \bmu$ given by $[\sigma]\mapsto [\sigma^{-1}]$.
\item 
There is a natural isomorphism $\bmu^\ddag \cong \bmu$ given by $[\sigma]\mapsto \sgn(\sigma) [\sigma]$. 
\item 
For $\kk \fb$-modules $X$, $Y$, there is a natural isomorphism $(X \otimes_\fb Y)\op \cong Y \op \otimes_\fb X \op$. 
\item 
\label{item:ddag_monoidal}
The functor $(-)^\ddag$ is  monoidal; in particular,  for $\kk \fb$-bimodules $X$, $Y$, there is a natural isomorphism $(X \otimes_\fb Y)^\ddag \cong X^\ddag \otimes_\fb Y ^\ddag$. 
\item 
For a $\kk \fb$-bimodule $X$, there is a natural isomorphism $(X \op) ^\ddag \cong (X ^\ddag) \op$.
\end{enumerate}
\end{prop}

\begin{exam}
\label{exam:finj_ddag}
One can form the $\kk\fb$-bimodules $\kk \finj^\ddag$ and $(\kk \finj^\ddag)\op$. These both define augmented, unital monoids in $\kk\fb$-bimodules.  Hence $\kk \finj^\ddag$ defines a $\kk$-linear category and  $(\kk \finj^\ddag)\op$  its opposite category. 
\end{exam}

\section{Finiteness properties for modules and bimodules}
\label{sect:finiteness_properties}

This section is a continuation of Section \ref{sect:bimodules}, introducing various finiteness and boundedness conditions that are employed in the proofs of the main results.

 Throughout this section, $\kk$ is taken to be a field.

\subsection{Non-negativity and non-positivity}
\label{subsect:non-neg_non-pos}

\begin{defn}
\label{defn:finite-type}
A functor $F$ from an essentially small category $\calc$ to $\kk\dash\modules$ is of finite type if, for all objects $X$ of $\calc$, $\dim _\kk F(X)< \infty$. If $F$ takes values in graded $\kk$-vector spaces, then one uses the total dimension (i.e., the dimension of the underlying ungraded space).
\end{defn}

This applies, in particular, when considering $\kk \fb$-modules, $\kk \fb\op$-modules, and $\kk \fb$-bimodules, since these categories are respectively functor categories on $\fb$, $\fb\op$, and $\fb\op \times \fb$.

The following is clear:

\begin{lem}
\label{lem:preserve_ft}
A bimodule $B$ is of finite type if and only if  $B\op$ is of finite type. Likewise for $(-)^\ddag$ or $(-)^\sharp$ in place of $(-)\op$.
\end{lem}

\begin{rem}
The property of being of finite type is not always preserved under the constructions used here.  For example,  $\triv^{\fb\op} \otimes_\fb \triv^\fb$ has infinite dimension. 
\end{rem}

The following property is useful for considering conservation of finite typeness:

\begin{defn}
\label{defn:bimod_non_neg}
A bimodule $M$ is 
\begin{enumerate}
\item 
non-negative if $M\dg{n}=0$ for all $n<0$; 
\item 
non-positive if $M\dg{n}=0$ for all $n>0$.
\end{enumerate}
Thus, $M$ is non-negative if $M(\mathbf{a}, \mathbf{b})=0$ for $a>b$.
\end{defn}

\begin{rem}
\ 
\begin{enumerate}
\item 
The choice of sign in defining the $\zed$-grading in terms of $(-)\dg{n}$ (see Definition \ref{defn:Z-grade_fb-bimodules}) is somewhat arbitrary. Hence the {\em positive} and {\em negative} in Definition \ref{defn:bimod_non_neg} should be considered as being on an equal footing.
\item 
A bimodule $M$ is both non-negative and non-positive if and only if $M = M\dg{0}$. For example, $\bmu$ satisfies this condition and is of finite type.
\end{enumerate}
\end{rem}

The following is clear:

\begin{lem}
\label{lem:non_neg_op_sharp}
For $M$ a $\kk \fb$-bimodule, the following are equivalent:
\begin{enumerate}
\item 
$M$ is non-negative; 
\item 
$M^\ddag$ is non-negative; 
\item 
$M\op$ is non-positive; 
\item 
$M^\sharp$ is non-positive.
\end{enumerate}
The corresponding statements hold with the r\^oles of {\em negative} and {\em positive} switched.
\end{lem}

\begin{exam}
\label{exam:kk_finj_non-neg}
The underlying $\kk \fb$-bimodule of $\kk \finj$ is non-negative and of finite type, hence the same holds for $\kk \finj^\ddag$, whereas $(\kk \finj^\ddag)\op$ is non-positive of finite type.
\end{exam}

\begin{prop}
\label{prop:preserve_non-neg}
For $\kk \fb$-bimodules $B_1$, $B_2$ that are non-negative, 
$B_1 \otimes_\fb B_2$ is non-negative. More explicitly, for $s, t \in \nat$, 
\begin{eqnarray}
\label{eqn:B1_B2_non-neg}
B_1 \otimes _\fb B_2 \ (\mathbf{s}, \mathbf{t}) = \bigoplus_{s\leq a \leq t} 
B_1 (\mathbf{a}, \mathbf{t}) \otimes_{\sym_a} B_2 (\mathbf{s}, \mathbf{a}).
\end{eqnarray}
If, in addition, $B_1$ and $B_2$ are both of finite type, then so is $B_1 \otimes_\fb B_2$. 

Likewise:
\begin{enumerate}
\item 
for $N$ a $\kk \fb$-module, for $t \in \nat$:
\[
(B_1 \otimes_\fb N)(\mathbf{t}) = \bigoplus_{a \leq t} B_1 (\mathbf{a}, \mathbf{t}) \otimes_{\sym_a} N(\mathbf{a});
\]
in particular, if $B_1$ and $N$ are of finite type, then so is $B_1 \otimes_\fb N$;
\item 
if $M$ is a $\kk \fb\op$-module, for $t \in \nat$:
\[
 \hom_{\kk \fb\op} (B_1, M) (\mathbf{t}) 
= 
\bigoplus_{b \leq t} \hom_{\kk \sym_b\op} (B_1 (\mathbf{b}, \mathbf{t}) , M(\mathbf{b}) );
\] 
 in particular, if $B_1$ and $M$ are of finite type, then so is $ \hom_{\kk \fb\op} (B_1, M) (\mathbf{t}) $.
\end{enumerate}

The corresponding statements hold for non-positive bimodules, {\em mutatis mutandis}.
\end{prop}

\begin{proof}
By definition, for $s, t \in \nat$, $B_1 \otimes _\fb B_2 \ (\mathbf{s}, \mathbf{t}) = \bigoplus_{a \in \nat} 
B_1 (\mathbf{a}, \mathbf{t}) \otimes_{\sym_a} B_2 (\mathbf{s}, \mathbf{a})$. The hypothesis upon $B_1$ and $B_2$ implies that the factor indexed by $a$ vanishes if either $a>t$ or $s>a$, which gives (\ref{eqn:B1_B2_non-neg}).
 In particular, $B_1 \otimes _\fb B_2 \ (\mathbf{s}, \mathbf{t})$ vanishes if $s>t$, so that $B_1 \otimes_\fb B_2$ is non-negative. Moreover, since the indexing set $\{ a \mid s \leq a \leq t \}$ is finite,  $B_1 \otimes_\fb  B_2 (\mathbf{s}, \mathbf{t})$ has finite dimension if $B_1$ and $B_2$ are both of finite type.

The remaining statements are proved similarly. For instance, by definition, $ \hom_{\kk \fb\op} (B_1, M)$ evaluated on $\mathbf{t}$ is $\prod_{b \in \nat} \hom_{\kk \sym_b \op} (B_1 (\mathbf{b}, \mathbf{t}) , M(\mathbf{b}))$. Again, the hypothesis that $B_1$ is non-negative implies that only the terms $b \leq t$ contribute. 

The results for non-positive bimodules follow from those for non-negative bimodules by exploiting the functor $(-)\op$, using Lemma \ref{lem:non_neg_op_sharp}.
 \end{proof}

Likewise, one has the following, established by similar arguments:

\begin{prop}
\label{prop:bimod_non-neg_non-pos}
Suppose that $B_1, B_2$ are $\kk \fb$-bimodules such that $B_1$ is non-negative and $B_2$ is non-positive. 
Then, for $s, t \in \nat$,  
\begin{enumerate}
\item 
$(B_1 \otimes_\fb B_2) (\mathbf{s}, \mathbf{t}) = \bigoplus_{a \leq \min \{s, t \} } B_1 (\mathbf{a}, \mathbf{t}) \otimes_{\sym_a} B_2 (\mathbf{s}, \mathbf{a})$; in particular, if $B_1$ and $B_2$ are finite type, then so is $B_1 \otimes_\fb B_2$; 
\item 
$(B_2 \otimes_\fb B_1) (\mathbf{s}, \mathbf{t}) = \bigoplus_{a \geq \max \{s, t \} } B_2 (\mathbf{a}, \mathbf{t}) \otimes_{\sym_a} B_1 (\mathbf{s}, \mathbf{a})$.
\end{enumerate}
\end{prop}

\begin{exam}
\label{exam:preservation_kk_finj}
For $M,N$ as in the statement of Proposition \ref{prop:preserve_non-neg}, both $\kk \finj \otimes_\fb N$ and $\hom_{\kk \fb\op} (\kk \finj , M)$ are of finite type.
 \textit{A contrario}, neither $M \otimes_\fb \kk \finj$ nor $\hom_{\kk \fb} (\kk \finj, N)$ are of finite type in general.  For example, take $M = \triv^{\fb\op}$. Then one has 
\[
\triv^{\fb\op} \otimes_\fb \kk \finj \ (\mathbf{s})
= 
\bigoplus_{a \in \nat}
\triv_a \otimes_{\sym_a} \kk \finj (\mathbf{s}, \mathbf{a} )
.
\]
Now, $\kk \finj (\mathbf{s}, \mathbf{a} )$ is a non-zero $\kk \sym_a$ permutation module for $a \geq s$, so that the coinvariants $\triv_a \otimes_{\sym_a} \kk \finj (\mathbf{s}, \mathbf{a} )$ are non-zero. Thus, for any $s \in \nat$, $\triv^{\fb\op} \otimes_\fb \kk \finj \ (\mathbf{s})$ is infinite-dimensional.
\end{exam}

\subsection{Truncating}
 \label{subsect:truncating}

Truncation (as introduced below) is a standard and useful technique.

\begin{nota}
\label{nota:truncate_kk_fb}
For $N$ a $\kk \fb$-module and $s \in \nat$, let $N_{\leq s} \subseteq N$ be the direct summand such that $N_{\leq s} (\mathbf{a}) $ is $N(\mathbf{a})$ if $a \leq s$ and zero otherwise; similarly for the direct summand $N_{\geq s}$.  (Thus one has $N \cong N_{\leq s} \oplus N_{\geq s+1}$.)

The same notation is used for $\kk \fb\op$-modules.
\end{nota}

\begin{lem}
\label{lem:B_otimes_truncate}
Let $B$ be a $\kk \fb$-bimodule that is non-negative.  Then, for $N$ a $\kk \fb$-module and $M$ a $\kk \fb\op$-module, the natural inclusions induce isomorphisms for $s \in \nat$:
\begin{eqnarray*}
(B \otimes_\fb N_{\leq s} )(\mathbf{s})  & \cong &  (B \otimes_\fb N )(\mathbf{s})
\\
\hom_{\kk \fb\op} (B, M_{\leq s} )(\mathbf{s})  & \cong & \hom_{\kk \fb\op} (B, M )(\mathbf{s}).
\end{eqnarray*}
In particular, these expressions depend on $N_{\leq s}$ and $M_{\leq s}$ respectively, which have finite support. 

Correspondingly, 
\begin{eqnarray*}
(M_{\geq  s}\otimes_\fb  B ) (\mathbf{s}) & \cong & (M \otimes_\fb  B )(\mathbf{s})
\\
\hom_{\kk \fb} (B, N_{\geq s} )(\mathbf{s})& \cong & \hom_{\kk \fb} (B, N )(\mathbf{s}).
\end{eqnarray*}

The corresponding statements hold if $B$ is non-positive, {\em mutatis mutandis}.
\end{lem}

\begin{proof}
To illustrate the type of argument involved, consider the first statement. By definition $(B \otimes_\fb N ) (\mathbf{s}) = \bigoplus_{a \in \nat} B (\mathbf{a}, \mathbf{s}) \otimes_{\sym_a} N(\mathbf{a})$. The hypothesis that $B$ is non-negative implies that only the terms with $a \leq s$ contribute, by Proposition \ref{prop:preserve_non-neg}. Thus, one can replace $N$ by $N_{\leq s}$ without changing the value.
\end{proof}

Truncation leads to filtrations of modules over a non-negative (respectively non-positive) monoid:

\begin{prop}
\label{prop:sub/quotient_non_neg/pos}
For  $B$  a unital monoid in $(\kk (\fb\op \times \fb), \otimes_\fb , \bmu)$,  $M$ a $B$-module, and $s \in \nat$, 
\begin{enumerate}
\item 
if $B$ is non-negative then $M_{\geq s}$ is a sub $B$-module of $M$ and the canonical projection $M \twoheadrightarrow M_{\leq s}$ exhibits $M_{\leq s}$ as a quotient $B$-module of $M$;
\item 
if $B$ is non-positive, then $M_{\leq s}$ is a sub $B$-module of $M$ and the canonical projection $M \twoheadrightarrow M_{\geq s}$ exhibits $M_{\geq s}$ as a quotient $B$-module of $M$. 
\end{enumerate}
\end{prop}

\begin{proof}
This is a direct verification.
\end{proof}

\subsection{Locally bounded complexes}
\label{subsect:loc_bd_cx}

In the differential graded (henceforth abbreviated to DG) context, the notion of a bounded complex is too restrictive; instead, the following locally bounded property is used:

\begin{defn}
\label{defn:locally-bounded}
\ 
\begin{enumerate}
\item 
A DG $\kk \fb$-module $M$ (equivalently a $\kk \fb\op$-module) is locally bounded if, for each $s \in \nat$, $M(\mathbf{s})$ is a bounded complex. 
\item 
A DG $\kk \fb$-bimodule $B$ is locally bounded if, for each $s, t \in \nat$, $B(\mathbf{s}, \mathbf{t})$ is a bounded complex.
\end{enumerate}
This is applied to DG (bi)modules with additional structure by considering the underlying $\kk\fb$  (bi)module.
\end{defn}

We introduce the following notation:

\begin{nota}
\label{nota:lbdgmod}
Denote by $\kk\fb\dash \lbdgmod \subset \kk \fb \dash \dgmod$ the full subcategory of DG $\kk \fb$-modules on the locally bounded objects.

More generally, the  notation  $\lbdgmod$ will be employed for the full subcategory of locally bounded DG objects in other contexts.
\end{nota}

\begin{rem}
Clearly a DG $\kk \fb$-bimodule $B$ is locally bounded if and only if the following equivalent conditions are satisfied:
\begin{enumerate}
\item 
for each $t \in \nat$, $B (-, \mathbf{t})$ is a locally bounded $\kk \fb\op$-module; 
\item 
for each $s \in \nat$, $B (\mathbf{s}, -)$ is a locally bounded $\kk \fb$-module.
\end{enumerate}
\end{rem}

One also has the following useful result:

\begin{lem}
\label{lem:locally-bounded_truncation}
A DG $\kk \fb$-module $M$ is locally bounded if and only if, for all $s \in \nat$, $M_{\leq s}$ is a bounded complex of $\kk \fb$-modules.
\end{lem}

One has the following analogue of Proposition \ref{prop:preserve_non-neg} for the locally bounded property, in which we use the tensor product of complexes and the usual complex $\ihom$ for DG objects. 

\begin{prop}
\label{prop:loc_bd_preservation}
Let $B$ be a non-negative, locally bounded DG $\kk \fb$-bimodule. 
\begin{enumerate}
\item 
For $N$ a locally bounded DG $\kk \fb$-module, $B \otimes_\fb N$ is a locally bounded DG $\kk \fb$-module. 
\item 
For $M$ a locally bounded DG $\kk \fb\op$-module, $\ihom_{\kk \fb\op} (B, M)$ is a locally bounded DG $\kk \fb\op$-module. 
\end{enumerate}
\end{prop}

\begin{proof}
This is proved as for Proposition \ref{prop:preserve_non-neg}, using that $\otimes$ and $\ihom$ preserve  (in the obvious sense) local boundedness of complexes of $\kk$-modules.
\end{proof}

\begin{rem}
\ 
\begin{enumerate}
\item 
Proposition \ref{prop:loc_bd_preservation} can be paraphrased as stating that $B \otimes_\fb -$ and $\ihom_{\kk \fb\op} (B, - )$ both preserve the (respective) locally bounded property.
\item 
The result does not hold in general without the  non-negative hypothesis. For instance, take $N$ to be the $\kk \fb$-module defined by $N(\mathbf{a}) = \kk \sym_a$ for $a \in \nat$. Then  $(B \otimes_\fb N)(\mathbf{t}) = \bigoplus_{a \in \nat} B (\mathbf{a}, \mathbf{t})$; this complex need not be bounded.
\end{enumerate}
\end{rem}

\section{The PROP framework}
\label{sect:prop_framework}

The main purpose of this section is to introduce the framework of Section \ref{subsect:uppd_framework}, studying the PROP $\cat \uppd$ associated to a unital operad $\uppd$ satisfying certain hypotheses. To keep the exposition relatively self-contained and to fix notation, the construction of the PROP associated to an operad is recalled in Section \ref{subsect:PROP_recollections}.

\subsection{Recollections on PROPs}
\label{subsect:PROP_recollections}

Throughout this section, $\kk$ is a unital commutative ring. We work with operads in the category of $\kk$-modules, so that an operad $\opd$ has an underlying $\kk \fb\op$-module. The identity operad $I$ has underlying $\kk\fb\op$-module $\kk$ supported on $\mathbf{1}$ and  $\opd$ is equipped with the unit map  $I \rightarrow \opd$.

The PROP associated to $\opd$ is written $\cat \opd$. The morphism spaces are given explicitly by 
\begin{eqnarray}
\label{eqn:cat_opd}
\cat \opd (\mathbf{m}, \mathbf{n}) 
= \bigoplus _{f \in \fin ( \mathbf{m}, \mathbf{n})} \bigotimes_{i \in \mathbf{n}} \opd (f^{-1} (i)),
\end{eqnarray}
where  $\fin$ is the category of finite sets and all maps. In particular,  $\cat \opd (\mathbf{n}, \mathbf{1}) = \opd (\mathbf{n})$ as a $\kk \sym_n\op$-module.

The underlying $\kk$-linear category is encoded by the underlying $\kk \fb$-bimodule $\cat \opd$ equipped with the unital monoid structure 
\begin{eqnarray*}
\bmu & \rightarrow & \cat \opd 
\\
\cat \opd \otimes_\fb \cat \opd
& \rightarrow & 
\cat\opd,
\end{eqnarray*}
where the unit is induced by the operad unit $I \rightarrow \opd$ (using that $\cat I$ identifies as $\bmu$) and the composition is induced by the operadic composition. 

\begin{rem}
\label{rem:cat_opd_obf}
For fixed $n$, $\cat \opd (-, \mathbf{n})$ is a $\kk \fb\op$-module. This is isomorphic to $\opd ^{\obf n}$; via this isomorphism  the $\sym_n$-action (induced by automorphisms of $\n$) is given by the symmetry of $\obf$.
\end{rem}

We record for later use the following consequence of this description:

\begin{lem}
\label{lem:otimes_cat_opd}
For $\kk \fb\op$-modules $M_1$ and $M_2$, there is a natural isomorphism of $\kk \fb\op$-modules
\[
(M_1 \obf M_2) \otimes_\fb \cat \opd 
\cong 
(M_1 \otimes _\fb \cat \opd) \obf (M_2 \otimes_\fb \cat \opd).
\]
In particular,  for $M$ a $\kk\fb\op$-module, one has: 
\[
(M \obf \triv_1^{\fb\op}) \otimes_\fb \cat \opd 
\cong 
(M \otimes _\fb \cat \opd) \obf  \opd.
\]
\end{lem} 

\begin{rem}
\label{rem:obf_right_cat_opd}
The functor $\obf$ passes to right $\cat \opd$-modules and, using this structure on the right hand sides, 
the isomorphisms of Lemma \ref{lem:otimes_cat_opd} are ones of right $\cat\opd$-modules. This can be expressed using right modules over the operad $\opd$ (see \cite[Section 6.1]{MR2494775}, for example).
\end{rem}

\begin{nota}
\label{nota:fsbase}
Denote by $\fs$ the category of finite sets and surjective maps. For $m,n \in \nat$, denote by  $\fsbase (\mathbf{m}, \mathbf{n}) \subset \fs (\mathbf{m}, \mathbf{n})$  the set of surjections $f$ such that $\min (f^{-1} (i)) < \min (f^{-1} (j))$  if $i<j\in \mathbf{n}$.
\end{nota}

When the operad $\opd$ is reduced (i.e., $\opd (0)=0$),  one has the following standard result:

\begin{lem}
\label{lem:case_opd_reduced}
Suppose that $\opd$ is reduced. Then, for all $m, n \in \nat$: 
\begin{enumerate}
\item 
$\cat \opd (\mathbf{m}, \mathbf{n}) 
= \bigoplus _{f \in \fs ( \mathbf{m}, \mathbf{n})} \bigotimes_{i \in \mathbf{n}} \opd (f^{-1} (i))$; in particular $\cat \opd (m,n)=0$ if $m<n$; 
\item 
$\cat \opd (\mathbf{m}, \mathbf{n})$ is free as a $\kk \sym_n$-module. Explicitly,  there is an isomorphism of $\kk \sym_n$-modules 
\begin{eqnarray}
\label{eqn:catopd_fsbase}
\cat \opd (\mathbf{m}, \mathbf{n}) \cong \kk \sym_n \otimes_\kk \Big( 
\bigoplus _{f \in \fsbase ( \mathbf{m}, \mathbf{n})} \bigotimes_{i \in \mathbf{n}} \opd (f^{-1} (i))
\Big).
\end{eqnarray}
\end{enumerate}
\end{lem}

\begin{proof}
The first statement follows since, if $f \in \fin (\mathbf{m}, \mathbf{n})$ is not surjective, there exists $i \in \mathbf{n}$ such that $f^{-1}(i)= \emptyset$; the hypothesis  that $\opd$ is reduced implies that the term indexed by $f$ contributes nothing to $\cat \opd$. 

The second statement is a consequence of the fact that $\fs (\mathbf{m}, \mathbf{n})$ is freely generated as a $\sym_n$-set by $\fsbase (\mathbf{m}, \mathbf{n})$.
\end{proof}

\begin{exam}
\label{exam:ucom}
\ 
\begin{enumerate}
\item 
For $I$ the unit operad, $\cat I$ has underlying $\kk$-linear category $\kk \fb$. 
\item 
Let $\U$ be the unique operad in $\kk$-modules such that $\U (0) = \U (1) = \kk$. 
 The underlying $\kk$-linear category of $\cat \U$ is $\kk \finj$. 
 \item 
For $\com$ the operad encoding commutative associative algebras, the underlying $\kk$-linear category of $\cat \com$ is $\kk \fs$.
\item 
For $\ucom$  the operad encoding unital commutative associative algebras, the underlying $\kk$-linear category of $\cat \ucom$ is $\kk \fin$.
\end{enumerate}
There is a commutative diagram of inclusions of operads:
\[
\xymatrix{
I 
\ar@{^(->}[r]
\ar@{^(->}[d]
&
\U 
\ar@{^(->}[d]
\\
\com
\ar@{^(->}[r]
&
\ucom.
}
\]
On passing to the associated PROPs (more precisely, their underlying $\kk$-linear categories), this yields 
\[
\xymatrix{
\kk \fb 
\ar@{^(->}[r]
\ar@{^(->}[d]
&
\kk \finj
\ar@{^(->}[d]
\\
\kk \fs
\ar@{^(->}[r]
&
\kk \fin
}
\]
induced by the inclusions of the wide subcategories $\fb$, $\finj$, $\fs$ of $\fin$.
\end{exam}

\begin{rem}
\label{rem:U_kkfinj_quadratic}
The operad $\U$ is quadratic but {\em not} binary quadratic (its generator is in arity $0$ rather than arity $2$).
Correspondingly, as discussed in Section \ref{sect:fi}, the category $\kk \finj$ is quadratic over $\kk \fb$.
\end{rem}

\subsection{The $\uppd$-framework}
\label{subsect:uppd_framework}

We introduce the framework with which we will work in the following sections, generalizing Example \ref{exam:ucom}. Namely, we consider an operad $\uppd$ such that $\uppd (0) = \uppd (1) = \kk$ and let $\ppd\subset \uppd$ be the suboperad supported on positive arities, so that $\ppd$ is reduced whereas $\uppd$ is not. One has the commutative diagram of inclusions of operads
\[
\xymatrix{
I 
\ar@{^(->}[r]
\ar@{^(->}[d]
&
\U 
\ar@{^(->}[d]
\\
\ppd
\ar@{^(->}[r]
&
\uppd
}
\]
and hence the commutative diagram of inclusions of wide subcategories:
\[
\xymatrix{
\kk \fb
\ar@{^(->}[r]
\ar@{^(->}[d]
&
\kk \finj
\ar@{^(->}[d]
\\
\cat \ppd
\ar@{^(->}[r]
&
\cat \uppd.
}
\]

Clearly, one has the following:

\begin{lem}
\label{lem:ppd_augmented}
The operad $\ppd$ has a canonical augmentation $\ppd \rightarrow I$, given by sending $\ppd (t)$ for $t >1$ to zero. 
 This induces the augmentation $\cat \ppd \rightarrow \bmu$. 
\end{lem}

\begin{rem}
In general, $\uppd$ does not admit an augmentation (for example, $\ucom$ does not).
\end{rem}

One also has the following key fact, corresponding to Hypothesis \ref{hyp:A} in this context. (In this `multi-object ring' setting, the tensor product $\otimes$ corresponds to the  `external' tensor product; this is emphasized below by using the notation $\boxtimes$.)

\begin{lem}
\label{lem:decompose_cat_uppd}
The inclusions $\kk \finj \subset \cat \uppd$ and $\cat \ppd \subset \cat \uppd$ together with the composition in $\cat \uppd$ induce an isomorphism  of $\kk \finj \boxtimes \cat \ppd\op$-modules
\[
\kk \finj \otimes_\fb \cat \ppd \stackrel{\cong}{\rightarrow} \cat \uppd.
\]
\end{lem}

\begin{proof}
This generalizes the case $\uppd = \ucom$; the fact that every morphism in $\fin$ factors as a surjection followed by an injection induces the isomorphism $\kk \finj \otimes_{\fb } \kk \fs \cong \kk \fin$.

At the level of the underlying $\kk \fb\op$-modules, one has $\uppd \cong \U \circ \ppd$, where $\circ$ is the composition product. On passage to the associated PROPs, this yields the stated result, using that $\cat \U \cong \kk \finj$. (Alternatively, one can use the explicit description of $\cat \uppd$ (see (\ref{eqn:cat_opd}) together with the factorization of morphisms in $\fin$ to deduce this statement.)
\end{proof}

This is the starting point for applying the theory of Part \ref{part:one} in this context, considering  duality relative to $ \cat \ppd$. In particular, by the general considerations of Section \ref{sect:leftaug} one has:

\begin{prop}
\label{prop:right_augmentation}
The augmentation $\epsilon_{\kk \finj} : \kk \finj \rightarrow \bmu$ induces a right augmentation 
\[
\epsilon_{\cat \uppd} : 
\cat \uppd 
\rightarrow 
\cat \ppd, 
\]
with kernel $\ker \epsilon_{\kk \finj} \otimes_\fb \cat \ppd$.   In particular, this induces a  $\kk \finj\op$-module structure on $\cat \ppd$. 
\end{prop}

\begin{rem}
\label{rem:ppd_right_finj}
The $\kk \finj\op$-module structure of $\cat \ppd$ restricts to a $\kk \finj\op$-module structure on (the underlying $\kk \fb\op$-module of) $\ppd$. This, together with the operad structure of $\ppd$, determines the operad structure of $\uppd$. The existence of this structure is well-known (see \cite[Chapter 2]{MR3643404}, for example).
\end{rem}

One also has the following counterpart to Proposition \ref{prop:right_augmentation}:

\begin{prop}
\label{prop:left_augmentation}
The augmentation $\epsilon_{\cat \ppd} \cn \cat \uppd \rightarrow \bmu$ induces a left augmentation:
\[
\cat \uppd \rightarrow \kk \finj.
\]
In particular, this induces a  $\cat \ppd$-module structure on $\kk \finj$.
\end{prop}

\subsection{Finiteness properties}
\label{subsect:prop_finiteness_properties}

Here we take $\kk$ to be a field.  Recall from Example \ref{exam:kk_finj_non-neg} that the underlying $\kk \fb$-bimodule of $\kk \finj$ is of  finite type and non-negative. In the above framework, we also have:

\begin{prop}
\label{prop:cat_ppd_non-positive}
\ 
\begin{enumerate}
\item 
The underlying $\kk \fb$-bimodule of $\cat \ppd$ is non-positive. 
\item 
If the underlying $\kk \fb\op$-module of $\uppd$ has finite type, then the underlying $\kk \fb$-bimodules of  both $\cat \uppd$ and $\cat \ppd$ have finite type.
\end{enumerate}
\end{prop}

\begin{rem}
\ 
\begin{enumerate}
\item 
These finiteness properties are critical in the constructions below.
\item 
$\cat \uppd$ is never non-positive; it is non-negative if and only if $\ppd = I$, so that $\uppd  = \U$.
\end{enumerate}
\end{rem}

\section{Absolute Koszul duality for  $\kk \finj$}
\label{sect:fi}

The purpose of this section is two-fold: first to explain that the category $\kk \finj$ is Koszul and then to deduce a version of the BGG-correspondence, translated to this `universal' context. This is  the {\em homogeneous} case, using quadratic duality over $\kk \fb$; it will  be applied to the  {\em  nonhomogeneous} case in Section \ref{sect:prop} using relative duality.

Throughout,  $\kk$ is a field; for the main results it will be taken to have characteristic zero.

\subsection{The category $\kk \finj$ is quadratic}

The purpose of this section is to relate the categories $\kk \finj$ and $\kk \finj^\ddag$ using quadratic duality (see Example \ref{exam:finj_ddag} for $\kk \finj^\ddag$). 
 The following provides the (right) finitely-generated projectivity property that is required when dualizing (without requiring that $\kk$ have characteristic zero):

\begin{lem}
\label{lem:kkfinj_fg_proj_fbop}
For $a, b \in \nat$,  $\kk \finj (\mathbf{a}, \mathbf{b})$ and 
 $\kk \finj^\ddag (\mathbf{a}, \mathbf{b})$ are finitely-generated free $\kk \sym_a \op$-modules. 
  Hence, for fixed $b \in \nat$, the $\kk \fb\op$-modules $\kk \finj (-, \mathbf{b})$ and 
 $\kk \finj^\ddag (-, \mathbf{b})$ are finitely-generated projective. 
\end{lem}

\begin{proof}
For $\kk \finj$, the first statement follows from the fact that the $\sym_a\op$-set $\finj (\mathbf{a}, \mathbf{b})$ is finitely-generated free. The second then follows, since $\finj (\mathbf{a}, \mathbf{b}) =\emptyset $ if $a>b$. 
 The result for $\kk \finj^\ddag$ follows from that for $\kk \finj$.
\end{proof}

The underlying $\kk \fb$-bimodules can be described succinctly using the Day convolution product for $\kk \fb$-modules:

\begin{lem}
\label{lem:underlying_kkfb-bimodules}
There are isomorphisms of the underlying $\kk \fb$-bimodules:
\begin{eqnarray*}
\kk \finj & \cong & \bmu \ofb \triv^\fb
\\
\kk \finj^\ddag & \cong & \bmu \ofb \sgn^\fb,
\end{eqnarray*}
where the $\kk \fb\op$-module structure is induced by that of $\bmu$.
\end{lem}

\begin{proof}
The identification for $\kk \finj$ follows from the identification of the permutation module (for $a \leq b \in \nat$) 
$ 
\kk \finj (\mathbf{a}, \mathbf{b}) 
\cong 
\kk \sym_b \otimes_{\sym_{b-a}} \triv_{b-a}$, 
where the $\kk (\sym_a \op \times \kk \sym_b)$-module structure on the right hand side is induced by the canonical $\kk \sym_b$-bimodule structure of $\kk \sym_b$ given by the left and right regular actions. The case of $\kk \finj^\ddag$ follows from this. 
\end{proof}

Recall from Example \ref{exam:kk_finj} that $\kk \finj$ is $\nat$-graded, with $\kk \finj^{[0]}$ isomorphic to $\kk \fb \cong \bmu$. Moreover, it is clear that $\kk \finj$ is generated over $\kk \finj^{[0]}$ by $\kk \finj ^{[1]}$. 
 By definition, $\kk \finj\dg{1} (\mathbf{a}, \mathbf{b})$ is non-zero if and only if $b=a+1$; the following identifies its underlying $\kk \fb$-bimodule:

\begin{lem}
\label{lem:kk_finj_dg1}
For $a \in \nat$, the $\kk (\sym_a\op \times \sym_{a+1})$-module $\kk \finj (\mathbf{a}, \mathbf{a+1})$ is isomorphic to $\kk \sym_{a+1} \downarrow^{\sym_{a+1}\op}_{\sym_a\op}$, where $\kk \sym_{a+1}$ is equipped with the left and right regular $\kk \sym_{a+1}$-actions.
\end{lem}

\begin{proof}
This follows from the fact  that an injection $\mathbf{a} \hookrightarrow \mathbf{a+1}$ extends canonically to an automorphism of $\mathbf{a+1}$ and any such automorphism is obtained uniquely in this way.
\end{proof}

For later usage, we record the following basis for the underlying $\kk \sym_a\op$-module:

\begin{lem}
\label{lem:sym_a_basis_kkfinj[1]} 
For $a \in \nat$, the $\kk \sym_a\op$-module $\kk \finj (\mathbf{a}, \mathbf{a+1})$ is free with basis given by $\{ [i_x] \ | \ x \in \mathbf{a+1} \}$, where $i_x :\mathbf{a} \rightarrow  \mathbf{a+1}$ is the order-preserving injection  with image $(\mathbf{a+1}) \backslash \{ x\}$. 
\end{lem} 

The product of $\kk \finj$ induces  the surjection
\begin{eqnarray}
\label{eqn:prod_deg_two}
\kk \finj \dg{1} \otimes_\fb \kk \finj \dg{1}
\twoheadrightarrow 
\kk \finj \dg{2}.
\end{eqnarray}
For $a \in \nat$, it is straightforward to check that, evaluated on $(\mathbf{a}, \mathbf{a+2})$, this identifies with the surjection of $\kk (\sym_a\op \times \sym_{a+2})$-modules
\[
\kk \sym_{a+2}\downarrow _{\sym_a\op}^{\sym_{a+2}\op} \cong \kk \sym_{a+2} \otimes_{\sym_2} \kk \sym_2 \twoheadrightarrow \kk \sym_{a+2} \otimes_{\sym_a} \triv_2
\]
induced by applying the functor $\kk \sym_{a+2} \otimes_{\sym_2} -$ to the canonical surjection $\kk \sym_2 \twoheadrightarrow \triv_2$ given by $[\sigma] \mapsto 1 \in \kk \cong \triv_2 $. 

\begin{lem}
\label{lem:presentation_kkfinj_dg2}
The  product (\ref{eqn:prod_deg_two}) fits into the short exact sequence of $\kk \fb$-bimodules
\begin{eqnarray}
\label{eqn:ses_presentation}
0
\rightarrow 
\bmu \ofb \sgn_2^\fb
\rightarrow 
\kk \finj \dg{1} \otimes_\fb \kk \finj \dg{1}
\rightarrow 
\kk \finj \dg{2}
\rightarrow 
0
\end{eqnarray}
that is obtained by applying the functor $\bmu \ofb -$ to the short exact sequence of $\kk \sym_2$-modules
\[
0
\rightarrow 
\sgn_2
\rightarrow 
\kk \sym_2 
\rightarrow 
\triv_2 
\rightarrow 
0.
\]
Moreover, the short exact sequence  (\ref{eqn:ses_presentation}) splits as a sequence of $\kk \fb\op$-modules.
\end{lem}

\begin{proof}
The first statement follows from the discussion preceding the Lemma. The splitting as $\kk \fb\op$-modules is a consequence of Lemma \ref{lem:kkfinj_fg_proj_fbop}, which ensures that $\kk \finj \dg{2}$ is projective as a $\kk \fb\op$-module.
\end{proof}

Using these identifications, we have:

\begin{prop}
\label{prop:kk_finj_quadratic}
\ 
\begin{enumerate}
\item 
The category $\kk \finj$ is quadratic over $\kk \fb$, with quadratic datum
$
(\kk \fb; \bmu \ofb \triv_1, \bmu \ofb \sgn_2 ).
$
\item 
The category $\kk \finj^\ddag$ is quadratic over $\kk \fb$, with quadratic datum
$
(\kk \fb;  \bmu \ofb \sgn_1, \bmu \ofb \triv_2 ).
$
\end{enumerate}
\end{prop}

\begin{proof}
As stated above, $\kk \finj$ is generated over $\kk \fb$ by $\kk \finj\dg{1}$; moreover, Lemma \ref{lem:presentation_kkfinj_dg2} identifies the relations in degree two. It remains to check that there are no further relations; this is a consequence of the fact that the symmetric groups are generated by their transpositions.

The result for $\kk \finj^\ddag$ then follows, since $(-)^\ddag$ is monoidal, by  Proposition \ref{prop:basic_properties_ddag_op} (\ref{item:ddag_monoidal}). The identification of the quadratic datum is straightforward (noting that $\sgn_1 = \triv_1$).
\end{proof}

\subsection{Dualizing the quadratic categories $\kk \finj$ and $\kk \finj ^\ddag$}

Lemma \ref{lem:kkfinj_fg_proj_fbop} ensures that $\kk \finj$ and $\kk \finj ^\ddag$ are right finitely-generated projective over $\kk \fb$. This allows us to dualize the quadratic structures of Proposition \ref{prop:kk_finj_quadratic}.

\begin{rem}
\label{rem:appropriate_duality}
Since we are working with $\kk \fb$-bimodules $X$ with underlying $\kk \fb\op$-module that is finitely-generated projective, the appropriate duality functor is $X \mapsto \hom_{\kk\fb\op} (X, \bmu)$; more generally we consider the functor $\hom_{\kk\fb\op} (X, -)$. By Lemma  \ref{lem:duality}, we can identify $\hom_{\kk\fb\op} (X, \bmu)$ with $X^\sharp$, the naive duality functor.
\end{rem}

\begin{rem}
\label{rem:permutation_module}
In preparation for the results below, we recall  how duality works for permutation modules associated to finite $G$-sets, where $G$ is a discrete group. Consider such a set $S$, so that one has the $\kk G$-module $\kk S$. The dual module $(\kk S)^\sharp$ (considered as a right $\kk G$-module) is a permutation module associated to the finite {\em right} $G$-set $\{ \eta_s \mid s \in S \}$, where $\eta_s \in (\kk S)^\sharp$ is given by $\eta_s (t) = \delta_{st}$.  The group $G$ acts on the right on this set by $\eta_s g = \eta_{g^{-1} s}$. Hence $(\kk S)^\sharp$, considered as a left $\kk G$-module, is canonically isomorphic to $\kk S$, via the bijection $s \leftrightarrow \eta_s$ of the generating sets.
\end{rem}

\begin{lem}
\label{lem:sharp_duals}
There are isomorphisms of $\kk \fb$-bimodules:
\begin{eqnarray*}
(\bmu \ofb \triv^\fb)^\sharp &\cong & \bmu \obf \triv^{\fb\op} \cong (\bmu \ofb \triv^\fb)\op \\
(\bmu \ofb \sgn^\fb)^\sharp &\cong & \bmu \obf \sgn^{\fb\op} \cong (\bmu \ofb \sgn^\fb)\op .
\end{eqnarray*}
\end{lem}

\begin{proof}
The first statement is a global form of the isomorphism of $\kk (\sym_a \op \times \sym_b)$-modules 
$\kk \finj (\mathbf{a}, \mathbf{b}) ^\sharp 
\cong 
\kk \finj (\mathbf{a}, \mathbf{b})\op$, 
using $\op$ to adjust the variance; this follows from the fact that $\kk \finj (\mathbf{a}, \mathbf{b})$ is a finite dimensional permutation module (use Remark \ref{rem:permutation_module}). The second statement is  a twisted form of this. 
\end{proof}

\begin{prop}
\label{prop:right_qdual_kk_finj}
The (right) quadratic dual $\kk \finj\qd$ of $\kk \finj$ over $\kk \fb$ is isomorphic to $(\kk \finj^\ddag)\op$. 
\end{prop}

\begin{proof}
By Proposition \ref{prop:kk_finj_quadratic}, $\kk \finj$ is quadratic over $\kk \fb$, associated to the quadratic datum 
\[
(\kk \fb; \bmu\ofb \triv_1^\fb, \bmu \ofb\sgn_2 ^\fb)
\]
and $\kk \finj \dg{2} \cong \bmu \ofb \triv_2$; this is $2$-right finitely-generated projective. 

By Lemma \ref{lem:sharp_duals}, the dual quadratic datum is 
$ 
(\kk \fb; (\bmu\ofb \sgn_1^\fb)\op, (\bmu \ofb\triv_2 ^\fb)\op).
 $ 
By Proposition \ref{prop:kk_finj_quadratic}, one sees that this quadratic datum yields the quadratic category $(\kk \finj^\ddag)\op$. 
\end{proof}

\begin{rem}
In these statements,  the $\nat$-gradings have been suppressed. For $n\in \nat$, cohomological degree $n$ for $\kk \finj\qd$ corresponds to $(\kk \finj\qd)\dg{-n}$.
\end{rem}

Proposition  \ref{prop:right_qdual_kk_finj} and  the fact that $\kk \finj$ is non-negative of finite type together imply:

\begin{cor}
\label{cor:kk_finj_!_non-positive}
The underlying $\kk \fb$-bimodule of  $\kk \finj \qd$ is non-positive and of finite type.
\end{cor}

\subsection{The dualizing complex $\vk (\kk \finj)$}
\label{subsect:dualizing_cx_kk_finj}
The appropriate dualizing complex  is 
\[
\vk (\kk \finj) = 
\kk \finj \otimes_\fb \kk \finj\qd 
\cong 
\kk \finj \otimes_\fb (\kk \finj^\ddag) \op,
\]
defined as in  Section \ref{subsect:vk}, working over $\kk \fb$.  The cohomological grading is inherited from that of $\kk \finj\qd$.

The differential is given by inner multiplication with the class $\e$ (cf. Notation \ref{nota:efrak_prime}). This class  is given by the natural transformation:
\begin{eqnarray}
\label{eqn:diff}
\bmu \rightarrow \kk \finj \dg{1} \otimes_\fb (\kk \finj\qd)\dg{-1} ,
\end{eqnarray}
that is adjoint to the identity map of $ \kk \finj \dg{1}$,  using that $(\kk \finj\qd)\dg{-1}$ is the dual $(\kk \finj \dg{1})^\sharp$. (Observe that the domain (respectively the codomain) of (\ref{eqn:diff}) is non-zero evaluated on $(\mathbf{a}, \mathbf{b})$ if and only $a=b$.) 

This map is described explicitly as follows. By Lemma \ref{lem:sym_a_basis_kkfinj[1]},  $\kk \finj (\mathbf{a}, \mathbf{a+1})$ has basis as a $\kk \sym_a \op$-module given by $\{ [i_x] \mid x \in \mathbf{a+1} \}$. Denote by $\{ [i_x\op] \mid x \in \mathbf{a+1} \}$ the dual basis of the $\kk \sym_a$-permutation module $\kk \finj\qd (\mathbf{a+1}, \mathbf{a}) \cong \big( \kk \finj (\mathbf{a}, \mathbf{a+1})\big)^\sharp$ (cf. Remark \ref{rem:permutation_module}).

\begin{lem}
\label{lem:cpt_diff}
The map (\ref{eqn:diff}) evaluated on $(\mathbf{a+1}, \mathbf{a+1}) $ identifies with the morphism of $\kk \sym_{a+1}$-bimodules
\[
\kk \sym_{a+1} 
\rightarrow 
\kk \finj (\mathbf{a}, \mathbf{a+1}) \otimes_{\sym_{a}} 
\kk \finj\qd( \mathbf{a+1}, \mathbf{a}) 
\] 
induced by sending $[e] \in \kk \sym_{a+1}$ to $\sum_{x \in \mathbf{a+1}} [i_x] \otimes [i_x\op]$.
\end{lem}

We note the following finiteness property:

\begin{prop}
\label{prop:vk_kk_finj_finiteness}
For any $s, t \in \nat$, the complex $\vk (\kk \finj) (\mathbf{s}, \mathbf{t})$, considered as a complex of $\kk (\sym_s \op \times \sym_t)$-modules, is concentrated in cohomological degrees in $[s- \min \{s,t\}, s]$ and, in each cohomological degree, the module has finite $\kk$-dimension.
\end{prop}

\begin{proof}
This can either be proved directly from the definition of the complex $\vk (\kk \finj)$ or by using the general finiteness results of Section \ref{sect:finiteness_properties} (in particular Proposition \ref{prop:bimod_non-neg_non-pos}) in conjunction with Corollary \ref{cor:kk_finj_!_non-positive}.
\end{proof}

\begin{exam}
\label{exam:vk_cohomology}
Evaluating the complex $\vk (\kk \finj)$ on $(\mathbf{2}, \mathbf{2})$, one obtains the complex of $\kk (\sym_2\op \times \sym_2)$-modules: 
\[
 \kk \sym_2
\rightarrow 
 \kk \sym_2 \boxtimes \kk \sym_2 
\rightarrow 
\sgn_2 \boxtimes \triv_2,
\]
where the first term is in cohomological degree $0$ and the last in degree $2$.

If $\kk$ has characteristic zero,  
a simple calculation shows that the cohomology of this complex is $\triv_2 \boxtimes \sgn_2$  in degree $1$, zero elsewhere; in particular,  $\vk (\kk \finj)$ is not acyclic.
\end{exam}

\subsection{The adjunction induced by $\vk (\kk \finj)$ for left DG modules}
\label{subsect:adjunction_vk_kk_finj_left}

We focus upon the adjunction for left DG modules (the case of right DG modules is analogous). Weak equivalence is synonymous with quasi-isomorphism.

Proposition \ref{prop:vk_adjunctions} yields the adjunction:
\begin{eqnarray}
\label{eqn:adj_vk_kk_finj}
 \vk (\kk \finj)  \otimes_{\kk \finj\qd } - 
\cn
\kk \finj\qd \dash \dgmod 
\rightleftarrows
\kk \finj \dash \dgmod 
\cn
\ihom_{\kk \finj} (\vk (\kk \finj), -).
\end{eqnarray}
The underlying functors identify  as 
\begin{eqnarray*}
\vk (\kk \finj)  \otimes_{\kk \finj\qd } - 
&\cong & 
\kk \finj \otimes_\fb - \\
\ihom_{\kk \finj} (\vk (\kk \finj), -)
&\cong &
\ihom_{\kk \fb} (\kk \finj\qd, -),
\end{eqnarray*}
in which the right hand sides are  equipped with the appropriate twisted differentials.

Using Lemma \ref{lem:underlying_kkfb-bimodules}, these functors identify as follows: 

\begin{lem}
\label{lem:identify_underlying_functors_kk_finj}
For $M$ a $\kk \fb$-module, there are natural isomorphisms:
\begin{eqnarray*}
\kk \finj \otimes_\fb M & \cong & \triv^\fb \ofb M 
\\
\hom_{\kk \fb} (\kk \finj\qd, M) & \cong & \sgn^\fb \ofb M.
\end{eqnarray*}
\end{lem}

\begin{rem}
Using these identifications, the functors appearing in (\ref{eqn:adj_vk_kk_finj}) identify as the standard Koszul complexes in the theory. For example, for $M$ a $\kk \finj$-module, $\hom_{\kk \fb} (\kk \finj\qd, M)$ is the usual Koszul complex that calculates the $\finj$-homology of $M$.
\end{rem}

Lemma \ref{lem:identify_underlying_functors_kk_finj} makes the proof of the following transparent:

\begin{prop}
\label{prop:exactness_vk_finj_left_modules}
The functors $\vk (\kk \finj)  \otimes_{\kk \finj\qd } - $ and $\ihom_{\kk \finj} (\vk (\kk \finj), -)$ are exact and commute with colimits. Moreover, both these functors preserve weak equivalences.
\end{prop}

\begin{proof}
The first statement is clear from the identification of the underlying functors given in Lemma \ref{lem:identify_underlying_functors_kk_finj}. 

To show that the functors preserve weak equivalences, by using the natural truncations as in Section \ref{subsect:truncating} together with a standard comparison of spectral sequences argument, one reduces to the case where the $\kk \finj\qd$-module structure (respectively $\kk \finj$-module structure) is given by restriction along the augmentation $\kk \finj \qd \rightarrow \kk \fb$ (resp. $\kk \finj \rightarrow \kk \fb$). The key point in this reduction argument is that, after evaluation in any given arity, one is reduced to considering a filtration of finite length. 

Hence, suppose that $M$ is a complex of $\kk \finj \qd$-modules arising from a complex of $\kk \fb$-modules as above. Then $ \vk (\kk \finj)  \otimes_{\kk \finj\qd } M$ is naturally isomorphic to $\kk \finj \otimes_\fb M$, where the differential is given by that of $M$ (i.e., is not twisted). If $M \rightarrow N$ is a weak equivalence between two such complexes of  $\kk \finj \qd$-modules, it is then clear that the induced map 
\[
\vk (\kk \finj)  \otimes_{\kk \finj\qd } M \rightarrow \vk (\kk \finj)  \otimes_{\kk \finj\qd } N
\]
is a weak equivalence, as required. 

A similar argument applies when considering the functor $ \ihom_{\kk \fb} (\kk \finj\qd, -)$.
\end{proof}

\begin{rem}
\label{rem:gen_colimits_exact_we}
The proof of the above result relies on the following ingredients.  For the first statement, we used the finite projectivity properties and the fact that  $\kk \finj \qd$ is non-positive. For the preservation of weak equivalences, the key input was the fact that $\kk \finj$ is non-negative and $\kk \finj \qd$ is non-positive.

Hence, the argument can be applied whenever these ingredients are available.
\end{rem}

Then for $X$ a DG $\kk \finj\qd$-module and $Y$ a DG $\kk \finj$-module, the adjunction yields the unit and counit:
\begin{eqnarray}
\label{eqn:left_X_unit_kk_finj!}
X &  \rightarrow & \ihom_{\kk \fb} (\kk \finj\qd , \kk \finj \otimes_\fb X)
 \\
\label{eqn:left_X_counit_kk_finj}
\kk \finj \otimes_\fb \ihom_{\kk \fb} (\kk \finj\qd, Y) 
&\rightarrow & Y, 
\end{eqnarray}
one again using the appropriate twisted differentials.

\begin{exam}
\label{exam:unit_counit_vk_kk_finj_adj}
Using the structures provided by the augmentations  $\kk \finj\qd \rightarrow \bmu$ and $\kk \finj \rightarrow \bmu$ respectively, the adjunction (co)unit yield:
\begin{eqnarray}
\label{eqn:left_bmu_unit_kk_finj!}
\bmu & \rightarrow  & \ihom_{\kk \fb} (\kk \finj\qd , \kk \finj ) \\
\label{eqn:left_bmu_counit_kk_finj}
\kk \finj \otimes_\fb \ihom_{\kk \fb} (\kk \finj\qd, \bmu) 
&\rightarrow & \bmu.
\end{eqnarray}

Here, for example, one can identify the underlying $\kk \fb$-bimodule of $\ihom_{\kk \fb} (\kk \finj\qd , \kk \finj )$ as $\sgn^\fb \ofb \triv^\fb \ofb \bmu$. The differential is induced by the differential of the classical Koszul complex, which corresponds in this context to maps of $\kk \fb$-modules of the form 
\[
\sgn_a \ofb \triv_b \rightarrow \sgn_{a-1} \ofb \triv_{b+1}, 
\]
for $a, b \in \nat$. 

The term appearing in the counit is analysed similarly.
\end{exam}

\subsection{The Koszul property and the  BGG correspondence}
\label{subsect:finj_BGG}

The following result (using different language) is essentially contained  in  \cite{MR3430359}:

\begin{thm}
\label{thm:kk_finj_Koszul}
Suppose that $\kk$ is a field of characteristic zero. Then the category $\kk \finj$ is Koszul over $\kk \fb$, in the sense that the adjunction unit (\ref{eqn:left_bmu_unit_kk_finj!}) and adjunction counit (\ref{eqn:left_bmu_counit_kk_finj}) are both weak equivalences.
\end{thm}

\begin{proof}
This follows from the acyclicity of the Koszul complexes and the identifications given in Example \ref{exam:unit_counit_vk_kk_finj_adj}.
\end{proof}

\begin{rem}
\label{rem:kk_finj_dual}
\ 
\begin{enumerate}
\item 
This terminology Koszulity is consistent with that used in Section \ref{sect:koszul} (see Proposition \ref{prop:Koszul_equivalent_unit/counit}).
\item 
The result is stronger, since one has the identification $\kk \finj\qd \cong (\kk \finj^\ddag) \op$. This can be interpreted (using sheering, as reviewed in Appendix \ref{sect:sheer}) as saying that $\kk \finj$ is Koszul self-dual. This is used by Sam and Snowden in  \cite{MR3430359} to define their {\em Fourier transform}. 
\end{enumerate}
\end{rem}

Now, Theorem \ref{thm:kk_finj_Koszul} is the key step in establishing the BGG correspondence relating DG $\kk \finj$-modules and DG $\kk \finj\qd$-modules. This is essentially contained within \cite{MR3430359} (although, for their application, they impose further finiteness hypotheses, so that vector space double duality is an equivalence); Sam and Snowden refer to \cite{MR1990756} for a treatment of the classical BGG correspondence.

For current purposes, we use the notion of a locally bounded complex introduced in Section \ref{subsect:loc_bd_cx}.
First we note the following:

\begin{prop}
\label{prop:restrict_vk_kk_finj_adj_lbd}
The adjunction 
(\ref{eqn:adj_vk_kk_finj}) restricts to the adjunction for locally bounded objects
$$
 \vk (\kk \finj)  \otimes_{\kk \finj\qd } - 
\cn
\kk \finj\qd \dash \lbdgmod 
\rightleftarrows
\kk \finj \dash \lbdgmod 
\cn
\ihom_{\kk \finj} (\vk (\kk \finj), -).$$
\end{prop}

\begin{proof}
This follows from Proposition \ref{prop:loc_bd_preservation} by using the fact that $\kk \finj$ is non-negative and $\kk \finj \qd$ is non-positive, together with the identifications of the underlying functors given after equation 
 (\ref{eqn:adj_vk_kk_finj}).
\end{proof}

This leads to the universal algebra version of the BGG equivalence. To state this, we introduce the following notation:

\begin{nota}
\label{nota:holbd_kk_finj}
Write $ \holbd (\kk \finj)$ (respectively $\holbd (\kk \finj \qd)$) for the homotopy category of $\kk\finj \dash \lbdgmod $ (resp. $\kk \finj\qd \dash\lbdgmod$) obtained by inverting weak equivalences. (This is a locally-bounded analogue of the bounded derived category; it is a stable $\infty$-category.)
\end{nota}

\begin{thm}
\label{thm:BGG}
Suppose that $\kk$ is a field of characteristic zero.  The adjunction of Proposition \ref{prop:restrict_vk_kk_finj_adj_lbd} induces an equivalence of categories
 \[
\holbd (\kk \finj \qd) 
\stackrel{\simeq}{\rightleftarrows}
\holbd (\kk \finj).
\]
\end{thm}

\begin{proof}
The result is proved by establishing the following two facts.
\begin{enumerate}
\item 
For $X$ a DG $\kk \finj\qd$-module that is locally bounded,  the adjunction unit (\ref{eqn:left_X_unit_kk_finj!}) is a weak equivalence.
\item 
For $Y$ a DG $\kk \finj$-module that is locally bounded, the adjunction counit (\ref{eqn:left_X_counit_kk_finj}) is a weak equivalence.
\end{enumerate}
To explain the ideas involved, we sketch a proof of the first statement. 

As in the proof of Proposition \ref{prop:exactness_vk_finj_left_modules}, one first reduces to the case where $X$ is supported on $\mathbf{s}$. In particular, the locally bounded hypothesis implies that it is bounded. Then, by an analogous and standard argument, one reduces to the case where $X$ is concentrated in a single cohomological degree. 

Finally, since $\kk$ is a field of characteristic zero, one reduces immediately to the case where $X(\mathbf{s}) = \kk \sym_s$, considered as a $\kk \finj\qd$-module. That one obtains a weak equivalence in this case follows from Theorem \ref{thm:kk_finj_Koszul}.
\end{proof}

\section{Relative nonhomogeneous Koszul duality for $\cat \uppd$ over $\cat \ppd$}
\label{sect:prop}

This section feeds the homogeneous Koszul duality results for $\kk \finj$ considered in Section \ref{sect:fi} into the relative nonhomogeneous setting. Here $\uppd$ is as in Section \ref{subsect:uppd_framework}; in particular, the sub operad $\ppd \subset \uppd$ is reduced, and the composition in $\cat \uppd$ yields the isomorphism of $\kk \finj \boxtimes \cat \ppd\op$-modules:
\[
\kk \finj \otimes_\fb \cat \ppd \stackrel{\cong}{\rightarrow} \cat \uppd.
\]
By Proposition \ref{prop:right_augmentation}, the augmentation of $\kk \finj$ induces a right augmentation $\cat \uppd \rightarrow \cat \ppd$. This is used  to study the relative nonhomogeneous Koszul duality of $\cat \uppd$ over $\kk \finj$, applying the results of Part \ref{part:one}.

\subsection{$\cat \uppd$ is nonhomogeneous quadratic}

As explained in Section \ref{sect:fi}, $\kk \finj$ is  homogeneous quadratic over $\kk \fb$, associated to the quadratic datum
 $
(\kk \fb; \bmu \ofb \triv_1^\fb , \bmu \ofb \sgn_2^\fb).
$ 
We first check that this induces a filtration of $\cat \uppd$ making it into a nonhomogeneous (right) quadratic algebra over $\cat \ppd$. (This corresponds to checking that the equivalent conditions of Proposition \ref{prop:filt_hyp}, adapted to the right augmented context, are satisfied.)

\begin{nota}
\label{nota:filt_G_cat_uppd}
Denote the induced increasing filtration of $\cat \uppd$ by $\kk \fb \boxtimes \cat \ppd\op$-modules by $G_\bullet \cat \uppd$, where, for $n \in \zed$,
\[
G_n  \cat \uppd := \bigoplus_{t \leq n} \kk \finj\dg{t} \otimes_\fb \cat \ppd.
\]
\end{nota}

\begin{exam}
One has $G_{-1} \cat \uppd =0$, $G_0 \cat \uppd = \cat \ppd$, and 
$
G_1 \cat \uppd = \cat \ppd \ \oplus \ \kk \finj^{[1]} \otimes_\fb \cat \ppd$,
equipped with the obvious inclusions. 
\end{exam}

A direct verification shows the following:

\begin{lem}
\label{lem:identify_filt_cat_uppd}
\
\begin{enumerate}
\item 
For $t\in \nat$, there is an isomorphism of $\kk \fb$-bimodules
\[
G_t \cat \uppd (\mathbf{a}, \mathbf{b}) 
= \bigoplus_{\substack {f \in \fin (\mathbf{a} , \mathbf{b}) \\
|\mathbf{b} \backslash \mathrm{im}(f)| \leq t}}
 \bigotimes_{i \in \mathbf{n}} \uppd (f^{-1} (i)).
\]
\item 
$G_\bullet \cat \uppd$ is  a filtration of the category  $\cat \uppd$; i.e., for each $s, t \in \nat$, the composition restricts to 
\[
G_s \cat \uppd \otimes_\fb G_t \cat \uppd \rightarrow G_{s+t} \cat \uppd. 
\]
\item 
The morphisms of the  $\kk$-linear category $\cat \uppd$ are generated over $\cat \ppd$ by $\kk \finj\dg{1} \otimes _\fb \cat \ppd$.
\end{enumerate} 
\end{lem}

Moreover, one has:

\begin{lem}
\label{lem:rgrad_cat_uppd_fg_projective}
For each $n \in \nat$, 
\begin{enumerate}
\item 
there is an isomorphism of $\cat \ppd\op$-modules $ \rgrad_n \cat \uppd \cong \cat \ppd \ofb \triv_n^\fb$, where $\cat \ppd$ acts on the right via its action on $\cat \ppd$;
\item 
for $b \in \nat$, $(\rgrad_n \cat \uppd)(-, \mathbf{b})$ is finitely-generated projective as a $\cat \ppd\op$-module.
\end{enumerate}
\end{lem}

\begin{proof}
For the first statement, one uses the isomorphism of $\kk \fb$-bimodules $\kk \finj \dg{n} \cong \bmu \ofb \triv_n^\fb$ that follows from Lemma \ref{lem:underlying_kkfb-bimodules}. Then, since $\rgrad_n \cat\uppd$ is, by definition, isomorphic as a $\cat \ppd \op$-module to $\kk \finj\dg{n} \otimes_\fb \cat \ppd$, the result follows by a straightforward verification.

For the second statement, one uses the  isomorphism of $\cat \ppd\op$-modules $(\rgrad_n \cat \uppd)(-, \mathbf{b})\cong \kk \finj\dg{n} 
(- , \mathbf{b}) \otimes_\fb \cat \ppd$. The result then follows since $\kk \finj\dg{n} 
(- , \mathbf{b})$ is a finitely-generated projective $\kk \fb\op$-module, by Lemma \ref{lem:kkfinj_fg_proj_fbop}.
\end{proof}

Now, Lemma \ref{lem:identify_filt_cat_uppd} implies that each $\rgrad_n \cat \uppd$ has the structure of a $\cat \ppd$-bimodule. The left $\cat \ppd$-action can be described using the following, which is the counterpart of the structure recalled in Remark \ref{rem:obf_right_cat_opd}:

\begin{prop}
\label{prop:cat_ppd_modules_ofb}
\cite{MR4835394}
For  $\cat \ppd$-modules $M, N $, the Day convolution product $M \ofb N$ of their underlying $\kk \fb$-modules has a canonical $\cat \ppd$-module structure that extends its $\kk \fb$-module structure. 
\end{prop}

By hypothesis, $\ppd$ is reduced with $\ppd (1)= \kk$, hence there is a canonical augmentation $\cat \ppd \twoheadrightarrow \bmu$. This implies that any $\kk \fb$-module can be considered as a $\cat \ppd$-module. 

\begin{cor}
\label{cor:left_cat_ppd}
For $M$ a $\cat \ppd$-module and $W$ a $\kk\fb$-module, considered as a $\cat \ppd$-module via the augmentation, $M \odot _\fb W$ has a canonical $\cat \ppd$-module structure. Hence, $\cat \ppd \ofb W$ has a natural $\cat \ppd$-bimodule structure.
\end{cor}

Using these structures, the general theory gives:

\begin{prop}
\label{prop:grad_cat_uppd_quadratic}
The category $\rgrad \cat \uppd$ is quadratic over $\cat \ppd$ on the quadratic datum:
\[
(\cat \ppd ;  \kk \finj^{[1]} \otimes_\fb \cat \ppd \cong \cat \ppd \ofb \triv_1, (\bmu \ofb \sgn_2) \otimes_\fb \cat \ppd \cong \cat \ppd \ofb \sgn_2 ).
\]
Moreover, for all $b, n \in \nat$, $(\rgrad_n \cat \uppd)(-, \mathbf{b})$ is finitely-generated projective as a $\cat \ppd \op$-module.
\end{prop}

\begin{proof}
This  follows from Proposition \ref{prop:A_quadratic} (adapted to the right augmented context). 
\end{proof}

\begin{rem}
The underlying $\cat \ppd$-bimodule of $\rgrad \cat \uppd$ is isomorphic to 
$
\cat \ppd \ofb \triv
$.
The product  on $\grad \cat \uppd$ can be described with respect to this isomorphism by giving the structure maps
\begin{eqnarray}
\label{eqn:struct_grad_cat_uppd}
( \triv_m \ofb \cat \ppd ) \otimes _\fb (\triv_n \ofb  \cat \ppd ) 
\rightarrow 
\triv_{m+n} \ofb  \cat \ppd 
\end{eqnarray}
for each $m,n \in \nat$. 

The left hand side is isomorphic to $ \triv_m\ofb   \big( \cat \ppd \otimes_\fb ( \triv_n\ofb \cat \ppd ) \big)$ and the $\cat \ppd$-module structure of $(\cat \ppd \ofb \triv_n)$ gives the map
\[
\triv_m\ofb   \big( \cat \ppd \otimes_\fb ( \triv_n\ofb \cat \ppd ) \big)
\rightarrow 
\triv_m \ofb  \big( \cat \ppd \ofb \triv_n \big) 
\cong 
\triv_m \ofb \triv_n \ofb \cat \ppd .
\]
Then, composing with the map induced by the product of $\triv$
\[
\triv_m \ofb \triv_n \ofb \cat \ppd 
\rightarrow 
 \triv_{m+n} \ofb \cat \ppd 
\]
gives the desired map (\ref{eqn:struct_grad_cat_uppd}).
\end{rem}

\subsection{The (right) quadratic dual of $\rgrad \cat \uppd$}
\label{subsect:rgrad_cat_uppd_qd}

In this section we describe the (right) quadratic dual $(\rgrad \cat \uppd)\qd$  of  $\rgrad \cat \uppd$ relative to $\cat \ppd$. To do so, we first describe the generating $\cat \ppd$-bimodule, using the following analogue of Corollary \ref{cor:left_cat_ppd}:

\begin{lem}
\label{lem:right_cat_ppd_obf}
For a $\cat \ppd\op$-module $M$ and any $\kk \fb\op$-module $X$, $M \obf X $ is naturally a $\cat \ppd\op$-module with structure map 
\[
(M \obf X) \otimes_\fb \cat\ppd 
\cong 
(M \otimes _\fb \cat \ppd) \obf (X \otimes_\fb \cat \ppd)  
\rightarrow 
M \obf X,
\]
where the isomorphism is given by Lemma \ref{lem:otimes_cat_opd} and the map  by the structure map $M \otimes _\fb \cat \ppd\rightarrow M$ on the first term and that induced by the canonical augmentation $\cat \ppd \rightarrow \bmu $ on the second.

If $M$ is a $\cat \ppd$-bimodule, then this, together with the obvious $\cat \ppd $-module structure, makes $M \obf X$ into a $\cat \ppd$-bimodule. 
\end{lem}

\begin{proof}
The proof is analogous to that of Corollary \ref{cor:left_cat_ppd}. (Indeed, a more general statement holds, using the more well-known analogue of Proposition \ref{prop:cat_ppd_modules_ofb} for $\cat \ppd\op$-modules with respect to $\obf$.)
\end{proof}

\begin{exam}
\label{exam:visualize_cat_ppd_sgn}
The key examples here are provided by the $\cat \ppd$-bimodules $\cat \ppd \obf \sgn_n^{\fb\op}$, for $n \in \nat$, where the bimodule structure is provided by Lemma \ref{lem:right_cat_ppd_obf} using the canonical bimodule structure of $\cat \ppd$. 

It can be useful to visualize $\cat \ppd \obf \sgn_n^{\fb\op}$ (omitting the $^{\fb\op}$ for clarity) as follows:

\ 

\begin{tikzpicture}[scale=0.5]
\draw [thick] (0,1)--(5,1)--(5,-1) -- (0,-1) -- (0,1);
\draw (1,1) -- (1,2);
\node at (2.5,1.5) {\ldots};
\draw (4,1) -- (4,2);
\node at (2.5,0) {$\cat \ppd$};
\draw (1.5, -1) -- (1.5, -2); 
\node at (2.5,-1.5) {\ldots};
\draw (3.5, -1) -- (3.5, -2); 
\begin{scope}
    \clip (5.8,1) rectangle (9.2,-.6);
    \draw  [thick] (7.5,1) circle(1.5);
\end{scope}
\draw  [thick] (6,1) -- (9,1);
\node at (7.5,.3) {$\sgn_n$}; 
\draw (6.5,1) -- (6.5,2);
\node at (7.5, 1.5) {\ldots}; 
\draw (8.5,1) -- (8.5,2);
\end{tikzpicture}

\noindent 
in which the `inputs' (corresponding to the $\fb\op$-structure) are at the top. Thus, the $\obf$ is interpreted as a form of `horizontal composition', analogous to the symmetric monoidal structure of $\cat \ppd$.

The $\cat \ppd$-bimodule structure of Lemma \ref{lem:right_cat_ppd_obf} can then be visualized using this, equipping  $\sgn_n^{\fb\op}$ with the `trivial' $\cat \ppd \op$-module structure.
\end{exam}

\begin{lem}
\label{lem:dual_generators}
The dual $\cat \ppd$-bimodule $\hom_{\cat \ppd\op} (\kk \finj^{[1]} \otimes_\fb \cat \ppd , \cat \ppd)$ is isomorphic to $\cat \ppd \obf \sgn_1^{\fb\op}$, equipped with the $\cat \ppd$-bimodule structure given by Lemma \ref{lem:right_cat_ppd_obf}.

Similarly, applying $\hom_{\cat \ppd\op} (- , \cat \ppd)$  to the short exact sequence of $\cat \ppd$-bimodules:
\[
0
\rightarrow 
\sgn_2 \ofb \cat \ppd 
\rightarrow 
\kk \sym_2 \ofb \cat \ppd 
\rightarrow 
\triv_2 \ofb \cat \ppd 
\rightarrow 
0
\]
yields the short exact sequence of $\cat \ppd$-bimodules:
\[
0
\rightarrow 
\cat \ppd \obf \triv_2 ^{\fb\op}
\rightarrow 
\cat \ppd \obf \kk \sym_2 ^{\fb\op}
\rightarrow 
\cat \ppd \obf \sgn_2    ^{\fb\op}
\rightarrow 
0
\]
equipped with the $\cat \ppd$-bimodule structures of Lemma \ref{lem:right_cat_ppd_obf}.
\end{lem}

\begin{proof}
This is a direct verification,  checking that one  obtains the asserted $\cat \ppd$-bimodule structures.  
\end{proof}

The general theory of Section \ref{subsect:quad_dual_grad_A} applies to determine the right quadratic dual of $\rgrad \cat \uppd$. In particular, Proposition \ref{prop:dual_quadratic_datum_R_C} gives the dual quadratic datum and Proposition \ref{prop:dual_C_R} shows that the underlying $\kk \fb$-bimodule is isomorphic to 
$
\cat \ppd \otimes_\fb (\kk \finj)\qd$.
 
Proposition  \ref{prop:right_qdual_kk_finj} identifies the right quadratic dual $\kk \finj\qd$  of $\kk \finj$ as $(\kk \finj^\ddag) \op$.
Then, using Lemma \ref{lem:underlying_kkfb-bimodules}, one deduces the isomorphism of the underlying $\cat \ppd$-bimodules:
\[
\cat \ppd \otimes_\fb (\kk \finj)\qd
\cong 
\cat \ppd \obf \sgn^{\fb\op}.
\]
The right hand side corresponds to the $\cat \ppd$-bimodules considered in Example \ref{exam:visualize_cat_ppd_sgn}.

\begin{prop}
\label{prop:quad_dual_grad_cat_uppd}
The right quadratic dual $(\rgrad \cat \uppd)\qd$ of $\rgrad \cat \uppd$ over $\cat \ppd$ is given by the quadratic datum:
\[
(\cat \ppd; \cat \ppd \obf \sgn_1^{\fb\op} , \cat \ppd \obf \triv_2^{\fb\op} ),
\]
in which the $\cat \ppd$-bimodule structures are those of Lemma \ref{lem:right_cat_ppd_obf}.

This quadratic dual  is isomorphic to the cohomologically graded algebra $\cat \ppd \obf \sgn^{\fb\op}$, where $\sgn_c^{\fb\op}$ has cohomological degree $c$, and with product (writing $\sgn$ for $\sgn^{\fb\op}$ for typographical clarity):
\[
\xymatrix{
(\cat \ppd \obf \sgn) \otimes _\fb (\cat \ppd \obf \sgn) 
\ar[d]_\cong 
\\
\big((\cat \ppd \obf \sgn) \otimes _\fb \cat \ppd\big)  \obf \sgn
\ar[r] 
&
 \cat \ppd \obf \sgn 
 \obf \sgn
 \ar[r] 
& 
 \cat \ppd \obf \sgn,
}
\]
where the first map is given by the $\cat \ppd\op$-module structure of $\cat \ppd \obf \sgn^{\fb\op} $ and the second by the product of $\sgn^{\fb\op}$.

Moreover, for any $a \in \nat$, $(\cat \ppd \obf \sgn^{\fb\op}) (\mathbf{a},-) $ is projective as a $\cat \ppd$-module.  
\end{prop}

\begin{proof}
The identification of the dual quadratic datum is given by Lemma \ref{lem:dual_generators} in conjunction with Proposition \ref{prop:dual_quadratic_datum_R_C}. The rest of the statement follows from the general results, using the identifications given before the statement of the Proposition.
\end{proof}

\begin{rem}
The reader is encouraged to give a diagrammatic interpretation of the product of $(\rgrad \cat \uppd)\qd$ using the visualization proposed in Example \ref{exam:visualize_cat_ppd_sgn}.
\end{rem}

By Theorem \ref{thm:DG_split_R_C}, $(\rgrad \cat\uppd)\qd $  has a DG structure:

\begin{thm}
\label{thm:DG_qdual_grad_cat_uppd}
The $\nat$-graded quadratic $\kk$-linear category $(\rgrad \cat\uppd)\qd \cong \cat \ppd \obf \sgn^{\fb\op} \cong \cat \ppd \otimes_\fb \kk \finj\qd $ has a DG structure with differential $d$ that is zero on $\kk \finj\qd$ and on $\cat\ppd$ is given by 
\[
\cat \ppd \rightarrow \hom_{\kk \fb\op} (\kk \finj^{[1]} , \cat \ppd) \cong \cat \ppd \obf \sgn_1^{\fb\op}
\]
that is adjoint to the $\kk \finj\op$-structure map 
\[
\cat \ppd \otimes_\fb \kk \finj^{[1]} \rightarrow \cat \ppd.
\]

In particular, the underlying complex of $\kk \fb$-bimodules is isomorphic to the Koszul complex for the $\kk \finj\op$-structure of $\cat \ppd$.
\end{thm}

We have the following finiteness property:

\begin{prop}
\label{prop:rgrad_finiteness}
The underlying (graded) $\kk \fb$-bimodule of $(\rgrad \cat\uppd)\qd$ is non-positive and locally bounded.  
Explicitly, for $s, t\in \nat$, evaluated on $(\mathbf{s}, \mathbf{t})$, $(\rgrad \cat\uppd)\qd$ is concentrated in cohomological degrees $[0, s-t]$.
\end{prop}

\begin{proof}
The underlying (graded) $\kk \fb$-bimodule is isomorphic to $\cat \ppd \otimes_\fb \kk \finj\qd$ (with cohomological degree inherited from $\kk \finj\qd$), where both $\cat \ppd$ and $\kk \finj\qd$ are non-positive. The result thus follows from the general result Proposition \ref{prop:preserve_non-neg}.
\end{proof}

On passing to cohomology, one obtains:

\begin{thm}
\label{thm:cohom_rgrad_qd}
The cohomology $H^* ((\rgrad \cat \uppd)\qd)$ is an $\nat$-graded $\kk$-linear category with respect to the cohomological grading. In particular:
\begin{enumerate}
\item 
$H^0 ((\rgrad \cat \uppd)\qd)$ is a $\kk$-linear subcategory of $\cat \ppd$; 
\item 
for each $n \in \nat$, $H^n((\rgrad \cat \uppd)\qd)$ is a bimodule over $ H^0 ((\rgrad \cat \uppd)\qd)$.
\end{enumerate}

Hence, cohomology induces a functor
$$
H^*(-)
\cn 
(\rgrad \cat \uppd)\qd\dash \dgmod 
\longrightarrow 
H^* ((\rgrad \cat \uppd)\qd)\dash\modules,
$$ 
in which the codomain denotes graded modules. In particular, for each $n \in \zed$, $H^n(-)$ takes values in the category of $H^0 ((\rgrad \cat \uppd)\qd)$-modules.
\end{thm}

\subsection{Relative nonhomogeneous  Koszul duality using $\vk (\cat \uppd)$}
\label{subsect:rel_nonhom_Koszul_vk}

The  appropriate dualizing complex is
\[
\vk (\cat \uppd) 
=
 \cat \uppd \otimes_{\cat \ppd} (\rgrad \cat \uppd)\qd,
\] 
equipped with the differential induced by inner multiplication by the class $\e$, as in Section \ref{subsect:dualizing_cx_kk_finj}. 

We have the following finiteness property:

\begin{prop}
\label{prop:finiteness_prop_vk_cat_uppd}
The underlying cohomologically-graded $\kk \fb$-bimodule of $\vk (\cat \uppd)$ is isomorphic to $\cat \uppd \otimes_\fb \kk \finj\qd$, with cohomological grading induced by that of $\kk \finj \qd$.

In particular, for $s, t \in \nat$, the underlying complex of $\kk (\sym_s \op \times \sym_t)$-modules of $\vk (\cat \uppd) (\mathbf{s}, \mathbf{t})$ is concentrated in cohomological degrees $[0,s]$ and each term is finitely-generated.
\end{prop}

By construction, $\vk (\cat \uppd)$ is a complex of left $\cat \uppd$, right $((\rgrad \cat \uppd)\qd, d)$ bimodules  (henceforth we omit this differential from the notation). It induces the adjunction:
\begin{eqnarray}
\label{eqn:adj_vk_cat_uppd}
\\
\nonumber
  \vk (\cat \uppd) \otimes_{(\rgrad \cat \uppd)\qd} - 
  \cn 
(\rgrad \cat \uppd)\qd \dash  \dgmod 
\rightleftarrows 
 \cat \uppd \dash  \dgmod 
\cn 
\ihom_{ \cat \uppd } (\vk (\cat \uppd), -) .
\end{eqnarray}
The underlying functors to $\kk\fb$-modules identify as 
\begin{eqnarray*}
  \vk (\cat \uppd) \otimes_{(\rgrad \cat \uppd)\qd} - 
  &\cong & 
  \kk \finj \otimes_\fb - 
  \\
  \ihom_{ \cat \uppd } (\vk (\cat \uppd), -) 
  &\cong & 
\ihom_{ \kk \fb } (\kk \finj\qd, -) ,
\end{eqnarray*}
using the appropriate twisted differentials on the right hand side. 

\begin{rem}
We focus upon the adjunction for left DG modules; there is also one  for right DG modules. 
\end{rem}

Proposition \ref{prop:exactness_vk_finj_left_modules} carries over to this relative framework:

\begin{prop}
\label{prop:exactness_vk_cat_uppd_left_modules}
The functors $\vk (\cat \uppd)  \otimes_{(\rgrad \cat \uppd)\qd} - $ and   $\ihom_{ \cat \uppd } (\vk (\cat \uppd), -) $ are exact and commute with colimits. Moreover, both these functors preserve weak equivalences.
\end{prop}

\begin{proof}
The first statement is proved by the same argument as for  Proposition \ref{prop:exactness_vk_finj_left_modules}. 

For the preservation of the weak equivalences, one uses the following  facts. The twisted differential of $\kk \finj \otimes_\fb -$ only depends on the underlying $\kk \finj \qd$-module of the $
(\rgrad \cat \uppd)\qd$-module to which it applies; likewise, the twisted differential of $\ihom_{ \kk \fb } (\kk \finj\qd, -)$ only depends on the underlying $\kk \finj$-module of the $\cat \uppd$-module to which it is applied (cf. the  identifications in Section \ref{sect:adjoints_bis}). Hence, the result follows from Proposition \ref{prop:exactness_vk_finj_left_modules}.
\end{proof}

Proposition \ref{prop:restrict_vk_kk_finj_adj_lbd} has the counterpart:

\begin{prop}
\label{prop:adj_vk_cat_uppd_lb}
The adjunction (\ref{eqn:adj_vk_cat_uppd}) restricts to the adjunction for locally bounded objects
$$
 \vk (\cat \uppd) \otimes_{(\rgrad \cat \uppd)\qd} - 
  \cn 
(\rgrad \cat \uppd)\qd \dash  \lbdgmod 
\rightleftarrows 
 \cat \uppd \dash  \lbdgmod 
\cn 
\ihom_{ \cat \uppd } (\vk (\cat \uppd), -) .
$$
\end{prop}

For $X$ a DG $(\rgrad \cat \uppd)\qd$-module, (respectively $Y$ a DG $\cat \uppd$-module) one has the adjunction unit (respectively counit):
\begin{eqnarray}
\label{eqn:adj_unit_vk_cat_uppd}
X \rightarrow  \ihom_{ \cat \uppd } (\vk (\cat \uppd), \vk (\cat \uppd) \otimes_{(\rgrad \cat \uppd)\qd} X)
\\
\label{eqn:adj_counit_vk_cat_uppd}
\vk (\cat \uppd) \otimes_{(\rgrad \cat \uppd)\qd} \ihom_{\cat \uppd} (\vk (\cat \uppd), Y)
\rightarrow  Y.
\end{eqnarray}
Retaining only the underlying DG $\kk \fb$-module structure, the underlying map identifies with (\ref{eqn:left_X_unit_kk_finj!}) (respectively \ref{eqn:left_X_counit_kk_finj}). (This is the analogue of Remark \ref{rem:adj_unit_counit_R_C_case}.)

Hence, the proof of Theorem \ref{thm:BGG} implies the following:

\begin{thm}
\label{thm:adj_unit_vk_cat_uppd}
Suppose that $\kk$ is a field of characteristic zero.
\begin{enumerate}
\item 
For $X$ a DG $(\rgrad \cat \uppd)\qd$-module that is locally bounded, the adjunction unit (\ref{eqn:adj_unit_vk_cat_uppd}) is a weak equivalence.
\item 
For $Y$ a DG $\cat \uppd$-module that is locally bounded, the adjunction counit (\ref{eqn:adj_counit_vk_cat_uppd}) is a weak equivalence.
\end{enumerate}
\end{thm}

Again, in the spirit of  Positselski's general results \cite{MR4398644}, one deduces an equivalence at the level of the associated  homotopy categories, denoted by $\holbd (-)$ (generalizing Notation \ref{nota:holbd_kk_finj}):

\begin{cor}
\label{cor:relative_BGG}
Suppose that $\kk$ is a field of characteristic zero. 
The adjunction of Proposition \ref{prop:adj_vk_cat_uppd_lb} induces an equivalence of categories:
 \[
\holbd ((\rgrad \cat \uppd)\qd ) 
\stackrel{\simeq}{\rightleftarrows}
\holbd (\cat \uppd ).
\]
\end{cor}

\section{Absolute Koszul duality for binary quadratic operads}
\label{sect:prop_bin_quad_abs}

The purpose of this Section is to review Koszul duality for binary quadratic operads (see \cite{MR2954392}, for example) in a form suitable for generalization to the  relative case in  Section \ref{sect:prop_left}. This  is  closely related to the original approach by Ginzburg and Kapranov  \cite{MR1301191}. 

Throughout,  $\kk$ is taken to be a field. Moreover, the following is assumed to hold:

\begin{hyp}
\label{hyp:ppd_bin_quad}
The operad $\ppd$  is a finitely-generated  binary quadratic operad. In particular, $\ppd$ is generated by $\ppd (2)$ (which is a finitely-generated $\kk \sym_2\op$-module) and the relations are in arity $3$.
\end{hyp}

This hypothesis implies that $\ppd$ is reduced, is canonically augmented, and that $\dim \ppd (t) < \infty$ for all $t \in \nat$.

\subsection{Analysing $\cat \ppd$}

Hypothesis \ref{hyp:ppd_bin_quad} implies that $\cat \ppd$ is a homogeneous quadratic category over $\kk \fb$.
This is made explicit below, using the following weight grading (this is simply the negative of the grading introduced in Definition \ref{defn:Z-grade_fb-bimodules}):

\begin{defn}
\label{defn:wt_grading}
For $n \in \nat$, let $\cat \wt{n} \ppd$ be the $\kk \fb$-bimodule 
\[
\cat \wt{n} \ppd (\mathbf{a}, \mathbf{b}) = 
\left\{
\begin{array}{ll}
\cat  \ppd (\mathbf{a}, \mathbf{b})
& 
n= a-b 
\\
0 & \mbox{otherwise.}
\end{array}
\right.
\]
\end{defn}

Clearly $\cat \wt{0} \ppd$ is isomorphic to $\bmu\cong \kk \fb$  and 
$
\cat \ppd \cong \bigoplus_{n \in \nat} \cat \wt{n} \ppd 
$  
as $\kk \fb$-bimodules. The augmentation corresponds to the projection to weight zero.
Moreover, the multiplication respects the weight; namely it restricts to maps 
\[
\cat \wt {m} \ppd \otimes_\fb \cat \wt{n} \ppd  \rightarrow \cat \wt{m+n} \ppd .
\]

The homogeneous quadratic structure is summarized in the following statement:

\begin{prop}
\label{prop:cat_ppd_homog_quad}
The $\kk$-linear category $\cat \ppd$ is homogeneous quadratic over $\kk \fb$, associated to the quadratic datum
$ 
( \kk \fb ; \cat \wt{1} \ppd , I_{\cat \ppd} )
$,  
where $I_{\cat \ppd}$ is generated by the quadratic relations from  the operad $\ppd$ together with the quadratic relations arising in forming $\cat \ppd$.

Moreover, for each $a \in \nat$, $\cat \wt{1} \ppd  (\mathbf{a}, -) $, $\cat \wt{2} \ppd  (\mathbf{a}, -) $, and $I_{\cat \ppd}(\mathbf{a}, -)$ are finitely-generated projective $\kk \fb$-modules.
\end{prop}

\begin{proof}
The first statement is standard; for example, the result is contained within  \cite{MR1301191} (see \cite{MR4835394} for further details).

The projectivity statements follow from Lemma \ref{lem:case_opd_reduced}, which ensures that $\cat \wt{1} \ppd $ and $\cat \wt{2} \ppd$ are both free as $\kk \fb$-modules. From this, one deduces that $I_{\cat \ppd}$ is projective. The finite generation statements follow from the hypothesis that $\ppd (2)$ is a finitely-generated module.
\end{proof}

The $\kk \fb$-bimodule $\cat \wt{1} \ppd$ can be described explicitly in terms of $\ppd (2)$ (considered as a $\kk\fb\op$-module) using the description of $\cat \ppd$ in the reduced case given in Lemma \ref{lem:case_opd_reduced}. One has the following identification of   $\fsbase (-,-)$ from  Notation \ref{nota:fsbase} in `weight one':

\begin{lem}
\label{lem:fsbase_wt1}
For $n \in \nat$, $\fsbase (\mathbf{n+1}, \mathbf{n})$ is the set $\{ f_{i,j} \mid 1 \leq i < j \leq n+1 \}$, where $f_{i,j}$ is the unique surjection such that $ f(j)=i$ and $f_{i,j}$ restricted to $\mathbf{n+1} \backslash \{ j\}$ is order preserving.  
Explicitly:
\[
f_{i,j} (k) = \left\{ 
\begin{array}{ll}
k & k< j \\
i & k=j \\
k-1 &k >j.
\end{array}
\right.
\]
\end{lem}

By construction,  $\kk \fs (\mathbf{n+1}, \mathbf{n})$  is freely generated as a $\kk \sym_n$-module by $\{ [f_{i,j}] \mid 
1 \leq i < j \leq n+1 \}$.

We adopt the following shorthand notation:

\begin{nota}
\label{nota:ppd_ij}
For $ 1 \leq i < j \leq n+1$, let $\ppd (\{i,j \})$ denote the direct summand of $ \cat \wt{1} \ppd (\mathbf{n+1}, \mathbf{n}) $ indexed by $f_{i,j} \in \fsbase (\mathbf{n+1}, \mathbf{n})$. 
\end{nota}

\begin{lem}
\label{lem:cat_wt1_ppd}
For $n \in \nat$, there is an isomorphism of $\kk (\sym_{n+1}\op \times \sym_n)$-modules
\[
\cat \wt{1} \ppd (\mathbf{n+1}, \mathbf{n}) 
\cong
\kk \sym_n \otimes_\kk 
\bigoplus_{1 \leq i<j \leq n} 
\ppd (\{i, j\}) 
\]
where the $\kk \sym_{n+1}\op$-module structure on the right hand side is given by Lemma \ref{lem:case_opd_reduced}.
 In particular, as a $\kk \sym_n$-module, $\cat \wt{1} \ppd (\mathbf{n+1}, \mathbf{n})$ is freely-generated by  $\bigoplus_{1 \leq i<j \leq n} 
\ppd (\{i, j\}) $.
\end{lem}

\begin{proof}
This follows directly from Lemma \ref{lem:case_opd_reduced} in conjunction with the description of $\fsbase (\mathbf{n+1}, \mathbf{n})$ given in Lemma \ref{lem:fsbase_wt1}. 
\end{proof}

The finitely-generated projectivity property of Proposition \ref{prop:cat_ppd_homog_quad} allows us to apply (left) quadratic duality as in Section \ref{sect:qdual}.

\begin{cor}
\label{cor:quad_dual_cat_ppd}
The left quadratic dual $(\cat \ppd)\qd$ is the quadratic algebra associated to the right quadratic datum:
\[
(\kk \fb ; \hom_{\kk \fb} (\cat \wt{1} \ppd, \kk \fb) , \hom_{\kk \fb} (\cat\wt{2} \ppd, \kk \fb) ).
\]
Moreover, there are isomorphisms of $\kk \fb$-bimodules:
\begin{eqnarray*}
\hom_{\kk \fb} (\cat \wt{1} \ppd, \kk \fb) 
&\cong &
(\cat \wt{1} \ppd)^\sharp \\
\hom_{\kk \fb} (\cat\wt{2} \ppd, \kk \fb)
&\cong &
(\cat \wt{2} \ppd)^\sharp.
\end{eqnarray*}

For each $a \in \nat$, the $\kk \fb\op$-modules $\hom_{\kk \fb} (\cat \wt{1} \ppd, \kk \fb) (-,\mathbf{a})$ and $\hom_{\kk \fb} (\cat\wt{2} \ppd, \kk \fb) (-,\mathbf{a})$ are finitely-generated projective.
\end{cor} 

\begin{proof}
The first statement follows immediately from Proposition \ref{prop:cat_ppd_homog_quad}, by the definition of quadratic duality. 
 The bimodule identifications follow from the general result for duality, Lemma \ref{lem:duality}.
\end{proof}

\begin{rem}
\label{rem:cat_wt1_ppd_dual}
Lemma \ref{lem:cat_wt1_ppd} allows the generating $\kk \fb$-bimodule $(\cat \wt{1} \ppd)^\sharp$ to be identified explicitly. 
In particular, for $n \in \nat$, $(\cat \wt{1} \ppd)^\sharp (\mathbf{n},\mathbf{n+1})$ is freely generated as a $\kk \sym_n\op$-module by 
$
\bigoplus_{1 \leq i<j \leq n} 
\ppd (\{i, j\})^\sharp.
$
\end{rem}

The following is clear from the construction:

\begin{prop}
\label{prop:cat_ppd_!_non-negative}
The underlying (graded) $\kk \fb$-bimodule of $(\cat \ppd)\qd$ is non-negative and takes values of (total) finite dimension.
 Moreover, for $n \in \nat$, $\big((\cat \ppd)\qd\big)\dg{n}$ is concentrated in cohomological degree $n$.
\end{prop}

\subsection{The  dualizing complex $K^\vee (\cat \ppd)$}

In this context, working over $\kk\fb$,   the appropriate dualizing complex is $K^\vee (\cat \ppd)$, as in Section \ref{sect:kcx}. This has underlying object
\[
K^\vee (\cat \ppd)= (\cat \ppd)\qd \otimes_\fb (\cat \ppd) 
\]
as a $(\cat \ppd)\qd\boxtimes (\cat \ppd)\op$-module, where
the cohomological degree is inherited from  $(\cat \ppd)\qd$.

The differential is induced by the inner multiplication with the class $e$ (cf. Notation \ref{nota:e_e_prime});  this is given by the natural transformation 
\begin{eqnarray}
\label{eqn:e_cat_ppd}
\bmu \rightarrow (\cat \wt{1} \ppd)^\sharp \otimes_\fb \cat \wt{1} \ppd 
\end{eqnarray}
that is described explicitly as follows. 

\begin{nota}
Fix $\{ \zeta^\ell \mid \ell \}$ a (finite) $\kk$-vector space basis for $\ppd (2)$. 
\begin{enumerate}
\item 
For natural numbers $1 \leq i<j$, denote by  $\{ \zeta^\ell _{i,j} \mid \ell \}$ the corresponding $\kk$-vector space  basis for $\ppd (\{i, j\})$ induced by the bijection $\mathbf{2} \cong \{i,j\}$ that sends $1 \mapsto i$. 
\item 
Denote by $\{ (\zeta^\ell _{i,j})^\sharp \mid \ell \}$ the dual $\kk$-vector space  basis for $\ppd (\{i, j\})^\sharp$.
\end{enumerate}
\end{nota}

Unravelling the definitions, one has:

\begin{lem}
\label{lem:e_kv_cat_ppd}
For $n \in \nat$, the map (\ref{eqn:e_cat_ppd}) evaluated on $(\mathbf{n+1}, \mathbf{n+1})$ 
 identifies with the $\kk \sym_{n+1}$-bimodule map
\begin{eqnarray*}
\kk \sym_{n+1} 
&\rightarrow &
 (\cat \wt{1} \ppd)^\sharp (\mathbf{n}, \mathbf{n+1})\otimes_{\sym_n} \cat \wt{1} \ppd 
(\mathbf{n+1}, \mathbf{n})
\\
\ [e] 
&\mapsto &  
\sum _{\substack{\ell \\ 1 \leq i<j\leq n+1} }
(\zeta^\ell _{i,j})^\sharp \otimes \zeta^\ell_{i,j} .
\end{eqnarray*}
\end{lem}

We have the following finiteness property, established using Proposition \ref{prop:bimod_non-neg_non-pos}:

\begin{prop}
\label{prop:finiteness_kv_cat_ppd}
The complex $K^\vee (\cat \ppd)$ evaluated on $(\mathbf{s}, \mathbf{t})$, for $s, t \in \nat$, is concentrated in cohomological degrees $[t-\min \{s,t\}, t]$. It is finite-dimensional in each cohomological degree.
\end{prop}

The  complex $K^\vee (\cat \ppd)$ yields the adjunction:
\begin{eqnarray}
\label{eqn:adjunction_kv_cat_ppd}
K^\vee (\cat \ppd) \otimes_{\cat \ppd} - 
\cn 
\cat \ppd \dash \dgmod
\rightleftarrows 
(\cat \ppd)\qd \dash \dgmod
\cn 
\ihom_{(\cat \ppd)\qd} (K^\vee (\cat \ppd), -).
\end{eqnarray}

Moreover, there are natural isomorphisms:
\begin{eqnarray*}
K^\vee (\cat \ppd) \otimes_{\cat \ppd} -  & \cong & (\cat \ppd)\qd \otimes^e_\fb - \\
\ihom_{(\cat \ppd)\qd} (K^\vee (\cat \ppd), -)&\cong & \ihom_{\kk \fb}^e (\cat \ppd, -) ,
\end{eqnarray*}
where the superscript $e$ indicates that the differential is twisted using $e$. 

\begin{rem}
\label{rem:cat_ppd_non-pos}
This identification ensures that the underlying functors (with values in  $\kk \fb$-modules) behave well, since $(\cat \ppd)\qd$ is non-negative and $\cat \ppd$ is non-positive, using the results of Section \ref{sect:finiteness_properties}.
\end{rem}

The following is the counterpart of Proposition \ref{prop:exactness_vk_finj_left_modules}:

\begin{prop}
\label{prop:exact_colimit_cat_P}
The functors $K^\vee (\cat \ppd) \otimes_{\cat \ppd} - $ and $\ihom_{(\cat \ppd)\qd} (K^\vee (\cat \ppd), -)$ are exact and commute with colimits. Moreover, both functors preserve weak equivalences.
\end{prop}

\begin{proof}
Exactness is a consequence of the fact that $\cat \ppd$ is projective as a $\kk \fb$-module and $(\cat \ppd)\qd$ is projective as a $\kk \fb\op$-module; this property of $(\cat \ppd)\qd$ is proved similarly to the case of $\cat \ppd$ (see Lemma \ref{lem:case_opd_reduced}).

The fact that $K^\vee (\cat \ppd)\otimes_{\cat \ppd} - $ preserves colimits is clear; for $\ihom_{(\cat \ppd)\qd} (K^\vee (\cat \ppd), -)$ one uses Proposition \ref{prop:cat_ppd_!_non-negative}. 

The proof that these functors preserves weak equivalences is similar to that of Proposition \ref{prop:exactness_vk_finj_left_modules}, based on the fact that $\cat \ppd$ is non-positive and $(\cat \ppd)\qd$ is non-negative (cf. Remark \ref{rem:gen_colimits_exact_we}).
\end{proof}

For a DG $\cat \ppd$-module $M$, one has the adjunction unit:
\begin{eqnarray}
\label{eqn:adj_unit_cat_ppd_absolute}
M 
\rightarrow \ihom_{\kk \fb}^e (\cat \ppd, (\cat \ppd)\qd \otimes^e_\fb M),
\end{eqnarray}
and, for a DG $(\cat \ppd)\qd$-module $N$, the adjunction counit:
\begin{eqnarray}
\label{eqn:adj_counit_cat_ppd_absolute}
(\cat \ppd)\qd \otimes^e_\fb \ihom_{\kk\fb}^e (\cat \ppd,N) \rightarrow N.
\end{eqnarray}

Now, taking $M = \bmu$ and $N= \bmu$, considered respectively as a $\cat \ppd$-module (resp. a $(\cat \ppd)\qd$-module) via the canonical augmentations, these yield the natural maps:
\begin{eqnarray}
\label{eqn:ppd_unit}
\bmu
&\rightarrow  & \ihom_{\kk\fb}^e (\cat \ppd, (\cat \ppd)\qd)\cong (\cat \ppd)^\sharp \otimes_\fb^e (\cat \ppd)^\sharp
\\
\label{eqn:ppd_counit}
(\cat \ppd)\qd \otimes^e_\fb (\cat \ppd)^\sharp &\rightarrow &\bmu
\end{eqnarray}
which, forgetting structure, can be considered as maps of DG $\kk \fb$-modules.

The main structure theorem (see \cite{MR2954392} for example) for operadic Koszul duality (in the binary quadratic case) can be restated as:

\begin{thm}
\label{thm:bin_quad_operadic_duality}
Suppose that $\kk$ is a field of characteristic zero. The following conditions are equivalent: 
\begin{enumerate}
\item
the operad $\ppd$ is Koszul;
\item 
the unit (\ref{eqn:ppd_unit}) is a weak equivalence; 
\item 
the counit (\ref{eqn:ppd_counit}) is a weak equivalence.
\end{enumerate}
\end{thm}

\begin{rem}
\label{rem:cat_ppd_Koszul}
The  conditions of Theorem \ref{thm:bin_quad_operadic_duality} are also equivalent to the condition that $\cat \ppd$ is a Koszul  category over $\kk \fb$. 
\end{rem}

The arguments used in establishing the BGG correspondence in Section \ref{subsect:finj_BGG} carry over to this context. For example, one has:

\begin{thm}
\label{thm:absolute_lbd_equivalence}
The adjunction (\ref{eqn:adjunction_kv_cat_ppd}) restricts to the adjunction for locally bounded objects:
\[
K^\vee (\cat \ppd) \otimes_{\cat \ppd} - 
\cn 
\cat \ppd \dash \lbdgmod
\rightleftarrows 
(\cat \ppd)\qd \dash \lbdgmod
\cn 
\ihom_{(\cat \ppd)\qd} (K^\vee (\cat \ppd), -).
\]

Suppose that $\kk$ is a field of characteristic zero. Then  $\ppd$ is Koszul  if and only if this adjunction induces an equivalence between the respective homotopy categories:
\[
\holbd(\cat \ppd)
\stackrel{\simeq}{\rightleftarrows}
\holbd((\cat \ppd)\qd).
\]
\end{thm}

\subsection{Desuspending $(\cat \ppd)\qd$}
\label{subsect:desusp_cat_ppd_qd}

Since the cohomological degree of $(\cat \ppd)\qd$ and its canonical bimodule grading coincide (see Proposition \ref{prop:cat_ppd_!_non-negative}), following the  usual practice in operad theory, one desuspends.  For this, one uses the sheering functor $(-)\desusp$ reviewed in Section \ref{sect:sheer}, 
 Namely, one considers 
$
\big( (\cat \ppd)\qd \big) \desusp. 
$ 
This is a unital monoid in $\kk \fb$-bimodules that is concentrated in cohomological degree zero. More precisely, it is quadratic, associated to the Koszul dual  operad $\ppd^!$ of $\ppd$ (up to the functor $(-)\op$):

\begin{prop}
\label{prop:desusp_cat_ppd_qd}
For $\ppd$ a binary quadratic operad, there is an isomorphism of $\kk$-linear categories:
\[
\big( (\cat \ppd)\qd \big) \desusp
\cong 
(\cat \ppd^!)\op.
\]
\end{prop}

\begin{rem}
For this statement, $\ppd$ is not required to be   Koszul. 
\end{rem}

\begin{exam}
One has the standard examples:
\begin{enumerate}
\item
For $\ppd = \com$, one has $\big( (\cat \com)\qd \big) \desusp \cong (\cat \lie)\op$.
\item 
For $\ppd = \lie$, one has $\big( (\cat \lie)\qd \big) \desusp  \cong (\cat \com)\op$.
\end{enumerate}
The operads $\com$ (encoding commutative, associative algebras) and $\lie$ (encoding Lie algebras) are both binary quadratic Koszul.
\end{exam}

\section{Relative nonhomogeneous Koszul duality for $\cat \uppd$ over $\kk \finj$}
\label{sect:prop_left}

The purpose of this section is to treat  duality relative to $\kk \finj$ when considering $\cat \uppd$, building upon the absolute theory recalled in Section \ref{sect:prop_bin_quad_abs}.
 Throughout $\kk$ is a field and $\uppd$ is as in Section \ref{subsect:uppd_framework}, in addition supposing that Hypothesis \ref{hyp:ppd_bin_quad} holds. In  particular $\ppd$ is binary quadratic, finite-dimensional in each arity.

In Theorem \ref{thm:grad_cat_uppd_dual} the quadratic dual category $(\grad\uppd)\qd$ is described and the dualizing complex  $K^\vee (\cat \uppd)$ is described in Section
\ref{subsect:dualizing_Kvee_cat_uppd}. The main result of the section is Theorem \ref{thm:relative_lbd_equivalence_cat_ppd}, which gives the associated equivalence of locally bounded homotopy categories.

\subsection{Filtering $\cat \uppd$}

We consider the  filtration of $\cat \uppd$ induced by the weight grading of $\cat \ppd$.

Recall  from Lemma \ref{lem:decompose_cat_uppd} that the product in $\cat \uppd$ induces the isomorphism of $\kk\finj \boxtimes \cat\ppd\op$-modules:
\[
\kk \finj \otimes_\fb \cat \ppd \stackrel{\cong}{\rightarrow} \cat \uppd.
\]
By Proposition \ref{prop:left_augmentation}, the augmentation of $\cat \ppd$ provides the left augmentation $\cat \uppd \rightarrow \kk \finj$ and the weight grading of $\cat \ppd$ gives the filtration $F_\bullet \cat \uppd$, where 
\[
F_n \cat \uppd := \bigoplus_{t \leq n} \kk \finj \otimes_\fb \cat \wt{t}\ppd.
\]

\begin{lem}
\label{lem:filter_cat_uppd}
The filtration $F_\bullet \cat \uppd$ yields a filtration as a $\kk $-linear category; namely, for $m,n \in \nat$, the product of $\cat \uppd$ restricts to 
$
F_m \cat \uppd \otimes_\fb F_n \cat \uppd 
\rightarrow 
F_{m+n} \cat \uppd.
$ 
\end{lem}

\begin{proof}
This can be checked directly using the definition of the weight grading. (Alternatively, it can be checked by using the criterion of Proposition \ref{prop:filt_hyp}.)
\end{proof}

Passing to the associated graded $\kk$-linear category  $\grad \cat \uppd$, one has 
\[
\grad_n \cat \uppd \cong \kk \finj \otimes_\fb \cat \wt{n} \ppd, 
\]
where the right hand side is equipped with the associated $\kk \finj$-bimodule structure. The (left) $\kk \finj$-module structure is the obvious one. The right $\kk \finj$-module structure is given by the interchange map 
\[
\cat \wt{n} \ppd \otimes_\fb \kk \finj 
\rightarrow 
\kk \finj \otimes_\fb \cat \wt{n} \ppd
\]
 induced by the product of $\cat \uppd$.

\begin{rem}
\label{rem:interchange_map}
This interchange map can be described as follows; it depends only upon the $\kk \fb\op$-module structure of $\ppd$. 
 To make this explicit, for given natural numbers $a$, $b$, $c$ (where $b-c=n\geq 0$ and $a \leq b$) we consider the interchange map 
\[
\cat \ppd (\mathbf{b}, \mathbf{c}) \otimes_{\sym_b} \kk \finj (\mathbf{a}, \mathbf{b}) 
\rightarrow 
\kk \finj (- ,\mathbf{c}) \otimes_\fb \cat \ppd (\mathbf{a}, -) \cong \cat \uppd (\mathbf{a}, \mathbf{c})
\]
induced by the product of $\cat \uppd$, using the explicit description of $\cat \ppd$ given by Lemma \ref{lem:case_opd_reduced}.

Consider $i \in \finj (\mathbf{a}, \mathbf{b})$ and $f \in \fs (\mathbf{b} , \mathbf{c})$. Define $\mathbf{B}$, $i'$ and $f'$ 
(uniquely up to automorphism of $\mathbf{B}$) by the commutative diagram of maps of sets
\[
\xymatrix{
\mathbf{a} 
\ar@{^(->}[r] ^i
\ar@{->>}[d]_{f'}
&
\mathbf{b}
\ar@{->>}[d]^f
\\
\mathbf{B}
\ar@{^(->}[r]_{i'}
&
\mathbf{c}.
}
\]
Then $a -B =n$ if and only if $f|_{\mathbf{b}\backslash i (\mathbf{a}) }$ is injective.

Consider $\xi \in \bigotimes_{j \in \mathbf{c}}  \ppd (f^{-1}(j))$, which represents an element of $\cat \wt{n} \ppd$ indexed by $f$. Then, by definition of the weight grading,  the image of $\xi \otimes [i]$ in  $\kk \finj \otimes_\fb \cat \wt{n} \ppd$ is zero unless $f|_{\mathbf{b}\backslash i (\mathbf{a}) }$ is injective. 
 If the latter condition is satisfied, the above commutative diagram gives rise to a well-defined element 
\[
[i'] \otimes \xi' \in \kk \finj (\mathbf{B}, \mathbf{c}) \otimes_{\sym_B} \cat \wt{n} \ppd (\mathbf{a}, \mathbf{B}). 
\]
(The construction of $\xi'$ only depends on the $\kk \fb$-bimodule structure of $\cat \ppd$.)
 The element $[i'] \otimes \xi'$ is the image of $\xi \otimes [i]$ under the interchange map in this case.
\end{rem}

\subsection{Relative quadratic duality for $\grad \cat \uppd$}

Proposition \ref{prop:A_quadratic} implies that $\grad \cat \uppd$ is a homogeneous quadratic $\kk$-linear category (so that $\cat \uppd$ is nonhomogeneous quadratic).
 More precisely (using the notation $I_{\cat \ppd}$ from Proposition \ref{prop:cat_ppd_homog_quad}):

\begin{prop}
\label{prop:grad_cat_uppd_quadratic_left}
The $\kk$-linear category $\grad \cat \uppd$ is the homogeneous quadratic category over $\kk \finj$ associated to the right quadratic datum
$
(\kk \finj; \kk \finj \otimes_\fb \cat \wt{1} \ppd, \kk \finj \otimes_\fb I_{\cat \ppd}).
$
\end{prop}

The required left finitely-generated  projectivity properties hold, so that we can form 
 the (left) relative quadratic dual over $\kk \finj$:

\begin{thm}
\label{thm:grad_cat_uppd_dual}
The relative quadratic dual $(\grad \cat \uppd)\qd$  is the homogenous quadratic category associated to the  quadratic datum 
$
(\kk \finj; (\cat \wt{1} \ppd)^\sharp \otimes_\fb \kk \finj ,  (\cat \wt{2} \ppd)^\sharp \otimes_\fb \kk \finj ).
$  
It  is a DG $\kk$-linear category equipped with the differential given by Theorem \ref{thm:DG_posit}.

Moreover, $(\grad \cat \uppd)\qd$ has underlying $(\cat \ppd)\qd \boxtimes \kk \finj\op$-module 
$ (\cat \ppd)\qd \otimes _\fb \kk \finj.$
 The cohomological grading is inherited from that of $(\cat \ppd)\qd$.
\end{thm}

One has the following finiteness result that is analogous to Proposition \ref{prop:rgrad_finiteness}.

\begin{prop}
\label{prop:grad_cat_uppd_qd_finiteness}
The (graded) $\kk \fb$-bimodule underlying $(\grad \cat \uppd)\qd$ is non-negative and is of (total) finite type. Moreover, it is locally bounded: explicitly,  for $s, t \in \nat$, the complex $(\grad \cat \uppd)\qd(\mathbf{s}, \mathbf{t})$ is concentrated in cohomological degrees $[0, t-s]$.
\end{prop}

\begin{rem}
One can analyse the structure of $(\grad \cat \uppd)\qd$ starting from that of the underlying $\kk \finj$-bimodule, $ (\cat \ppd)\qd \otimes _\fb \kk \finj$. By Lemma \ref{lem:underlying_kkfb-bimodules}, there is an  isomorphism of the underlying $\kk \fb$-bimodules:
$$
\kk \finj  \cong  \bmu \ofb \triv^\fb
$$
that allows us to give a visualization analogous to that of Example \ref{exam:visualize_cat_ppd_sgn}, keeping the same orientation. For this it is important to remember that $(\cat \ppd)\qd (\mathbf{a}, \mathbf{b})$ is zero if $a>b$; it is generated as an algebra  over $\kk \fb$ by the terms $(\cat \ppd)\qd (\mathbf{a}, \mathbf{a+1})$ for $a \in \nat$.

The term corresponding to $(\cat \ppd) \qd \otimes_\fb (\bmu  \ofb \triv_n^\fb)$ can be represented by:

\ 

\begin{tikzpicture}[scale=0.5]
\draw [thick] (0,1)--(7,1)--(7,-1) -- (0,-1) -- (0,1);
\draw (1,1) -- (1,3.5);
\node at (2,2) {\ldots};
\draw (3,1) -- (3,3.5);
\node at (3.5,0) {$(\cat \ppd)\qd$};
\draw (1.5, -1) -- (1.5, -2.5); 
\node at (3.5,-1.5) {\ldots};
\draw (5.5, -1) -- (5.5, -2.5); 
\begin{scope}
    \clip (3.5,2) rectangle (7.5,4);
    \draw  [thick] (5.5,2) circle(1.5);
\end{scope}
\draw  [thick] (4,2) -- (7,2);
\node at (5.5,2.5) {$\triv_n$}; 
\draw (6.5,1) -- (6.5,2);
\node at (5.5, 1.5) {\ldots}; 
\draw (4.5,1) -- (4.5,2);
\end{tikzpicture}

It is instructive to describe both the $\kk \finj$-bimodule structure of $ (\cat \ppd)\qd \otimes _\fb \kk \finj$ in terms of this visual representation, as well as the algebra structure of $(\grad \cat \uppd)\qd$. 
\end{rem}

The differential of $(\grad \cat \uppd)\qd$ is given by Theorem \ref{thm:DG_posit} or, more precisely, by its specialization Theorem \ref{thm:DG_split_R_C}. This is determined by the restriction of the $\cat \ppd$-module structure given by Proposition \ref{prop:left_augmentation}
\[
\cat \ppd \otimes_\fb \kk \finj \rightarrow \kk \finj
\]
to weight one $\cat \wt{1} \ppd \subset \cat \ppd$. Namely, the adjoint 
\[
\kk \finj \rightarrow (\cat \wt{1} \ppd)^\sharp \otimes_\fb \kk \finj \subset (\cat \ppd)\qd \otimes_\fb \kk \finj
\]
gives the restriction of the differential to $\kk \finj$, which determines the differential by Theorem \ref{thm:DG_split_R_C}.

Since the differential is a derivation, it suffices to understand (for each $n \in \nat$) the corresponding map
\[
\kk \finj (\n, \mathbf{n+1}) 
\rightarrow 
(\cat \ppd (\mathbf{n+1}, \mathbf{n}) )^\sharp \otimes _{\sym_n} \kk \finj (\n, \n) \cong (\cat \ppd (\mathbf{n+1}, \mathbf{n}) )^\sharp.
\]

Since this is $\kk (\sym_n \times \sym_{n+1}\op)$-equivariant, one reduces to specifying the image of the generator $[\iota_{n,n+1}]$ for  $\iota_{n,n+1}: \n \subset \mathbf{n+1}$ the canonical inclusion. The identification that one obtains is analogous to that of Lemma \ref{lem:e_kv_cat_ppd} and uses the notation of that section. In particular, Lemma \ref{lem:cat_wt1_ppd} shows that $\cat \ppd (\mathbf{n+1}, \mathbf{n})$ is freely generated as a $\kk \sym_n$-module by the vector space with basis $\{ \zeta ^\ell_{i,j} \}$ (see {\em loc. cit.} for the indexing).

\begin{prop}
\label{prop:identify_diff_grad_cat_uppd_qd}
The  differential of $(\grad \cat \uppd)\qd$ is determined by its restrictions to $\kk \finj (\n, \mathbf{n+1})$ (for all $n \in \nat$). The map 
\[
\kk \finj (\n, \mathbf{n+1}) \rightarrow (\cat \ppd (\mathbf{n+1}, \mathbf{n}) )^\sharp
\]
is determined by its value on $[\iota_{n, n+1}]$ that is given by 
\[
\ 
[\iota_{n,n+1}] 
\mapsto 
\sum_{i=1}^n 
\sum_\ell 
(\zeta^\ell \circ \iota_{1,2}) (\zeta^\ell_{i,n+1})^\sharp,
\]
where $(\zeta^\ell \circ \iota_{1,2})\in \kk = \uppd (\mathbf{1})$ is given by the composition $\ppd (\mathbf{2}) \otimes \kk \finj (\mathbf{1}, \mathbf{2}) \rightarrow \uppd (\mathbf{1})$ arising from $\cat \uppd$.
\end{prop}

\begin{proof}
By construction, the differential restricted to $\kk \finj (\n, \mathbf{n+1})$ is adjoint to the module structure map
\[
\cat \ppd (\mathbf{n+1}, \mathbf{n}) \otimes_{\sym_{n+1}} \kk \finj (\n, \mathbf{n+1}) 
\rightarrow 
\kk \finj (\n, \n) \cong \kk \sym_n.
\]
This adjoint corresponds to 
\[
\kk \finj (\n, \mathbf{n+1}) \rightarrow 
\hom _{\sym_n} ( \cat \ppd (\mathbf{n+1}, \mathbf{n}), \kk \sym_n) 
\cong 
(\cat \ppd (\mathbf{n+1}, \mathbf{n}) )^\sharp.
\]

Recall that, for $H$ a finite group and $M$ a $\kk H$-module, the natural isomorphism $\hom_H (M, \kk H) \stackrel{\cong}{\rightarrow } \hom_\kk (M, \kk)$ is induced by composition with the $\kk$-linear map $\kk H \rightarrow \kk$ that sends $[h]$ to zero unless $h=e$, for which $[e] \mapsto 1$. Using this, it is straightforward to check that the image of $[\iota_{n,n+1}]$ is as given. 
\end{proof}

\subsection{The dualizing complex $K^\vee (\cat \uppd)$}
\label{subsect:dualizing_Kvee_cat_uppd}

The appropriate  dualizing complex is $K^\vee (\cat \uppd)$. By definition, the underlying bimodule is 
\[
K^\vee (\cat \uppd)= 
(\grad \cat \uppd)\qd
\otimes_{\finj} \cat \uppd.
\]
This is equipped with the differential that is defined (with respect to the right hand expression) by inner multiplication by the element $e'$ that is constructed from $e$ corresponding to the natural transformation (\ref{eqn:e_cat_ppd}). The latter was  identified explicitly in Lemma \ref{lem:e_kv_cat_ppd}.

\begin{rem}
The underlying (graded) $\kk \fb$-bimodule of $K^\vee (\cat \uppd)$ is isomorphic to $(\cat \ppd)\qd \otimes_\fb \cat \uppd$, with cohomological grading inherited from $(\cat \ppd)\qd$. Since $(\cat \ppd)\qd$ is non-negative; the general results of Section \ref{sect:finiteness_properties}  thus imply that, for $s, t \in \nat$,  $ K^\vee (\cat \uppd) (\mathbf{s}, \mathbf{t})$ is concentrated in cohomological degrees $[ 0, t]$.   
\end{rem}

The complex $K^\vee (\cat \uppd)$ yields the adjunction:
\begin{eqnarray}
\label{eqn:vk_cat_uppd_adjunction}
\\
\nonumber
K^\vee (\cat \uppd) \otimes_{\cat \uppd} - 
\cn
(\cat \uppd)\dash\dgmod 
\rightleftarrows 
(\grad \cat \uppd)\qd \dash \dgmod 
\cn
\ihom_{(\grad \cat \uppd)\qd} (K^\vee (\cat \uppd), -) .
\end{eqnarray}
Moreover, there are natural isomorphisms:
\begin{eqnarray*}
K^\vee (\cat \uppd) \otimes_{\cat \uppd} - 
& \cong & 
(\cat \ppd)\qd \otimes_\fb^{e} - \\
\ihom_{(\grad \cat \uppd)\qd} (K^\vee (\cat \uppd), -)
&\cong &
\ihom_{\kk \fb}^{e} (\cat \ppd, -)
\end{eqnarray*}
of the underlying functors. 

Since $(\cat \ppd)\qd$ is non-negative and $\cat \ppd$ is non-positive, the proof of Proposition \ref{prop:exact_colimit_cat_P} generalizes to give:

\begin{prop}
\label{prop:exact_colimit_cat_uppd}
The functors $K^\vee (\cat \uppd) \otimes_{\cat \uppd} - $ and $\ihom_{(\cat \uppd)\qd} (K^\vee (\cat \uppd), -)$ are exact and commute with colimits. Moreover, both functors preserve weak equivalences.
\end{prop}

Now, supposing that $\kk$ is a field of characteristic zero and that $\ppd$ is Koszul, the absolute Koszul duality result Theorem \ref{thm:absolute_lbd_equivalence} has the following counterpart:

\begin{thm}
\label{thm:relative_lbd_equivalence_cat_ppd}
The adjunction (\ref{eqn:vk_cat_uppd_adjunction}) restricts to the adjunction for locally bounded objects:
\[
K^\vee (\cat \uppd) \otimes_{\cat \uppd} - 
\cn
(\cat \uppd)\dash\lbdgmod 
\rightleftarrows 
(\grad \cat \uppd)\qd \dash \lbdgmod 
\cn
\ihom_{(\grad \cat \uppd)\qd} (K^\vee (\cat \uppd), -) .
\]

If $\kk$ is a field of characteristic zero and  $\ppd$ is Koszul, this induces an equivalence between the respective homotopy categories:
\[
\holbd(\cat \uppd)
\stackrel{\simeq}{\rightleftarrows}
\holbd((\grad \cat \uppd)\qd).
\]
\end{thm}

\subsection{Desuspending $(\grad \uppd)\qd$}
\label{subsect:desuspend_grad_uppd_qd}

As in Section \ref{subsect:desusp_cat_ppd_qd}, it is useful to apply the sheering functor $(-)\desusp$ (see Appendix \ref{sect:sheer}) to $(\grad\uppd)\qd$. This yields the DG category $\big((\grad\uppd)\qd\big)\desusp$. 

\begin{prop}
\label{prop:desusp_grad_cat_uppd_qd}
There is a commutative diagram of DG $\kk$-linear categories
\[
\xymatrix{
\kk \fb
\ar[r]
\ar[d]
&
\kk\finj \desusp
\ar[d]
\\
(\cat \ppd^!)\op
\ar[r]
&
\big((\grad\uppd)\qd\big)\desusp
}
\]
such that 
\begin{enumerate}
\item 
$\kk \fb$ and $(\cat \ppd^!)\op$ are concentrated in cohomological degree zero; 
\item
for $n \in \nat$, 
$(\kk \finj \desusp)\dg{n}$ is in cohomological degree $-n$;
\item 
$\big((\grad\uppd)\qd\big)\desusp$ is concentrated in non-positive cohomological degree.  
\end{enumerate}
Moreover, the multiplication induces an isomorphism of bimodules:
\[
(\cat \ppd^!)\op \otimes_\fb \kk \finj \desusp
\stackrel{\cong}{\rightarrow} 
\big((\grad\uppd)\qd\big)\desusp.
\]

The differential of $\big((\grad\uppd)\qd\big)\desusp$ acts by zero on $(\cat \ppd^!)\op$ and hence is determined by its restriction to $\kk \finj \desusp$.
\end{prop}

\begin{proof}
This follows from the structure of $(\grad \uppd)\qd$ using the general properties of the desuspension functor $(-)\desusp$ given in Appendix \ref{sect:sheer}, together with the identification of $\big((\cat \ppd)\qd\big)\desusp$ given in Proposition \ref{prop:desusp_cat_ppd_qd}.
\end{proof}

The underlying $(\cat \ppd^!)\op$-bimodule of $\big((\grad\uppd)\qd\big)\desusp$ is described by the following:

\begin{prop}
\label{prop:desusp_underlying_bimodule}
The underlying $(\cat \ppd^!)\op$-bimodule of $\big((\grad\uppd)\qd\big)\desusp$  is isomorphic to 
\[
(\cat \ppd^!)\op \ofb \big( (\cat \ppd^!)\op \otimes_\fb \sgn^\fb \big),
\]
with cohomological degree induced by placing $\sgn^\fb_c$ in degree $-c$.

Here, the left $(\cat \ppd^!)\op$-structure is given by using the $\ofb$-product of the canonical structures of $(\cat \ppd^!)\op$ and $(\cat \ppd^!)\op \otimes_\fb \sgn^\fb$. The right $(\cat \ppd^!)\op$-module structure is given by that of $ (\cat \ppd^!)\op$.
\end{prop}

\begin{proof}
Forgetting the cohomological grading, $\kk \finj \desusp$ is isomorphic as a $\kk \fb$-bimodule to $\bmu \ofb \sgn^\fb$, by Lemma \ref{lem:underlying_kkfb-bimodules}. The identification of the underlying $\kk\fb$-bimodule then follows from the opposite of Lemma \ref{lem:otimes_cat_opd} (i.e., taking into account that we are working with the opposite of $\cat \ppd^!$).

The identification of the $(\cat \ppd^!)\op$-bimodule structure uses arguments similar to those of Section  \ref{subsect:rgrad_cat_uppd_qd}.  
\end{proof}

Since $\big((\grad\uppd)\qd\big)\desusp$ is concentrated in non-positive cohomological degrees, it is natural to regrade using {\em homological} degree. One then has the following counterpart of Theorem \ref{thm:cohom_rgrad_qd}:

\begin{thm}
\label{thm:homology_desusp_grad_uppd_qd}
The homology $H_* \Big( \big((\grad\uppd)\qd\big)\desusp \Big)$ is an $\nat$-graded $\kk$-linear category with respect to the homological degree. In particular:
\begin{enumerate}
\item 
$H_0 \Big( \big((\grad\uppd)\qd\big)\desusp \Big)$ is a $\kk$-linear quotient category of $(\cat \ppd^!)\op$; 
\item 
for $n \in \nat$, $H_n \Big( \big((\grad\uppd)\qd\big)\desusp \Big)$  is a   $H_0 \Big( \big((\grad\uppd)\qd\big)\desusp \Big)$-bimodule.
\end{enumerate}

Hence, the desuspended homology functor induces:
$$
H_* ((-)\desusp) 
\cn
(\grad\uppd)\qd \dash\modules 
\longrightarrow 
H_* \Big( \big((\grad\uppd)\qd\big)\desusp \Big)\dash\modules^{\mathrm{gr}},
$$
in which the codomain denotes graded modules. In particular, for each $n \in \zed$, $H_n((-)\desusp)$ takes values in $H_0 \Big( \big((\grad\uppd)\qd\big)\desusp \Big)$-modules.
\end{thm}

\section{The composite adjunction for $\ppd$ binary quadratic}
\label{sect:composite}

The purpose of this section is to apply the general theory of Section \ref{sect:biquad} in the framework of Section \ref{subsect:uppd_framework}, supposing in addition that $\ppd$ is a binary quadratic operad  satisfying Hypothesis \ref{hyp:ppd_bin_quad}. This puts together the two relative duality theories provided by Sections \ref{sect:prop} and \ref{sect:prop_left}.

In particular, we have the commutative diagram 
\begin{eqnarray}
\label{eqn:diag:cat_uppd}
\xymatrix{
\kk \fb
\ar@{^(->}[r]
\ar@{^(->}[d]
&
\kk \finj
\ar@{^(->}[d]
\\
\cat \ppd
\ar@{^(->}[r]
&
\cat \uppd 
}
\end{eqnarray}
and the  isomorphism of $\kk \finj \boxtimes \cat \ppd\op$-modules 
$ 
\kk \finj \otimes_\fb \cat \ppd \stackrel{\cong}{\rightarrow} \cat \uppd.
$ 

\subsection{Recap of the identifications}
So as to facilitate comparison with  Section \ref{sect:biquad}, we give the following dictionary:
$\kring = \kk \fb$, 
$R = \kk \finj$, 
$C = \cat \ppd$, 
 and 
$A = \cat \uppd$.

 Moreover: 
\begin{enumerate}
\item 
Quadratic duality relative to $\cat \ppd$ (see Section \ref{sect:prop}) yields the DG $\kk$-linear category $(\rgrad \cat \uppd)\qd$, which has underlying $\kk \fb$-bimodule 
\[
(\rgrad \cat \uppd)\qd
\cong
\cat \ppd\otimes_\fb \kk \finj\qd .
\]
\noindent 
$(\rgrad \cat \uppd)\qd$ is non-positive and locally bounded (see Proposition \ref{prop:rgrad_finiteness}).

\bigskip
\item 
The associated Koszul dualizing complex $\vk := \vk (\cat \uppd) $ has underlying bimodule $ \cat \uppd \otimes_{\cat \ppd} (\rgrad \cat \uppd)\qd$, which is isomorphic to both the following
\[
\cat \uppd \otimes_\fb \kk \finj\qd 
\cong 
\kk \finj \otimes_\fb (\rgrad \cat \uppd)\qd.
\]

\bigskip
\item 
Quadratic duality relative to $\kk \finj$   (see Section \ref{sect:prop_left}) yields the DG $\kk$-linear category $(\grad \cat \uppd)\qd$, which has underlying $\kk \fb$-bimodule 
\[
(\grad \cat \uppd)\qd
\cong
(\cat \ppd)\qd \otimes_\fb \kk \finj .
\]
$(\grad \cat \uppd)\qd$ is non-negative and locally bounded (see Proposition \ref{prop:grad_cat_uppd_qd_finiteness}).

\bigskip
 \item 
The associated Koszul dualizing complex $K^\vee := K^\vee (\cat \uppd)$ has underlying bimodule $(\grad \cat \uppd)\qd \otimes_{\finj} \cat \uppd$, which is isomorphic to both the following
\[
(\cat \ppd)\qd \otimes_\fb \cat \uppd 
\cong 
(\grad \cat \uppd)\qd \otimes_\fb \cat \ppd.
\]

\bigskip
\item 
The dualizing complex for the composite adjunctions (see Section \ref{sect:biquad}) is $K^\vee \otimes_{\cat \uppd} \vk$. This is a complex of left $(\grad \cat \uppd)\qd$, right $(\rgrad \cat \uppd)\qd$ bimodules. The underlying (non DG) bimodule identifies as 
\[
(\grad \cat \uppd)\qd
\otimes_\fb 
(\rgrad \cat \uppd)\qd.
\]
This is isomorphic to both the following
\[
(\cat \ppd)\qd  \otimes_\fb \cat \uppd \otimes_\fb \kk \finj\qd
\cong 
(\cat \ppd)\qd  \otimes_\fb \kk \finj \otimes_\fb \cat \ppd \otimes_\fb \kk \finj\qd,
\]	
where the differential is twisted using both $e'$ and $\e'$.
\end{enumerate}

\subsection{The composite adjunction for left DG modules}
\label{subsect:composite_left}

Consider the composite adjunction for left DG modules:
\[
\xymatrix{
(\rgrad \cat \uppd)\qd \dash \dgmod
\ar@<.5ex>[rr]^(.55){\vk \otimes -} 
&&
\cat \uppd \dash \dgmod
\ar@<.5ex>[ll]^(.45){\ihom (\vk, -)}
\ar@<.5ex>[rr]^{K^\vee \otimes -}
& &
(\grad \cat \uppd)\qd \dash \dgmod
\ar@<.5ex>[ll]^{\ihom (K^\vee, -) },
}
\]
(in which the notation for the functors has been abbreviated for legibility).

This composite adjunction is induced by the complex $K^\vee \otimes_{\cat \uppd} \vk$:
\begin{eqnarray}
\label{eqn:composite_adjunction}
\xymatrix{
(\rgrad \cat \uppd)\qd \dash \dgmod
\ar@<.5ex>[rrrr]^{K^\vee \otimes_{\cat \uppd} \vk \otimes_{(\rgrad \cat \uppd)\qd}-}
&\quad \quad &\quad \quad &&
(\grad \cat \uppd)\qd \dash \dgmod.
\ar@<.5ex>[llll]^{\ihom_{(\grad \cat \uppd)\qd} (K^\vee \otimes _{\cat \uppd} \vk, -)}
}
\end{eqnarray}
\noindent 
Forgetting the differentials, we have the following identifications.
\begin{enumerate}
\item 
The functor  underlying  $K^\vee \otimes_{\cat \uppd} \vk \otimes_{(\rgrad \cat \uppd)\qd}-$ identifies as 
\[
(\cat \ppd)\qd \otimes _\fb \kk \finj \otimes_\fb - .
\]
Here $(\cat \ppd)\qd \otimes _\fb \kk \finj$ is the object underlying $(\grad \cat \uppd)\qd$; it is non-negative and locally bounded. 
\item 
The functor underlying  $\ihom_{(\grad \cat \uppd)\qd} (K^\vee \otimes _{\cat \uppd} \vk, -)$ identifies as 
\[
\ihom_{\kk \fb} (\cat \ppd \otimes_\fb \kk \finj\qd , -) .
\]
Here $\cat \ppd \otimes_\fb \kk \finj\qd$ is the object underlying $(\rgrad \cat \uppd)\qd$;  it is non-positive and locally bounded.
\end{enumerate}

\bigskip
Using the finiteness properties stated above (or combining Proposition \ref{prop:adj_vk_cat_uppd_lb} with 
the first statement of Theorem \ref{thm:relative_lbd_equivalence_cat_ppd}), one has:

\begin{prop}
\label{prop:comp_adj_lbd}
The adjunction (\ref{eqn:composite_adjunction}) restricts to the adjunction between the full subcategories of locally bounded objects:
\[
\xymatrix{
(\rgrad \cat \uppd)\qd \dash \lbdgmod
\ar@<.5ex>[rrrr]^{K^\vee \otimes_{\cat \uppd} \vk \otimes_{(\rgrad \cat \uppd)\qd}-}
&\quad \quad &\quad \quad &&
(\grad \cat \uppd)\qd \dash \lbdgmod.
\ar@<.5ex>[llll]^{\ihom_{(\grad \cat \uppd)\qd} (K^\vee \otimes _{\cat \uppd} \vk, -)}
}
\]
\end{prop}

\begin{exam}
\label{exam:koszul_type_adjunction}
\ 
\begin{enumerate}
\item 
Consider $\bmu$ as a $(\rgrad \uppd)\qd$-module via the augmentation of $(\rgrad \uppd)\qd$. The image of $\bmu$ under the left adjoint is $(\grad \uppd)\qd$, considered as a left DG module over itself. 
\item 
Likewise, considering $\bmu$ as a $(\grad \uppd)\qd$-module, its image under the right adjoint is $\big((\rgrad \uppd)\qd\big)^\sharp$, with its left $(\rgrad \uppd)\qd$-module structure. 
\end{enumerate}
This exhibits the expected behaviour of a Koszul-type adjunction (cf. the discussion in the Prologue to \cite{MR4398644}). Namely, the object $\bmu$ can be viewed as encoding the `simple' objects of the DG category of modules (up to cohomological shift); the left adjoint sends this to a free DG module. Correspondingly, the right adjoint sends $\bmu$ to a cofree object. 
\end{exam}

\subsection{The case $\ppd$ Koszul}
\label{subsect:composite_case_ppd}

When $\ppd$ is Koszul, this leads to  `Koszul duality' results. In particular, we have the following:

\begin{thm}
\label{thm:composite_holbd_equivalence}
Suppose that $\kk$ is a field of characteristic zero and that $\ppd$ is Koszul. Then the adjunction of Proposition \ref{prop:comp_adj_lbd} induces an equivalence between the respective homotopy categories
\[
\holbd ((\rgrad \cat \uppd)\qd)
\stackrel{\simeq}{\rightleftarrows}
\holbd ((\grad \cat \uppd)\qd).
\]
\end{thm}

\begin{proof}
This follows by combining the equivalences of Corollary \ref{cor:relative_BGG} and Theorem \ref{thm:relative_lbd_equivalence_cat_ppd}.
\end{proof}

\begin{exam}
Considering $\bmu$ as a $(\rgrad \cat \uppd)\qd$-module,  the adjunction unit 
\[
\bmu 
\rightarrow 
\ihom_{\kk \fb} ^{e', \e'} ((\rgrad \cat \uppd)\qd, (\grad \cat \uppd)\qd  )
\]
is a weak equivalence. This gives an explicit homological relationship between the DG categories $(\rgrad \cat \uppd)\qd$ and $(\grad \cat \uppd)\qd$.
\end{exam}

\begin{rem}
\ 
\begin{enumerate}
\item 
Heuristically one can think of Theorem \ref{thm:composite_holbd_equivalence} as stating that, if $\ppd$ is Koszul, then  $(\grad \cat \uppd)\qd $ and $(\rgrad \cat \uppd)\qd$ are `Koszul dual' DG categories.
\item
To relate this to `classical' operadic Koszul duality for binary quadratic operads, one should replace $(\grad \cat \uppd)\qd $ by its desuspension  $\big(( \grad \cat \uppd)\qd\big)\desusp$ and adjust the assertion accordingly.
\end{enumerate}
 \end{rem}

\section{The case $\uppd =  \ucom$}
\label{sect:ucom}

In this section, we apply the results of Section \ref{sect:composite} in the case $\uppd = \ucom$, so that $\ppd = \com$, working over a field $\kk$ of characteristic zero. The operad $\com$ is the archetypal example of a Koszul binary quadratic operad; its Koszul dual $\com^!$ is the operad $\lie$ that encodes Lie algebras.

\begin{rem*}
This case motivated the work of this paper. Namely, the author sought to relate work on the $\lie$-side (see  
\cite{MR4696223}, \cite{MR4835394}, and \cite{2023arXiv230907607P}) to work on the $\com$-side (see \cite{MR4518761} and \cite{P_filterFS}), which is related to studying representations of $\kk \fin$ (see \cite{P_finmod}).
\end{rem*}

Recall from Section \ref{subsect:PROP_recollections} that
 the $\kk$-linear category underlying $\cat \com$ is equivalent to $\kk \fs$, where $\fs$ is the category of finite sets and surjective maps, and that underlying $\cat \ucom$ is equivalent to $\kk \fin$, where $\fin$ is the category of finite sets and all maps. 
The diagram (\ref{eqn:diag:cat_uppd}) identifies as 
\[
\xymatrix{
\kk \fb
\ar@{^(->}[r]
\ar@{^(->}[d]
&
\kk \finj
\ar@{^(->}[d]
\\
\kk \fs
\ar@{^(->}[r]
&
\kk \fin 
}
\]
induced by the inclusions of wide subcategories of $\fin$.  Moreover, the composition in $\kk \fin$ induces the isomorphism of $\kk \finj \boxtimes \kk \fs\op$-modules 
$
\kk \finj \otimes_\fb \kk \fs \stackrel{\cong}{\rightarrow} \kk \fin.
$

\begin{rem}
\label{rem:left/right_augmentations_kk_fin}
\ 
\begin{enumerate}
\item
The general framework yields the left augmentation 
$
\kk \fin \twoheadrightarrow \kk \finj
$.  
The underlying map (on morphisms) is the retract (of the inclusion $\kk \finj \subset \kk \fin$) that sends non-injective maps to zero. For $g$ a map in $\fin$, the image of $[g]$ (considered as a morphism of $\kk \fin$)  in $\kk \finj$ is written $[g]_\finj$, which is $[g]$ (considered as a morphism of $\kk \finj$) if $g$ is injective and zero otherwise.
\item 
Similarly, one has the right augmentation
$ 
\kk \fin \twoheadrightarrow \kk \fs$. 
The underlying map (on morphisms) is the retract of the canonical inclusion that  sends non-surjective maps to zero. As above, for $g$ a map in $\fin$, the image of $[g]$ (considered as a morphism of $\kk \fin$)  in $\kk \fs$ is written $[g]_\fs$, which is $[g]$ (considered as a morphism of $\kk \fs$) if $g$ is surjective and zero otherwise.
\end{enumerate}
\end{rem}

\begin{rem}
\label{rem:left_right_module_structures_kk_fin}
\ 
\begin{enumerate}
\item 
The associated left $\kk \fs$-module structure on $\kk \finj$ has structure map given by the composite 
\[
\kk \fs \otimes_\fb \kk \finj \rightarrow \kk \fin \rightarrow \kk \finj,
\]
where the first map is given by composition in $\kk \fin$ and the second by the left augmentation. Explicitly, for composable maps $i$ in $\finj$ and $p$ in $\fs$, the image of $[p]\otimes [i]$ is $[p \circ i]_\finj$.
\item
Similarly, the associated right $\kk \finj$-module structure on $\kk \fs$ has structure map
\[
\kk \fs \otimes_\fb \kk \finj \rightarrow \kk \fin \rightarrow \kk \fs,
\]
where the first map is given by composition in $\kk \fin$ and the second by the right augmentation. For composable maps as above, the image of $[p]\otimes [i]$ is $[p \circ i]_\fs$.
\end{enumerate}
\end{rem}

\subsection{Identifying  $(\rgrad \kk \fin)\qd$ and $(\grad \kk \fin)\qd$}

The first step  is to identify the DG $\kk$-linear categories $(\rgrad \kk \fin)\qd$ and $(\grad \kk \fin)\qd$.
 Theorem \ref{thm:DG_qdual_grad_cat_uppd} gives: 

\begin{prop}
\label{prop:rgrad_kk_fin_qd}
The DG $\kk$-linear category $(\rgrad \kk \fin)\qd$ has underlying $\kk \fb$-bimodule 
$
\kk \fs \otimes_\fb \kk \finj\qd, 
$ 
with cohomological grading inherited from $\kk \finj\qd$. 

The differential is the Koszul complex differential associated to the right $\kk \finj$-module structure of $\kk \fs$. 
\end{prop}

\begin{rem}
\ 
\begin{enumerate}
\item 
The underlying complex of $\kk \fb\op$-modules was exploited in \cite{P_filterFS}, together with the fact that the cohomology $H^0$ yields a sub $\kk$-linear category of $\kk \fs$ (compare the general result, Theorem \ref{thm:cohom_rgrad_qd}). The full DG $\kk$-linear category structure was not used.
\item 
The cohomology of this complex was considered in \cite{MR4518761}.
\end{enumerate}
\end{rem}

Analogously, Theorem \ref{thm:grad_cat_uppd_dual} gives:

\begin{prop}
\label{prop:grad_kk_fin_qd}
The DG $\kk$-linear category $(\grad \kk \fin)\qd$ has underlying $\kk \fb$-bimodule 
$
(\kk \fs)\qd \otimes_\fb \kk \finj, 
$ 
with cohomological grading inherited from $(\kk \fs)\qd$. 

The differential is the Koszul complex differential associated to the left $\kk \fs$-module structure of $\kk \finj$. 
In particular, it is determined by its restriction to $\kk \finj$. The latter is given by the map 
\begin{eqnarray}
\label{eqn:diff_grad_kk_fin_!}
\kk \finj \rightarrow ((\kk \fs)\qd)\dg{1} \otimes_\fb \kk \finj
\end{eqnarray}
adjoint to the $\kk\fs$-action (where $((\kk \fs)\qd)\dg{1} \cong (\kk \fs \wt{1})^\sharp$).
\end{prop}

The differential appearing in Proposition \ref{prop:grad_kk_fin_qd} is described explicitly by Proposition \ref{prop:identify_diff_grad_cat_uppd_qd}, applied with $\uppd = \ucom$. 
 This is spelled out below.

Recall from Lemma \ref{lem:fsbase_wt1} that $\kk \fs (\mathbf{n+1}, \mathbf{n})$ is freely generated as a $\kk \sym_n$-module by $\{ [f_{i,j}] \mid 1 \leq i < j \leq n+1 \}$, where $f_{i,j}\in \fsbase (\mathbf{n+1}, \mathbf{n}) $ is described explicitly in {\em loc. cit.}. For $f \in \fs (\mathbf{n+1}, \mathbf{n})$, write $[f]^\sharp$ for the corresponding element of the dual basis of $\kk \fs (\mathbf{n+1}, \mathbf{n})^\sharp $. Thus, 
 $\kk \fs (\mathbf{n+1}, \mathbf{n})^\sharp$ is freely generated as a $\kk \sym_n\op$-module by $\{ [f_{i,j}] ^\sharp \ | \ 1 \leq i < j \leq n+1 \}$. Unravelling the definitions gives:

\begin{prop}
\label{prop:qhat_fs}
The component  $\kk \finj \rightarrow (\kk \fs\wt{1})^\sharp \otimes_\fb \kk \finj$ of (\ref{eqn:diff_grad_kk_fin_!})
is given
for $\alpha \in \finj (\mathbf{a}, \mathbf{n+1})$
 by 
\[
[\alpha] 
\mapsto
\sum_{1 \leq i < j \leq n+1} 
[f_{i,j}]^\sharp 
\otimes 
[f_{i,j} \circ \alpha]_{\finj}.
\]
\end{prop}

Since the differential is a derivation and is  equivariant,  it suffices to specify its action on the generators $\kk \finj^{[1]}$ of $\kk \finj$. Now, for any $n \in \nat$, $\kk \finj (\mathbf{n}, \mathbf{n+1})$ is generated as a $\kk \sym_{n+1}$-module by $[\iota_{n,n+1}]$ where $\iota_{n,n+1}$ is the  canonical inclusion $ \mathbf{n} \subset \mathbf{n+1}$, so it suffices to specify the differential on this.

\begin{cor}
\label{cor:image_qhat_canon_incl}
The differential restricted to $\kk \finj (\mathbf{n}, \mathbf{n+1})$, considered as a map 
\[
\kk \finj (\mathbf{n}, \mathbf{n+1})
\rightarrow 
\kk \fs (\mathbf{n+1}, \n) ^\sharp
\]
acts by 
$ 
[\iota_{n, n+1}] \mapsto \sum_{1 \leq i \leq n} [f_{i, n+1} ] ^\sharp . 
$ 
\end{cor}

\begin{proof}
This follows directly from Proposition \ref{prop:qhat_fs}, since $f_{i,j }\circ \iota_{n,n+1}$ can only be injective if $j =n+1$, in which case the composite $f_{i,n+1} \circ \iota_{n,n+1}$ is $\id_\mathbf{n}$.
\end{proof}

\begin{rem}
There is a key phenomenon at play for the case $\ppd = \com$. Namely, the $\kk$-linear map $\com (\mathbf{2}) \rightarrow \com (\mathbf{1}) \cong \kk$ induced by composition (in $\ucom$) with $\iota_{1,2} \in \finj (\mathbf{1}, \mathbf{2})$ is an {\em isomorphism} of $\kk$-vector spaces.

This accounts for the fact that the differential of $(\grad \kk \fin)\qd$ is related to the differential in the Chevalley-Eilenberg complex considered below; the latter can be constructed as a form of operadic Koszul complex.
\end{rem}

\subsection{The desuspension $\big((\grad \kk \fin)\qd\big)\desusp$}

As in Section \ref{subsect:desuspend_grad_uppd_qd}, it is natural to desuspend using the sheering functor $(-)\desusp$ of Appendix \ref{sect:sheer}. In particular, this gives a clearer relationship with $(\cat \lie)\op$; we have:

\begin{lem}
\label{lem:grad_kk_fin_qd_desusp}
The DG $\kk$-linear category $\big((\grad \kk \fin)\qd\big)\desusp$ has underlying object isomorphic to 
\[
(\cat \lie)\op \otimes_\fb \kk \finj \desusp.
\]
Hence the opposite category has underlying object isomorphic to 
$
(\kk \finj^\ddag)\op \otimes _\fb \cat \lie,
$
where the grading is inherited from that of $(\kk \finj^\ddag)\op$. 
\end{lem}

The remainder of this section is devoted to comparing this to a certain universal Chevalley-Eilenberg complex. First let us recall the classical (homological) Chevalley-Eilenberg complex for $\g$ a Lie algebra and $M$ a right $\g$-module. The associated Chevalley-Eilenberg complex has the form 
\[
M \otimes \Lambda^* \g, 
\]
where the homological grading is given by the length grading of the exterior algebra. The differential on $M \otimes \Lambda^n \g$ is given for $m \in M$ and $x_i \in \g$ by
\[
d (m \otimes x_1 \wedge \ldots \wedge x_n )
= 
\sum_{j=1}^n 
(-1)^j 
\Big(
[m, x_j] \otimes x_1 \wedge \ldots \wedge\widehat{x_j} \wedge \ldots \wedge x_n
+ 
\sum_{i<j}
m \otimes x_1 \wedge \ldots \wedge [x_i,x_j] \wedge \ldots \wedge\widehat{x_j} \wedge \ldots \wedge x_n
\Big)
\]
 where $[x_i, x_j]$ appears in the $i$th place.

In particular, this applies taking $M = \g ^{\otimes a}$, the $a$-fold tensor product of the (right) adjoint representation of $\g$, so that the underlying graded object is $\g^{\otimes a} \otimes \Lambda^* \g$. In particular, in homological degree $n$, we can view this term as the quotient 
\[
\g ^{\otimes a+n} \twoheadrightarrow \g^{\otimes a} \otimes \Lambda^n \g
\]
obtained by applying the functor $- \otimes_{\sym_n} \sgn_n$, using the restricted place permutation action on the last $n$ factors. 

The differential of the Chevalley-Eilenberg complex $\g^{\otimes a} \otimes \Lambda^n \g \rightarrow \g^{\otimes a} \otimes \Lambda^{n-1} \g$ lifts to a linear map 
$\g^{\otimes a+n} \rightarrow \g ^{\otimes a +n-1}$ that is given by 
\begin{eqnarray}
\label{eqn:tilde_d}
\tilde{d} ( x_1 \otimes \ldots \otimes x_{a+n} )
= 
\sum_{j=a+1}^{a+n} 
(-1)^{j-a} 
\sum_{i<j}
 x_1 \otimes \ldots \otimes [x_i,x_j] \otimes \ldots \otimes \widehat{x_j} \otimes \ldots \otimes x_{a+n}.
\end{eqnarray}

Moreover, we can rewrite 
\[
g^{\otimes a} \otimes \Lambda^n \g
\cong 
(\kk \sym_a \obf \sgn_n^{\fb\op} )
\otimes_{\sym_{a+n}}
\g ^{\otimes a+n}
\]
and this is an isomorphism of $\sym_a$-modules. Now, by Lemma \ref{lem:underlying_kkfb-bimodules}, there is an isomorphism of bimodules $\kk \sym_a \obf \sgn_n^{\fb\op} \cong (\kk \finj ^\ddag) \op (\mathbf{a+n}, \mathbf{a})$. It follows that, globally, we can rewrite the Chevalley-Eilenberg complex $\g^{\otimes \bullet} \otimes \Lambda^* \g$ as 
\[
(\kk \finj ^\ddag) \op
\otimes_\fb 
\g^{\otimes \ast}.
\]
The differential is induced by (\ref{eqn:tilde_d}).

To proceed, we replace the symmetric monoidal category of $\kk$-vector spaces with respect to the usual tensor product by the category of right $\lie$-modules, equipped with the Day convolution product. The operad $\lie$ is (essentially tautologically) a Lie algebra in this category. Taking $\g= \lie$ and applying the above Chevalley-Eilenberg complex construction yields 
\[
(\kk \finj ^\ddag) \op
\otimes_\fb 
\cat \lie,
\]
using the fact that $\lie ^{\odot \ast}$ has underlying object $\cat \lie$. From this viewpoint, a right $\lie$-module is equivalent to a right $\cat \lie$-module.

As a first step towards understanding the differential, we note the following:

\begin{lem}
\label{lem:CE_complex}
\ 
\begin{enumerate}
\item 
The Chevalley-Eilenberg differential on $(\kk \finj ^\ddag) \op
\otimes_\fb 
\cat \lie$ is a morphism of right $\cat \lie$-modules.
\item 
In particular, the differential is determined by its restriction 
\begin{eqnarray}
\label{eqn:restrict_CE_diff}
(\kk \finj ^\ddag) \op
\rightarrow 
(\kk \finj ^\ddag) \op \otimes_{\fb} \cat  \wt{1} \lie .
\end{eqnarray}
\end{enumerate}
\end{lem}

\begin{proof}
The first statement is immediate, since we are working within the category of right $\cat \lie$-modules. 

The second statement follows, once one verifies that the differential restricted to $(\kk \finj ^\ddag) \op$ takes values in $(\kk \finj ^\ddag) \op \otimes_{\fb} \cat  \wt{1} \lie$ (i.e., using the weight one part of $\cat \lie$). This can be read off from the form of (\ref{eqn:tilde_d}).
\end{proof}

\begin{exam}
\label{exam:basic_CE_diff}
Evaluating (\ref{eqn:restrict_CE_diff}) on $(\mathbf{a+1}, \mathbf{a})$ gives 
\[
\kk \finj \op (\mathbf{a+1}, \mathbf{a}) 
\rightarrow 
\cat \lie (\mathbf{a+1}, \mathbf{a}) 
\]
where we have dropped the $(-)^\ddag$ since $\sgn_1= \triv_1$. To understand this, it suffices to consider the image of $[\iota_{a,a+1}\op]$. 

By inspection, up to the sign $(-1)$ and taking into account the twisting provided by $(-)\desusp$, the image identifies with that of Corollary \ref{cor:image_qhat_canon_incl}, using the identification of $\cat \wt{1} \lie$ and $\cat \wt{1} \com$ up to twisting (recall that $\cat \com$ is isomorphic, as a $\kk$-linear category, to $\kk \fs$).
\end{exam}

These ingredients allow us to state the main result of this subsection:

\begin{thm}
\label{thm:identify_with_CE}
The opposite of the DG $\kk$-linear category $\big((\grad \kk \fin)\qd\big)\desusp$ is isomorphic as a complex of right $\cat \lie$-modules to the Chevalley-Eilenberg complex of $\lie$ with coefficients in $\cat \lie$, which has underlying object 
\[
(\kk \finj ^\ddag) \op
\otimes_\fb 
\cat \lie.
\]
\end{thm}

\begin{proof}
The above discussion together with Lemma \ref{lem:CE_complex} has already shown that the Chevalley-Eilenberg complex for $\lie$ with coefficients in $\cat \lie$ has the form in the statement and is a complex of right $\cat \lie$-modules. It  only remains to show that the differentials correspond. 

The final ingredient  is the fact that the Chevalley-Eilenberg differential behaves as a derivation with respect to the $(\kk \finj ^\ddag) \op$ structure. This follows by analysing the expression (\ref{eqn:tilde_d}); note that the signs $(-1)^j$ that arise can be interpreted as  Koszul signs taking into account the grading of $\kk\finj^\ddag$ and the fact that the differential has homological degree $-1$. (This property corresponds to the fact that the differential of the DG $\kk$-linear category is a derivation.) 

Using this, the proof reduces to the case treated in Example \ref{exam:basic_CE_diff}, which serves to establish the relationship with the differential of $\big((\grad \kk \fin)\qd\big)\desusp$ via Corollary \ref{cor:image_qhat_canon_incl}.
\end{proof}

\begin{rem}
This establishes the relationship with the work \cite{MR4696223} and \cite{2023arXiv230907607P}, where the Lie algebra homology of $\lie$ with coefficients in $\cat \lie$ is studied.
\end{rem}

\subsection{The Koszul duality theory}

The general theory of Section \ref{sect:composite} yields the adjunction:
\begin{eqnarray}
\label{eqn:adj_kk_fin}
\xymatrix{
(\rgrad \kk \fin)\qd \dash \dgmod
\ar@<.5ex>[rrr]^{K^\vee \otimes_{\fin} \vk \otimes_{(\rgrad \kk \fin)\qd}-}
&\quad \quad &\quad \quad &
(\grad  \kk \fin )\qd \dash \dgmod.
\ar@<.5ex>[lll]^{\ihom_{(\grad \kk\fin)\qd} (K^\vee \otimes _{\fin} \vk, -)}
}
\end{eqnarray}

Since $\com$ is Koszul,  the results of Section \ref{subsect:composite_case_ppd} apply. In particular:

\begin{thm}
\label{thm:relative_Koszul_duality_kk_fin}
Suppose that $\kk$ is a field of characteristic zero. Then the adjunction (\ref{eqn:adj_kk_fin}) restricts to an adjunction between the full subcategories of locally bounded modules and this induces an equivalence between the associated homotopy categories:
\[
\holbd ( (\rgrad \kk \fin)\qd) 
\stackrel{\simeq}{\rightleftarrows} 
\holbd ((\grad  \kk \fin )\qd).
\]

In particular, the adjunction unit 
\[
\bmu 
\rightarrow 
\ihom_{\kk \fb} ^{e', \e'} ((\rgrad \kk\fin)\qd, (\grad \kk \fin)\qd  )
\]
is a weak equivalence.
\end{thm}

\begin{rem}
Heuristically, this shows that $(\rgrad \kk\fin)\qd$ and $(\grad \kk \fin)\qd$ are `Koszul dual'. Here, $(\grad \kk \fin)\qd$ is the $\finj\op$-Koszul complex for $\kk \fs$ given in Proposition \ref{prop:rgrad_kk_fin_qd}; up to suspension, $(\rgrad \kk\fin)\qd$ is the Chevalley-Eilenberg complex with coefficients in $\cat \lie$ by Theorem \ref{thm:identify_with_CE}. This duality  extends that between $\cat \lie$ (up to suspension and $(-)\op$) and $\cat \com \cong \kk \fs$ exploited in \cite{MR4835394}.
\end{rem}

\begin{exam}
Consider a $\kk \fin$-module $M$ (taken to be a complex concentrated in cohomological degree $0$). Then, using the adjunctions, we have the associated objects of $(\rgrad \kk \fin)\qd \dash \dgmod$ and of $(\grad  \kk \fin )\qd \dash \dgmod$.
\begin{enumerate}
\item 
The complex in $(\rgrad \kk \fin)\qd \dash \dgmod$ has underlying complex in $\kk \fb$-modules
\[
\ihom_{\kk \fb} (\kk \finj \qd, M),
\]
equipped with the twisted differential. This identifies with the usual Koszul complex for the underlying $\kk \finj$-module of $M$. 

There is more structure, namely the action of $(\rgrad \kk \fin)\qd$. In particular, on taking cohomology, in each degree one obtains a $H^0 ((\rgrad \kk \fin)\qd)$-module.
\item 
The complex in $(\grad  \kk \fin )\qd \dash \dgmod$ has underlying complex in $\kk \fb$-modules
\[
\kk \fs\qd \otimes_\fb M,
\]
equipped with the twisted differential. (Note that $\kk \fs\qd$ is isomorphic to $((\cat \lie)\op)\susp$, the suspension of $(\cat \lie)\op$.) This is the Koszul complex associated to the underlying $\kk \fs$-module structure of $M$.

Again, there is more structure, corresponding to the action of $(\grad \kk \fin)\qd$ in this case. In particular, on applying the desuspended homology functor $H_* ((-)\desusp)$, in each degree one obtains a $H_0(\lie; \cat \lie)\op$-module.
\end{enumerate}
\end{exam}

\appendix 
\section{Sheering for bimodules}
\label{sect:sheer}

This  appendix  introduces  sheering operations for $\kk \fb$-bimodules, analogous to those for $\kk \fb$-modules introduced in \cite[Section 6.4]{MR3430359}. These constructions are intimately related to the suspension that is used in the theory of properads and, more generally, PROPs. 

Here, $\kk \fb$-bimodules are taken to be (cohomologically) graded and we work with the monoidal structure $(\kk (\fb\op \times \fb)\dash\modules, \otimes_\fb , \bmu)$, where $\otimes_\fb$ is formed using the tensor product of graded $\kk$-modules and $\bmu$ is placed in degree zero. The category of graded $\kk$-modules is equipped with the symmetric monoidal structure using Koszul signs.

\begin{nota}
\label{nota:susp}
The suspension functor $\Sigma^t$, for $t \in \zed$, is defined to be $\Sigma^t \kk \otimes - $, where $\Sigma^t \kk$ is $\kk$ placed in degree $t$. (Thus, for a graded $\kk$-module $V$, $(\Sigma^t V)^n= V^{n-t}$.) 
\end{nota}

\begin{defn}
\label{defn:ls_rs}
For $M$ a graded $\kk \fb$-module, define the bimodules $M \rs$ and $M\ls$ respectively by:
\begin{eqnarray*}
M\rs (\mathbf{a}, \mathbf{b}):= \Sigma^{b}\kk \otimes  M(\mathbf{a}, \mathbf{b}) \otimes \Sigma^{-a}\kk   \\
M\ls (\mathbf{a}, \mathbf{b}):= \Sigma^{-b}\kk \otimes M(\mathbf{a}, \mathbf{b}) \otimes \Sigma^a\kk.
\end{eqnarray*}
(Here, the symmetric groups act trivially upon the terms  $\Sigma^* \kk$.)
\end{defn}

\begin{rem}
\label{rem:Z_grading_sheering}
This is  related to the (auxiliary) $\zed$-grading introduced in Definition \ref{defn:Z-grade_fb-bimodules}. For example, for an integer $n$, one has natural isomorphisms:
\[
(M\dg{n})\rs \cong (M \rs)\dg{n} \cong \Sigma^n M \dg{n}.
\]
Hence $M\rs \cong \bigoplus_n \Sigma^n M\dg{n}$. (Note  that these isomorphisms  involve Koszul signs if $M$ is graded.)
\end{rem}

The following is clear (again noting that the isomorphisms involve Koszul signs):

\begin{lem}
\label{lem:op_ls_rs}
For $M$ a $\kk \fb$-bimodule, there are natural isomorphisms 
\begin{eqnarray*}
(M\rs)\op &\cong & (M\op) \ls
\\
(M\ls)\op &\cong & (M\op) \rs.
\end{eqnarray*}
\end{lem}

Moreover, sheering behaves well with respect to $\otimes_\fb$:

\begin{prop}
\label{prop:sheering_monoidal}
The functors $(-)\rs$ and $(-)\ls$ induce quasi-inverse monoidal self-equivalences of $(\kk (\fb\op \times \fb)\dash\modules, \otimes_\fb , \bmu)$.
\end{prop}

The sheering functors $(-)\rs$ and $(-)\ls$ can be composed with the involution $(-)^\ddag$. 
These functors commute in the following sense:

\begin{lem}
\label{lem:ddag_ls_rs}
For $M$ a $\kk \fb$-bimodule, there are natural isomorphism:
\begin{eqnarray*}
(M \rs) ^\ddag & \cong & (M ^\ddag)\rs
\\
(M \ls) ^\ddag & \cong & (M ^\ddag)\ls.
\end{eqnarray*}
\end{lem}

\begin{nota}
\label{nota:susp_desusp}
For concision, these composite functors will be written  as $(-)\susp$ and $(-)\desusp$ respectively.
\end{nota}

Combining Proposition \ref{prop:sheering_monoidal} with Proposition \ref{prop:basic_properties_ddag_op} (\ref{item:ddag_monoidal}), one has:

\begin{cor}
\label{cor:composites_monoidal}
The  functors $(-)\susp$ and  $(-)\desusp$ are mutually inverse monoidal equivalences. 
Hence: 
\begin{enumerate}
\item 
$(-)\susp$ and  $(-)\desusp$ induce mutually inverse self-equivalences of the category of unital monoids in $(\kk (\fb\op \times \fb)\dash\modules, \otimes_\fb , \bmu)$.
\item 
If $A$ is a unital monoid in $(\kk (\fb\op \times \fb)\dash\modules, \otimes_\fb , \bmu)$, then 
$(-)\susp$ and  $(-)\desusp$  induce mutually inverse equivalences between the categories of left (respectively right) $A$-modules and of left (resp. right) $A\susp$-modules.
\end{enumerate}
\end{cor}

This extends to the DG context. To spell this out, we first consider general derivations. 

Recall that, for $A$ a unital monoid in $\kk \fb$-bimodules, an $A$-bimodule is a $\kk \fb$-bimodule $M$ equipped with structure maps $A\otimes_\fb M \rightarrow M$ and $M \otimes_\fb A \rightarrow M$ satisfying the usual axioms. Then a degree $\ell$ derivation from $A$ to $M$ is a map $\phi : \Sigma^\ell A \rightarrow M$ such that the following diagram commutes 
\[
\xymatrix{
\Sigma^\ell (A \otimes_\fb A) 
\ar[r]
\ar[d]
&
A \otimes_\fb M \ \oplus \ M \otimes_\fb A 
\ar[d]
\\
\Sigma^\ell A 
\ar[r]_\phi
&
M,
}
\]
in which the left hand vertical map is induced by the product of $A$, the right hand one by the structure morphisms of $M$, the top map by $\phi$ applied respectively to the left and right hand factors.

\begin{lem}
\label{prop:derivations_rs_ls}
For $A$ and $M$ as above, with $\phi$ a degree $\ell$-derivation, 
\begin{enumerate}
\item 
$\phi\susp$ together with the isomorphism $(\Sigma^\ell A)\susp \cong \Sigma^\ell (A\susp)$ yields a degree $\ell$-derivation  from $A\susp$ to $M\susp$; 
\item 
 $\phi\desusp$ together with the isomorphism $(\Sigma^\ell A)\desusp \cong \Sigma^\ell (A\desusp)$ yields a degree $\ell$-derivation  from $A\desusp$ to $M\desusp$.
 \end{enumerate}
 \end{lem}

Proposition \ref{prop:derivations_rs_ls} applies in particular when $M$ is taken to be $A$, equipped with its canonical $A$-bimodule structure.

\begin{cor}
\label{cor:DG_algebra_kFB}
For $A$  a unital monoid in $(\kk (\fb\op \times \fb)\dash\modules, \otimes_\fb , \bmu)$ equipped with a degree $1$ derivation $d$ such that $d^2=0$, the morphism $d\susp$ induces a natural square zero degree $1$ derivation on $A\susp$ (similarly for $(-)\desusp$ in place of $(-)\susp$).
\end{cor}

\begin{proof}
The required derivation on $A\susp$ is given by the composite 
\[
\Sigma (A \susp) \cong (\Sigma A)\susp \stackrel{d\susp}{\rightarrow}  A \susp.
\]
It is a derivation by Proposition \ref{prop:derivations_rs_ls}. One checks directly that it squares to zero.
\end{proof}

\begin{rem}
As in Corollary \ref{cor:composites_monoidal}, $(-)\susp$ and $(-)\desusp$ induce mutually inverse self-equivalences of the category of DG unital monoids in $\kk \fb$-bimodules. Likewise for DG modules over such.
\end{rem}


\providecommand{\bysame}{\leavevmode\hbox to3em{\hrulefill}\thinspace}
\providecommand{\MR}{\relax\ifhmode\unskip\space\fi MR }
\providecommand{\MRhref}[2]{%
  \href{http://www.ams.org/mathscinet-getitem?mr=#1}{#2}
}
\providecommand{\href}[2]{#2}

\end{document}